%% file: main.tex
\documentclass[11pt,leqno]{article}

\input{cmds/packages}

\input{cmds/commands}

\title{Consistency of semi-supervised learning, stochastic tug-of-war games, and the $p$-Laplacian\thanks{{\bf Funding:} Calder was supported by NSF-DMS grant 1944925, the Alfred P. Sloan foundation, the McKnight foundation, and an Albert and Dorothy Marden Professorship. {\bf Source Code:} \url{https://github.com/jwcalder/p-Laplace-consistency}}}
\author{Jeff Calder}
\affil{School of Mathematics \\ University of Minnesota\thanks{{\bf Email:} \textit{jcalder@umn.edu}}}
\author{Nadejda Drenska}
\affil{Department of Mathematics \\ Louisiana State University\thanks{{\bf Email:} \textit{ndrenska@lsu.edu}}}
\begin{document} 

\maketitle

\begin{abstract}
In this paper we give a broad overview of the intersection of partial differential equations (PDEs) and graph-based semi-supervised learning. The overview is focused on a large body of recent work on PDE continuum limits of graph-based learning, which have been used to prove well-posedness of semi-supervised learning algorithms in the large data limit. We highlight some interesting research directions revolving around \emph{consistency} of graph-based semi-supervised learning, and present some new results on the consistency of $p$-Laplacian semi-supervised learning using the stochastic tug-of-war game interpretation of the $p$-Laplacian. We also present the results of some numerical experiments that illustrate our results and suggest directions for future work. 
\end{abstract}

\tableofcontents

\input{sections/intro}
\input{sections/tug}

\input{sections/ssl}

\input{sections/results}

\input{sections/numerics}

\input{sections/conclusion}

\bibliography{main}
\bibliographystyle{calder}

\end{document}

%% file: cmds/packages.tex
\usepackage[T1]{fontenc}
\usepackage{amsmath}
\usepackage{amsthm}
\usepackage{amsxtra}
\usepackage{amsfonts}
\usepackage{amssymb}
\usepackage[margin=1.25in]{geometry}
\usepackage{color,hyperref}
\usepackage{graphicx}
\usepackage{subfig}
\usepackage{cite}
\usepackage{authblk}
\hypersetup{colorlinks,breaklinks,
             linkcolor=blue,urlcolor=blue,
             anchorcolor=blue,citecolor=blue}
\usepackage{color}
\usepackage{array}
\usepackage{booktabs}
\usepackage{enumerate}
\usepackage{dsfont}
\usepackage{bbm}

%% file: cmds/commands.tex
\DeclareMathOperator*{\argmax}{argmax}
\DeclareMathOperator*{\argmin}{argmin}
\DeclareMathOperator{\Lip}{Lip}

\let\div\relax
\DeclareMathOperator{\div}{div}
\DeclareMathOperator{\dist}{dist}
\DeclareMathOperator{\diag}{diag}
\DeclareMathOperator{\Var}{Var}
\DeclareMathOperator{\diam}{diam}

\newcommand{\Lp}{\mathcal{L}_p}
\renewcommand{\L}{\mathcal{L}}
\newcommand{\Lrw}{\mathcal{L}_{rw}}
\renewcommand{\bar}[1]{{\overline{#1}}}

\newcommand{\one}{\mathds{1}}
\newcommand{\iid}{\emph{i.i.d.}}
\newcommand{\eps}{\varepsilon}
\renewcommand{\epsilon}{\varepsilon}
\renewcommand{\phi}{\varphi}

\newcommand{\F}{X}
\newcommand{\E}{{\mathbb E}}

\newcommand{\N}{\mathbb{N}}

\renewcommand{\P}{{\mathbb P}}

\newcommand{\R}{\mathbb{R}}

\newcommand{\X}{{\mathcal X}}
\newcommand{\Y}{{\mathcal Y}}

\newcommand{\perc}{\%}

\newcommand{\red}{}

\newcommand{\nc}{}

\def\Xint#1{\mathchoice
{\XXint\displaystyle\textstyle{#1}}%
{\XXint\textstyle\scriptstyle{#1}}%
{\XXint\scriptstyle\scriptscriptstyle{#1}}%
{\XXint\scriptscriptstyle\scriptscriptstyle{#1}}%
\!\int}
\def\XXint#1#2#3{{\setbox0=\hbox{$#1{#2#3}{\int}$ }
\vcenter{\hbox{$#2#3$ }}\kern-.6\wd0}}

\def\dashint{\Xint-}

\newtheorem{theorem}{Theorem}
\newtheorem{lemma}[theorem]{Lemma}
\newtheorem{corollary}[theorem]{Corollary}

\newtheorem{proposition}[theorem]{Proposition}
\theoremstyle{definition}
\newtheorem{remark}[theorem]{Remark}
\newtheorem{definition}[theorem]{Definition}

\numberwithin{equation}{section}
\numberwithin{theorem}{section}

\newcommand{\wij}{w_{xy}}
\newcommand{\gij}{g_{xy}}
\newcommand{\gji}{g_{yx}}
\newcommand{\epsi}{\eps_x}
\newcommand{\epsj}{\eps_y}
\newcommand{\di}{d_{x}}
\renewcommand{\dj}{d_{y}}
\newcommand{\as}[1]{\textbf{(A#1)}}
\renewcommand{\k}{k}

%% file: sections/intro.tex

\section{Introduction}

Machine learning refers to algorithms that learn how to  perform tasks, like image classification or text generation, from examples or experience, and are not explicitly programmed with step-by-step instructions in the way a human may be instructed to perform a similar task. The recent surge in machine learning and artificial intelligence is being driven by deep learning, which uses deep artificial neural networks and has found applications in nearly all areas of science, engineering and everyday life \cite{goodfellow2016deep}. Modern deep learning excels when provided with massive amounts of training data and computational resources. However, there are many applications, specifically with real-world problems, where labeled training data is hard to come by and number in the hundreds or thousands, instead of millions. For example, in medical image analysis, a human expert, i.e., a highly trained doctor, must annotate images in order to provide data for machine learning algorithms to train in. Obtaining labeled data is thus costly, and there is a tremendous interest in developing machine learning algorithms that perform well with as few labeled examples as possible. 

There are many frameworks for learning from limited data. One effective method is \emph{semi-supervised learning}, which makes use of both labeled and unlabeled data in the learning task. In contrast, the most common type of machine learning, called \emph{fully supervised} learning, makes use of only labeled training data. The labeled training data for a fully supervised classification task includes data/label pairs $(x_1,y_1),\dots,(x_n,y_n)$, where $x_i\in \R^d$ and $y_i \in \R^k$. The goal of fully supervised learning can generally be stated as finding, or ``learning'', a function $f:\R^d\to \R^k$ so that $f(x_i)\approx y_i$ for all $i$. This can be viewed as an \emph{ill-posed} problem, especially when the number of labeled training points $n$ is small, since there are many possible functions $f$ that fit the training data. Furthermore, the ultimate goal is not just to fit the training data, but to learn a function that \emph{generalizes} well to new data that has not been seen before. Semi-supervised learning uses unlabeled data to improve the performance of classification algorithms in the context of small training sets. In many applications, unlabeled data is abundant and easy to obtain. In medical image analysis, for example, unlabeled data would correspond to medical images of a similar type and modality that have not been labeled or annotated by an expert. 

An effective technique for exploiting unlabeled data in semi-supervised learning is to utilize a graph structure, which may be intrinsic to the data, or constructed based on similarities between data points. Graphs encode interdependencies between constituent data points that have proven useful for analyzing and representing high dimensional data. There has been a surge of interest recently in graph-based semi-supervised learning techniques for problems where very few labeled examples are available, which is a setting that is challenging for existing techniques based on Laplacian regularization and harmonic extension. Various methods have been proposed, including $p$-Laplacian regularization, higher order Laplacian methods, Poisson learning, and many others. Many of these algorithms are inspired by insights from the theory of partial differential equations (PDEs) or the calculus of variations, by examining the PDE-continuum limits of discrete graph-based learning algorithms, and their well-posedness properties, or lack thereof. 

While there has been a substantial amount of work on PDE-continuum limits of graph-based learning, there has been relatively little work on the question of \emph{consistency} of graph-based learning, using these well-developed PDE tools. The basic question of \emph{consistency} is whether the machine learning algorithm is making the correct predictions, under a simplified model for the data. The current PDE-continuum limit results simply describe how the algorithms behave in the large data limit, but have not yet, with few exceptions, been used to prove that they work properly --- that is, that they are consistent.  It is arguably the case that consistency is more important than well-posedness in the continuum limit, yet the question has rarely been studied, which we suspect is due to the difficulty in defining what consistency means, and the difficulty in obtaining meaningful results outside of toy settings.

In this paper, we provide a broad overview of semi-supervised learning and its connections to PDEs, and we present some new consistency results for $p$-Laplacian based semi-supervised learning. Our new results make use of the tug-of-war with noise interpretation of the $p$-Laplacian, for which we also provide a brief literature survey.  In particular, one of our results uses the tug-of-war game on a stochastic block model graph, which does not have the geometric structure that is usually required for PDE-based analysis. Our consistency results for the $p$-Laplacian are preliminary results meant to spark new work, and they certainly leave many questions unanswered. Our overall goal in this paper is to highlight a number of open research problems that will benefit from fruitful collaboration between PDE analysts and more theoretically minded machine learning researchers. It would be interesting in future work to improve these results and extend them to other graph-based semi-supervised learning algorithms, and other graph structures. 

\subsection{Outline}

This paper is organized as follows. In Section \ref{sec:tug} we provide a brief overview of the $p$-Laplacian and the stochastic tug-of-war game interpretation. In Section \ref{sec:ssl} we give a thorough survey of graph-based semi-supervised learning, and its connections to PDEs and tug-of-war games, with a particular emphasis on the $p$-Laplacian. Then in Section \ref{sec:results} we present some preliminary results on consistency properties of the $p$-Laplacian using the tug-of-war game interpretation. This includes results both on geometric graphs, and stochastic block models, as well as some numerical results to illustrate the main theorems. Finally we conclude and discuss directions for future work in Section \ref{sec:conc}.  

%% file: sections/tug.tex

\section{Tug-of-war games and the \texorpdfstring{$p$}{p}-Laplacian}
\label{sec:tug}

In this section, we provide a brief overview of the $p$-Laplacian and the connection to stochastic tug-of-war games. 

\subsection{The \texorpdfstring{$p$}{p}-Laplacian}

The $p$-Laplacian arises as the Euler-Lagrange equation, or necessary conditions, for the nonlinear potential problem in the calculus of variations
\begin{equation}\label{eq:min_lp}
\min_{u\in W^{1,p}(\Omega)} \int_{\Omega}|\nabla u|^p \, dx,
\end{equation}
where $p\geq 1$ and $\Omega\subset \R^d$, subject to some boundary conditions, such as a Dirichlet condition $u=g$ on $\partial\Omega$. The Euler-Lagrange equation \cite{EvansPDE} for \eqref{eq:min_lp} is
\begin{equation}\label{eq:plap_def}
\Delta_p u := \div\left( |\nabla u|^{p-2}\nabla u\right) = 0 \ \ \text{ in } \ \ \Omega,
\end{equation}
and we call $\Delta_p u$ the $p$-Laplacian of $u$. We must take $p\geq 1$ to ensure \eqref{eq:min_lp} is convex and admits a minimizer. The case of $p=2$ corresponds to the usual Laplacian $\Delta_p=\Delta$. When  $1 \leq p < 2$ the diffusion is singular when $\nabla u=0$, while for $p>2$ the diffusion becomes degenerate when $\nabla u=0$; both cases lead to drastically different properties and regularity theory compared to the uniformly elliptic case of $p=2$ \cite{evans1982new}. Functions that satisfy $\Delta_p u =0$ are called \emph{$p$ harmonic}.

If we expand the divergence in \eqref{eq:plap_def}, we find that any $p$-harmonic function $u$ satisfies (provided $\nabla u\neq 0$ when $p<2$)
\begin{equation}\label{eq:neargame}
\Delta_p u = |\nabla u|^{p-2}(\Delta u + (p-2)\Delta_\infty u )  = 0 \ \ \text{ in } \ \ \Omega,
\end{equation}
where $\Delta_\infty u$ is the $\infty$-Laplacian, defined by 
\begin{equation}\label{eq:inf_lap}
\Delta_\infty u = \frac{1}{|\nabla u|^2}\nabla u \cdot \nabla^2 u \nabla u = \frac{1}{|\nabla u|^2}\sum_{i,j=1}^d u_{x_ix_j}u_{x_i}u_{x_j},
\end{equation}
and $\nabla^2 u$ is the Hessian of $u$. The $\infty$-Laplacian is so named because it is the $p\to \infty$ limit of the $p$-Laplacian in the sense that
\[\Delta_\infty u = \lim_{p\to \infty} \frac{1}{p-2}|\nabla u|^{2-p}\Delta_p u,\]
provided again that $\nabla u \neq 0$, which follows directly from \eqref{eq:neargame}. It is possible to interpret the $\infty$-Laplacian as the Euler-Lagrange equation for a variational problem like \eqref{eq:min_lp} with $p=\infty$; we refer the reader to \cite{aronsson2004tour} for more details. For a more detailed overview of the $p$-Laplacian we also refer to \cite{lindqvist2019notes}.

\subsection{Tug-of-war games}

When $p=2$ there is a well-established classical connection between random walks, or Brownian motions, and harmonic functions. Indeed, any harmonic function $u$ satisfies the mean value property \cite{EvansPDE}
\begin{equation}\label{eq:mvp}
u(x_0) = \dashint_{B(x_0,\epsilon)}u(y) \, dy,
\end{equation}
where $\epsilon>0$ is any value for which $B(x_0,\epsilon)\subset \Omega$, and the notation $\dashint$ means
\[\dashint_V u \, dx = \frac{1}{|V|}\int_{V} u \, dx.\]
We can interpret the right hand side of \eqref{eq:mvp} as the expectation of $u(X)$, where $X$ is a random variable uniformly distributed on the ball $B(x_0,\eps)$. Thus, if we define a random walk $X_1,X_2,\dots,$ on $\Omega$, which is a sequence of random variables for which $X_0=x_0\in \Omega$ and, conditioned on $X_{k}$, $X_{k+1}$ is uniformly distributed on $B(X_{k},\epsilon)$, we have
\[\E\left[ u(X_{k+1}) \, | \, X_{k}\right] = \dashint_{B(X_{k},\epsilon)}u(y) \, dy = u(X_{k}),\] 
provided $B(X_k,\eps)\subset \Omega$, provided we, for the moment, ignore the boundary $\partial\Omega$. Thus, since $u$ is harmonic we have that $Z_k = u(X_k)$ is a \emph{martingale} \cite{williams1991probability}. This connection to probability theory allows simple alternative proofs of various estimates in harmonic function theory, such as Harnack's inequality and gradient estimates \cite{lewicka2022robin,lewicka2022robinII}, and have found applications in proving gradient estimates on graphs as well \cite{calder2022Lip}. 

Over the past 15 years, there has been significant interest in extending these martingale techniques to the $p$-Laplacian. To do this, however, the definition \eqref{eq:plap_def}  of the $p$-Laplacian is not very useful. Instead, provided that $\nabla u \neq 0$, we can drop the  $|\nabla u|^{p-2}$ term in \eqref{eq:neargame} to obtain the equation
\begin{equation}\label{eq:homo_p}
\Delta u + (p-2) \Delta_\infty u = 0,
\end{equation}
and restrict our attention to $p\geq 2$. Given a smooth function $u$, we can average the Taylor expansion for $u$ about $x_0$ to obtain
\[\dashint_{B(x_0,\epsilon)} u(y) \, dy = u(x_0)  + \frac{\epsilon^2}{2(d+2)}\Delta u(x_0) + O(\epsilon^3),\]
and so 
\begin{equation}\label{eq:mvp2}
\epsilon^2 \Delta u(x_0) = 2(d+2) \dashint_{B(x_0,\epsilon)}u(y) \, dy - 2(d+2)u(x_0) + O(\epsilon^3).
\end{equation}
Noting that the $\infty$-Laplacian is the second derivative of $u$ in the direction of the gradient $v = \frac{\nabla u(x_0)}{|\nabla u(x_0)|}$, we have
\begin{align}\label{eq:mvpi}
\epsilon^2\Delta_\infty u(x_0) &= \left[u\left( x_0 + \epsilon v\right) - u(x_0)\right] - \left[ u(x_0) - u\left( x_0 - \epsilon v\right)\right] + O(\epsilon^3) \\
&= \max_{B(x_0,\epsilon)}u + \min_{B(x_0,\epsilon)}u - 2u(x_0) + O(\epsilon^3).\notag
\end{align}
Inserting \eqref{eq:mvp2} and \eqref{eq:mvpi} into \eqref{eq:homo_p} we see that if $u$ is a smooth $p$-harmonic function, i.e., satisfying \eqref{eq:homo_p}, then 
\begin{equation}\label{eq:mvp_pharm}
u(x_0) = \alpha\, \dashint_{B(x_0,\epsilon)}u(y) \, dy  + \frac{1-\alpha}{2}\left(\max_{B(x_0,\epsilon)}u + \min_{B(x_0,\epsilon)}u\right)  +\red O(\epsilon^3)\nc.
\end{equation}
as $\eps \to 0$ where $\alpha = \frac{d+2}{d+p} \in [0,1]$ since $p\geq2$. Thus, while $p$-harmonic functions do not satisfy a mean value property for any size ball, \eqref{eq:mvp_pharm} gives an asymptotic version of a mean value property, with min and max terms arising from the $\infty$-Laplacian. In fact, the asymptotic mean value property \eqref{eq:mvp_pharm} \emph{characterizes} $p$-harmonic functions \cite{manfredi2010asymptotic}. It is also an interesting equation to study in its own right; when the $O(\epsilon^3)$ is dropped the functions are called \emph{$p$-harmonious} and studied in detail in \cite{manfredi2012definition}.

The mean value property \eqref{eq:mvp_pharm} suggests a way to adapt the random walk construction earlier to $p$-harmonic functions. We simply define a stochastic process $X_0,X_1,X_2,\dots$ so that, given $X_k$, $X_{k+1}$ is defined in the following way: with probability $\alpha$ we take a random walk step, so $X_{k+1}$ is uniformly distributed on $B(X_k,\eps)$, with probability $\frac{1-\alpha}{2}$ we set $X_{k+1} = \argmax_{B(X_k,\epsilon)}u$, and likewise with probability $\frac{1-\alpha}{2}$ we set \red $X_{k+1} = \argmin_{B(X_k,\epsilon)}u$\nc. When the $\argmax$ or $\argmin$ is not unique, we make a rule to break ties, which can be deterministic or random. By the definition of this stochastic process, for any continuous function $u$ we have
\[\E[u(X_{k+1}) \, | \, X_k] = \alpha\, \dashint_{B(X_k,\epsilon)}u(y) \, dy  + \frac{1-\alpha}{2}\left(\max_{B(X_k,\epsilon)}u + \min_{B(X_k,\epsilon)}u\right).\]
In particular, if $u$ is smooth and $p$-harmonic, then the discussion above shows that
\[\E[u(X_{k+1}) \, | \, X_k] = u(X_k) + \red O(\epsilon^3)\nc.\]
Thus, we have recovered the martingale property, at least asymptotically as $\eps\to 0$. 

The stochastic process introduced above is often described as a two player \emph{tug-of-war} game with noise. The game involves a token $X_k$ that is moved by two players and by random noise. Player I is trying to move the token to locations that maximize $u$, while the goal of player II is to move the token to places that minimize $u$. The game is played by flipping two coins. The first comes up heads with probability $\alpha$, and tails with probability $1-\alpha$. If the first coin comes up heads, the token $X_k$ is moved to a uniformly random point $X_{k+1}$ in the ball $B(X_k,\epsilon)$. If the first coin comes up tails, then the game switches to a \emph{tug-of-war} game, where a second unbiased coin is flipped to decide which player gets to move the token to decide $X_{k+1}$. Whichever player wins the second coin flip is allowed to move the token wherever they like in the ball $B(X_k,\epsilon)$; the idea being that player I will move the token to maximize $u$, while player II will minimize. The exact goal of each player depends on the boundary condition; they may want to maximize/minimize the value of $u$ when the game stops by hitting the boundary $\partial \Omega$, in which case only the values of $u$ on or near the boundary need to be specified in the game and the players are assumed to play optimal strategies to maximize or minimize the payoff --- the value of $u$ --- at the end of the game. The random walk step, and the randomness in choosing between players, is interpreted as \emph{noise}, hence the term tug-of-war with noise. 

There is a close connection between tug-of-war games and non-local elliptic equations. Let us define the nonlocal $2$ and $\infty$ Laplacians by
\[\Delta^\eps_2 u(x) = \dashint_{B(x_0,\epsilon)}u(y) \, dy - u(x), \ \ \text{and} \ \ \Delta^\eps_\infty u(x) =  \frac{1}{2}\left(\max_{B(x,\epsilon)}u + \min_{B(x,\epsilon)}u\right) - u(x).\]
Then if we drop the error term in \eqref{eq:mvp_pharm} and rearrange, we arrive at the equation
\begin{equation}\label{eq:nonlocal_plaplacian}
\Delta^\eps_p u:=\alpha \Delta^\eps_2 u + (1-\alpha)\Delta^\eps_\infty u = 0.
\end{equation}
The operator $\Delta^\eps_p$ on the left above is a nonlocal approximation to the $p$-Laplacian that arises from the tug-of-war game perspective, and is closely related to the graph $p$-Laplacian discussed in Section \ref{sec:ssl}. We can also consider a corresponding nonlocal boundary value problem
\begin{equation}\label{eq:nonlocal_plaplace}
\left\{
\begin{aligned}
\Delta^\eps_p u &=  0,&& \text{in } \Omega_\eps\\
u &= g,&& \text{on } \partial_\eps \Omega,
\end{aligned}
\right.
\end{equation}
where $g$ is given, $\partial_\eps \Omega = \partial\Omega + B(0,\epsilon)$ and $\Omega_\eps = \Omega \setminus \partial_\eps \Omega$. If we choose the stopping time $\tau$ to be the first time that the tug-of-war game hits the boundary strip $\partial_\eps \Omega$, then the martingale property and the optional stopping theorem yield
\[u(x) = \E[u(X_\tau) \, | \, X_1 = x].\]
This gives a representation formula for solutions of the non-local $p$-Laplace equation \eqref{eq:nonlocal_plaplacian} that is useful for studying properties of the solution through martingale techniques. Properties that are independent of the nonlocal scale $\epsilon>0$ are inherited by $p$-harmonic functions by sending $\eps\to 0$. 

Tug-of-war games for the $p$-Laplacian were originally introduced in \cite{peres2008tug,peres2009tug} with a version that holds for $1 \leq p \leq \infty$ using ideas from earlier work on deterministic games for the $1$-Laplacian \cite{kohn2006deterministic}. The version for $p\geq 2$ described above was introduced and studied in \cite{manfredi2012dynamic}. This work motivated a study of the game-theoretic $p$-Laplacian on graphs \cite{manfredi2015nonlinear} as well as finite difference approaches for numerically approximating solutions \cite{oberman2013finite,oberman2005convergent,armstrong2012finite}. The tug-of-war interpretation of the $p$-Laplacian has led to simple alternative proofs of regularity for $p$-harmonic functions, including the Harnack inequality and gradient estimates \cite{luiro2013harnack,attouchi2021gradient}. In addition, many variants of tug-of-war games have been introduced, including games with bias \cite{peres2010biased},  mixed Neumann/Dirichlet boundary conditions \cite{charro2009mixed}, obstacle problems \cite{lewicka2017obstacle}, nonlocal tug-of-war for the fractional $p$-Laplacian \cite{lewicka2022non}, time dependent equations \cite{han2022time}, and variants on the core structure of the game \cite{lewicka2018random}. We also refer the reader to the survey article \cite{lewicka2014game} and two recent books on tug-of-war games \cite{parviainen2024notes,lewicka2020course} for more details.

%% file: sections/ssl.tex

\section{Semi-supervised learning and PDEs}
\label{sec:ssl}

In this section, we overview graph-based semi-supervised learning and the recent connections to PDEs in the continuum limit, with a specific focus on the $p$-Laplacian and tug-of-war games. 

\subsection{Graph-based semi-supervised learning}

\begin{figure}[!t]
\centering
\subfloat[Data]{\includegraphics[width=0.32\textwidth]{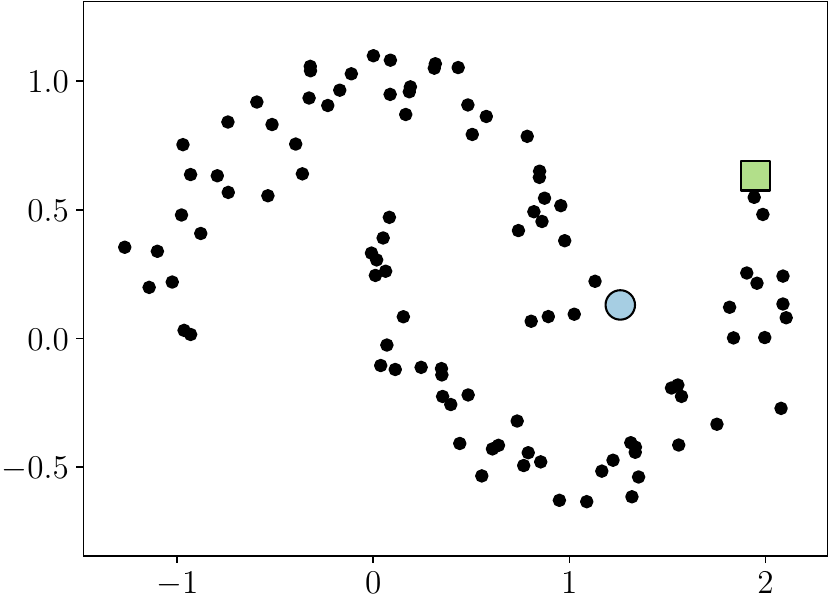}\label{fig:data}}
\subfloat[Fully supervised]{\includegraphics[width=0.32\textwidth]{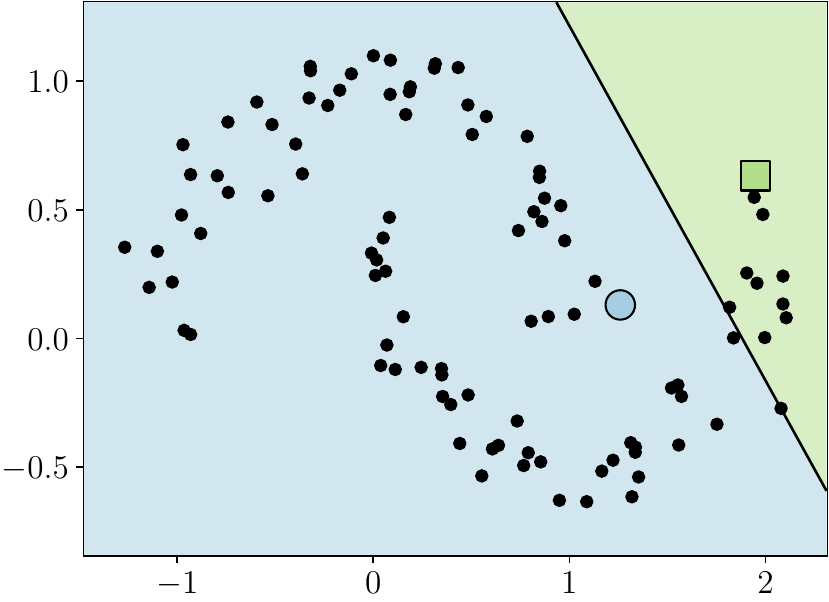}\label{fig:fsl}}
\subfloat[Semi-supervised]{\includegraphics[width=0.32\textwidth]{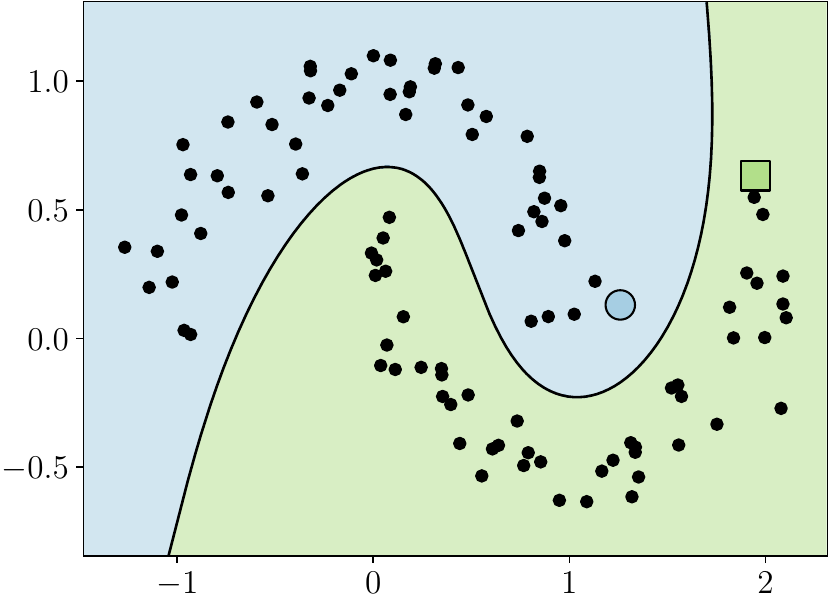}\label{fig:ssl}}
\caption{A toy example comparing fully and semi-supervised learning. In (a) we show a data set with 2 labeled examples --- the blue circle and green square --- along with $98$ unlabeled data points --- the black dots. In (b) we show the decision regions from training a fully supervised classification algorithm, while in (c)  we show the decision regions for a semi-supervised learning algorithm, which uses the unlabeled data to inform the decision boundary.}
\label{fig:ssl_vs_fully}
\end{figure}
Semi-supervised learning uses both labeled and unlabeled data.  As a toy example, in Figure \ref{fig:data} we show the famous two-moons data set with two labeled data points; the blue circle and green square. The black points are the \emph{unlabeled data}, which are not used by fully supervised learning, and will be discussed momentarily. With only two labeled data points, it is difficult to learn a general function $f$ that will correctly classify new data points. In Figure \ref{fig:fsl} we show the result of training a linear kernel support vector machine (SVM) with these two data points, which finds the linear decision boundary with maximal margin --- in this case the line equidistant\footnote{The line does not appear orthogonal to the vector between the training points because the axes are scaled differently.} from the two training points. Given no other information, this may be a reasonable thing to do. However, suppose now that we also have access to the black points, which are \emph{unlabeled} data points. That is, we have access to the coordinates $x_i\in \R^2$, but not the label $y_i$, which in this case is $y_i\in \{0,1\}$ for binary classification, and these unlabeled data points are the \emph{only} data points we will want to apply our classifier to in the future. In this case, the linear SVM decision boundary is a poor choice, since it cuts through a dense region of the unlabeled data --- the lower moon --- where we may not expect a true decision boundary to lie. Instead, we should seek to place the decision boundary in sparse regions between clusters. In Figure \ref{fig:ssl} we show the result of using a semi-supervised learning algorithm on the same data set, which correctly separates the two moons.\footnote{To be precise, we applied graph-based Poisson learning, described below, to obtain label predictions for all unlabeled points, and then we trained a radial basis kernel SVM using the label predictions. } 

There are many ways to incorporate unlabeled data into a machine learning algorithm. One common and successful approach is to utilize \emph{graph-based} learning, where each node in the graph corresponds to a data point, and the edges in the graph correspond to either intrinsic relationships between data points, or record some notion of similarity between data points. Many types of data have intrinsic graph structures, like molecules in drug discovery problems, citation datasets, or networks like the internet. In other problems, like image classification, the graph structure is not intrinsic, but can be constructed as a \emph{similarity graph}, in which similar data points are connected by an edge with a weight that encodes the degree of similarity. Let $\X=\{x_1,\dots,x_n\}$ be a set of data points, where $x_i\in \R^d$.  Here, we assume we have a graph with vertex set $\X$ described by a symmetric weight matrix $W=W^T\in \R^{n\times n}$ in which each entry $\wij$, for $x ,y \in \X$, is non-negative and encodes a notion of similarity between $x $ and $y $. A zero weight $\wij =0$ indicates there is no edge between $x $ and $y $, while a positive weight $\wij >0$ indicates the presence of an edge, and the larger the value the stronger the edge. We let $N_x\subset \X$ denote the graph neighbors of the vertex $x\in \X$, which is defined by
\begin{equation}\label{eq:neigh}
N_x = \{y\in \X\, : \, w_{xy}>0\}.
\end{equation}

A common way to construct a graph over a data set is the geometric graph construction
\begin{equation}\label{eq:eps_weights}
\wij  = \eta\left( \frac{|x -y |}{\eps}\right),
\end{equation}
where $\eps>0$ is the graph \emph{bandwidth} and $\eta:[0,\infty)\to [0,\infty)$ is a nonnegative function that is typically decreasing with compact support. A common choice is Gaussian weights where $\eta(t) = e^{-t^2}$, with a possible truncation to zero at some distance $t=\tau$. One issue with the geometric graph construction is that one has to choose a very large value for the graph bandwidth $\eps>0$ to ensure the graph is connected, even in the sparsest regions of the data set, which leads to a very non-sparse weight matrix $W$ that is difficult to work with computationally. To address this, it is common to adjust the bandwidth $\eps$ locally to reflect the density (or sparsity) of the data set, which results in various types of $k$-nearest neighbor graphs. One common construction is
\[\wij =  \eta\left( \frac{|x -y |}{\sqrt{\epsi\epsj}}\right),\]
where $\eps_i$ is the distance from $x $ to its $k^{\rm th}$ nearest neighbor, though other choices are possible. Throughout this section, we'll generally assume a graph with geometric weights of the form \eqref{eq:eps_weights}. 

\subsection{Laplacian regularization}

Suppose now that we have a graph structure over our data set, given by a weight matrix $W=(\wij )\in \R^{n\times n}$ along with a subset of graph nodes $\Gamma\subset \X$ that are labeled, with corresponding labels $y=g(x)\in \R$ for $g:\Gamma\to \R$ a given labeling function, which we are taking to be scalar only for simplicity of discussion. The seminal approach in graph-based semi-supervised learning is \emph{Laplacian regularization}, initially proposed in \cite{zhu2003semi}, which propagates the labels from $\Gamma$ to the rest of the graph by solving the optimization problem 
\begin{equation}\label{eq:graph_dirichlet}
\min_{u:\X\to\R}\sum_{x ,y \in \X} \wij (u(x ) - u(y ))^2,
\end{equation}
subject to the constraint that $u(x ) = g(x )$ for $x  \in \Gamma$. The real-valued solution $u(x )$ is then thresholded to the nearest label to make a class prediction.  The energy in \eqref{eq:graph_dirichlet} is called the \emph{graph Dirichlet energy} and the minimizer $u:\X\to \R$ is exactly the harmonic extension of the labeled data, which satisfies the boundary value problem 
\begin{equation}\label{eq:laplace_learning}
\left\{
\begin{aligned}
\L u(x ) &= 0,&& \text{if } x  \in \X\setminus \Gamma \\
u(x ) &= g(x ),&& \text{if } x \in \Gamma,
\end{aligned}
\right.
\end{equation}
where $\L$ is the graph Laplacian defined by 
\begin{equation}\label{eq:graph_laplacian}
\L u(x ) = \sum_{y \in \X} \wij (u(x ) - u(y )).
\end{equation}
In other words, we are seeking the \emph{smoothest} function --- in this case harmonic --- that correctly classifies the labeled data points. In semi-supervised learning, this is often called the \emph{semi-supervised smoothness assumption} \cite{chapelle2006}, which stipulates that the labeling is smooth in high density regions of the data set. Since the initial development in \cite{zhu2002learning,zhu2003semi}, graph-based learning using Laplacian regularization, and variations thereof, has grown into a wide class of useful techniques in machine learning \cite{bozorgnia2023graph,ando2007learning, zhou2004learning, zhou2004random, zhou2005learning, zhou2004ranking,he2004manifold, he2006generalized, wang2013multi, xu2011efficient, yang2013saliency,belkin2004semi,belkin2004regularization,bengio2006label,zhou2011semi,wang2013dynamic,ham2005semisupervised,lee2013graph,calder2020poisson,calder2020properly,belkin2002using,cervino2023learning}.

\red
Graph harmonic functions satisfy a mean value property, similar to harmonic functions on $\R^d$. Indeed, if we rearrange the condition $\L u(x) = 0$ we obtain
\begin{equation}\label{eq:mvp_graph}
u(x) = \frac{1}{d_x}\sum_{y\in \X}w_{xy}u(y),
\end{equation}
where $d_x = \sum_{y \in \X} \wij $ is the degree of vertex $x$. This is a local version of the continuous mean value property for harmonic functions. It says that the value $u(x)$ of a graph harmonic function at a node $x$ is equal to the weighted average of the values $u(y)$ at neighboring nodes $y$, weighted by the graph edge weights $w_{xy}$. In fact, one way to solve the equation \eqref{eq:laplace_learning} is by iterating the mean value property
\begin{equation}\label{eq:mvp_iterate}
u_{k+1}(x) = \frac{1}{d_x}\sum_{y\in \X}w_{xy}u_k(y),
\end{equation}
for $x\in \X\setminus \Gamma$, fixing $u_k(x)=g(x)$ for $x\in \Gamma$. This is exactly the classical Jacobi iteration for solving a linear system, and while it is not the fastest technique --- the preconditioned conjugate gradient method is much faster --- it has a nice interpretation as \emph{propagating} labels using the neighborhood structure of the graph, and is thus called \emph{label propagation} in the literature; see \cite{zhu2002learning}. Equation \eqref{eq:mvp_iterate} can also be viewed as a diffusion equation on the graph.
\nc

Laplacian regularized learning has an important connection to random walks on graphs. Let $X_1,X_2,\dots$ be a random walk on the vertices $\X$ of the graph, with transition probabilities $P=D^{-1}W$, where $D=\diag(d_{x_1},\dots,d_{x_n})$. That is, the probability $p_{x y }$ of the random walker moving from vertex $x $ to vertex $y $ is given by $p_{x y } = \di^{-1}\wij $. Given a function $u:\X\to\R$ on the vertices of the graph, we can compute
\[\E[u(X_{k+1}) - u(X_k) \, | \, X_k=x ] = \sum_{y \in \X} p_{x y }(u(y ) - u(x )) = -\frac{1}{d_{x }}\L u(x ).\]
The object on the right hand side is called the \emph{random walk} graph Laplacian, and we denote it by $\Lrw$; that is
\[\Lrw u(x ) = \frac{1}{d_{x }}\L u(x ).\]
The computation above thus yields
\begin{equation}\label{eq:martingale_laplace}
\E[u(X_{k+1}) \, | \, X_k] = u(X_k) - \Lrw u(X_k),
\end{equation}
and thus, the random walk graph Laplacian is the generator for the random walk. In particular, if $u$ is graph harmonic, so that $\L u(x )=0$, then $u(X_k)$ is a \emph{martingale}. Defining the stopping time $\tau$ as the first time the random walk hits the labeled data set $\Gamma$, the optional stopping theorem yields that the solution $u$ of Laplace learning \eqref{eq:laplace_learning} satisfies
\begin{equation}\label{eq:randomwalk_interp}
u(x ) = \E[u(X_\tau) \, | \, X_1 = x ] = \sum_{y \in \Gamma} \P(X_\tau = y  \, | \, X_1=x )g(y ).
\end{equation}
Thus, the solution of Laplacian regularized graph-based learning \eqref{eq:laplace_learning} can be interpreted as a weighted average of the given labels $g(y )$, weighted by the probability of the random walk hitting $y \in \Gamma$ first, before hitting any other labeled data point. This is a reasonable thing to do for semi-supervised learning at an intuitive level as well.  Indeed, if each cluster in the data set is well separated from the others, then provided the label rate is not too low, the random walk is likely to stay in the cluster it started in long enough to hit a labeled data point in that cluster, giving the correct label and propagating labels well within clusters.

At moderately low label rates, Laplacian regularization performs very well for graph-based semi-supervised learning. However, the performance can become quite poor at extremely low label rates, which was first noted in \cite{nadler2009semi} and later in \cite{el2016asymptotic}. When there are very few labels, the solution of the Laplace learning problem \eqref{eq:laplace_learning} prefers to be nearly constant over the whole graph, with sharp spikes at the labeled data points, as depicted in Figure \ref{fig:tp_laplace}. This is the configuration that gives the least graph Dirichlet energy, since the spikes are relatively inexpensive when the number of labeled data points is small. From the random walk perspective, when there are very few labels, the random walk takes very long to hit a labeled data point, so the stopping time $\tau$ is very large, and the distribution of the random walker approaches the limiting invariant distribution on the graph, which is proportional to the degree. Thus, by \eqref{eq:randomwalk_interp} we have
\begin{equation}\label{eq:randomwalk_limit}
u(x ) \approx \frac{\sum_{y \in \Gamma} \dj g(y )}{\sum_{y  \in \Gamma}\dj}
\end{equation}
at very low label rates. When the labels are binary $g(y)\in \{-1,1\}$, the sign of the right hand side of \eqref{eq:randomwalk_limit} is constant over the graph, and simply chooses the class whose labeled points have the largest cumulative degree, which is one way to explain the very poor performance of Laplace learning at low label rates. 

Even in settings where moderate amounts of labeled data will eventually be acquired, it is often necessary to consider, at least initially, problems with very few labels where Laplace learning is not useful. One of those areas is \emph{active learning}, which refers to machine learning algorithms that address the problem of choosing the \emph{best} data points to label in order to obtain superior model performance with as few labels as possible \cite{settles2009active}. Active learning methods  incorporate a human-in-the-loop that can be queried to label new data points as needed, and the goal is to achieve the highest accuracy with as few labels as possible by intelligently choosing the next point to label, often in a sequential manner. Active learning procedures normally start out with extremely small labeled sets that slowly grow during the acquisition of new labeled data points, which requires graph-based semi-supervised learning algorithms that perform well both with small and moderate amounts of labeled data. A number of graph-based active learning methods have been proposed \cite{zhu_combining_2003, ji_variance_2012, ma_sigma_2013, qiao_uncertainty_2019, miller2021model, murphy_unsupervised_2019}, including recent methods inspired by the PDE-continuum limits discussed below \cite{miller2023active}, and recent applications have been found in image classification problems \cite{enwright2023deep,chapman2023batch,brown2023contrastive,miller2022graphbased}.

\subsection{PDE-inspired insights and algorithms}

Another interpretation of the poor performance of Laplace learning at low label rates is through PDE continuum limits. To explain this, we need to place some additional assumptions on the graph construction. We assume the vertices $\X=\{x_1,\dots,x_n\}$ of the graph are \iid~random variables distributed on a domain $\Omega \subset \R^d$ with a continuous and positive density $\rho:\Omega\to (0,\infty)$. We assume the random geometric graph construction, which uses weights \eqref{eq:eps_weights} for a given bandwidth $\eps>0$. This is called a \emph{random geometric graph}, and in this setting, the expectation of the graph Dirichlet energy in \eqref{eq:graph_dirichlet} is
\begin{equation}\label{eq:dirichlet_continuum}
n^2\int_\Omega\int_\Omega \eta\left( \frac{|x-y|}{\eps}\right)(u(x) - u(y))^2\rho(x)\rho(y) \, dxdy \approx C_\eta n^2 \eps^{d+2}\int_\Omega \rho^2 |\nabla u|^2 \, dx,
\end{equation}
provided $u:\Omega\to \R$ is a smooth function and $\rho$ is Lipschitz, where $C_\eta$ is a positive constant depending only on $\eta$. The asymptotics in \eqref{eq:dirichlet_continuum} can be verified formally with Taylor expansions of $u$ and $\rho$, and can also be established rigorously in the language of Gamma convergence, as in \cite{garcia2016continuum}. Thus, the continuum version of Laplace learning \eqref{eq:laplace_learning} is the boundary value problem 
\begin{equation}\label{eq:laplace_continuum}
\left\{
\begin{aligned}
\div(\rho^2 \nabla u) &= 0,&& \text{in }  \Omega\setminus \Gamma\\
u &= g,&& \text{on } \Gamma,
\end{aligned}
\right.
\end{equation}
where $\Gamma\subset \Omega$ is purposely vaguely specified, and should encode some notion of the continuum limit of the labeled data set, and $g:\Gamma\to \R$ the continuum limit of the values of the labels. When $\Gamma$ does not contain the entire boundary $\partial \Omega$, we also must impose homogeneous Neumann conditions on $\partial \Omega \setminus \Gamma$. 

The presence of the squared density $\rho^2$ as a diffusion coefficient in the PDE continuum limit illustrates how Laplace learning is able to diffuse labels quickly in high density regions where $\rho$ is large, and place decision boundaries, which correspond to sharp transitions in the label function $u$, in sparse regions where $\rho$ is small. However, the presence of the labeled set $\Gamma$ as a boundary condition is problematic, since \eqref{eq:laplace_continuum} is well-posed only when $\Gamma$ is large enough, and regular enough. \red This is due to the fact that the Sobolev function space $H^1(\Omega)$, like the $L^p(\Omega)$ function spaces, consists of functions that are defined only up to sets of measure zero. So the values of $u(x)$ at particular points $x\in \Omega$ are not well-defined. This means, for example, that $\Gamma$ cannot be a collection of isolated points, as in Figure \ref{fig:tp_laplace}, otherwise \eqref{eq:laplace_continuum} is ill-posed. In order for \eqref{eq:laplace_continuum} to be well-posed, the boundary set $\Gamma$ must be \emph{large enough} and sufficiently regular so that we can intepret what it means to \emph{evaluate} $u(x)$ for $x\in \Gamma$. Such results are called \emph{trace theorems} in the PDE literature, and the restriction $u\vert_\Omega$ is the trace of $u$ on $\Gamma$; we refer to \cite{EvansPDE} for more details.  \nc

\begin{figure}[!t]
\centering
\subfloat[Laplace \cite{zhu2003semi}]{\includegraphics[height=0.2\textheight]{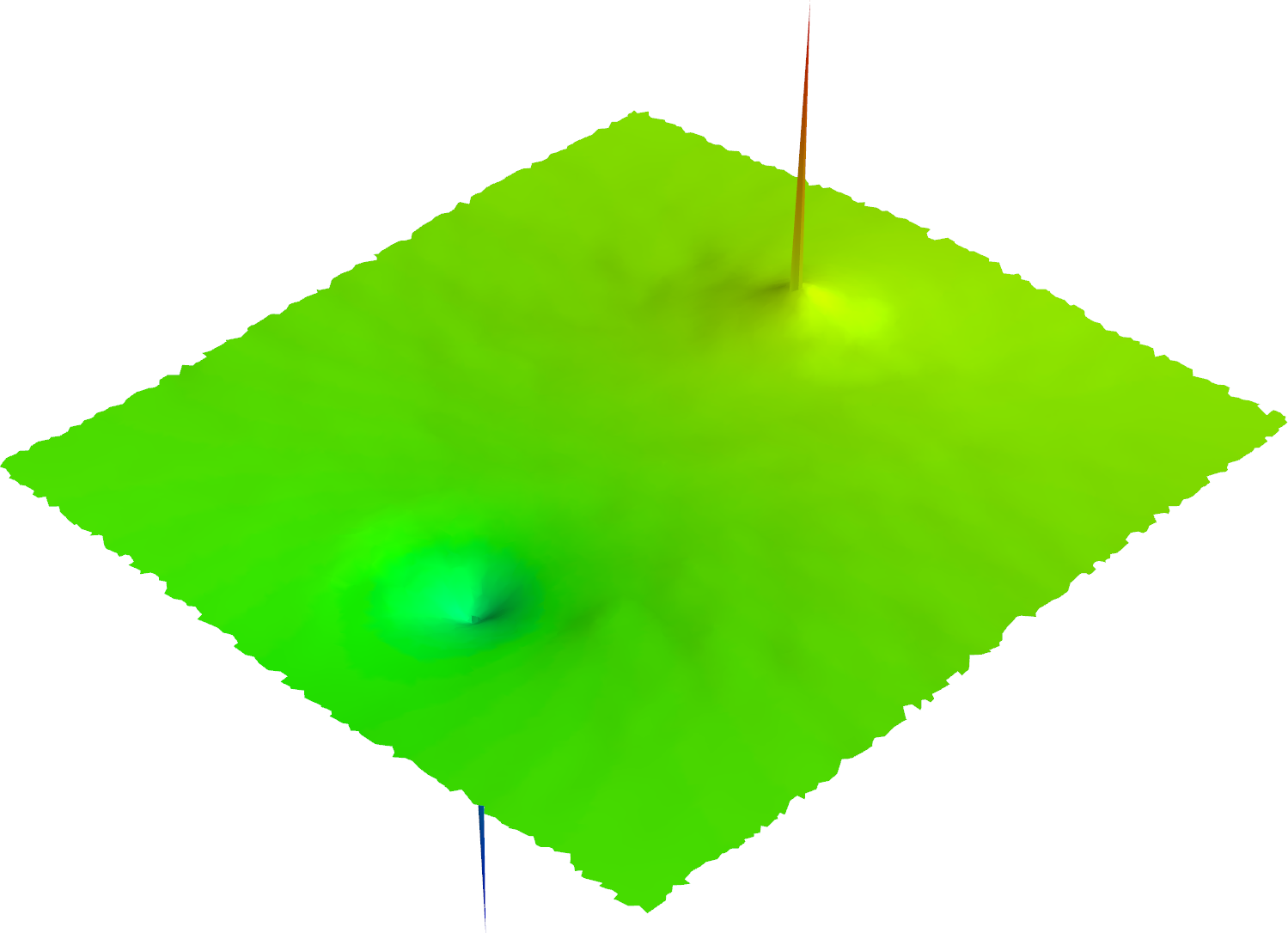}\label{fig:tp_laplace}}
\subfloat[$p$-Laplace, $p=2.05$ \cite{flores2022algorithms}]{\includegraphics[height=0.2\textheight]{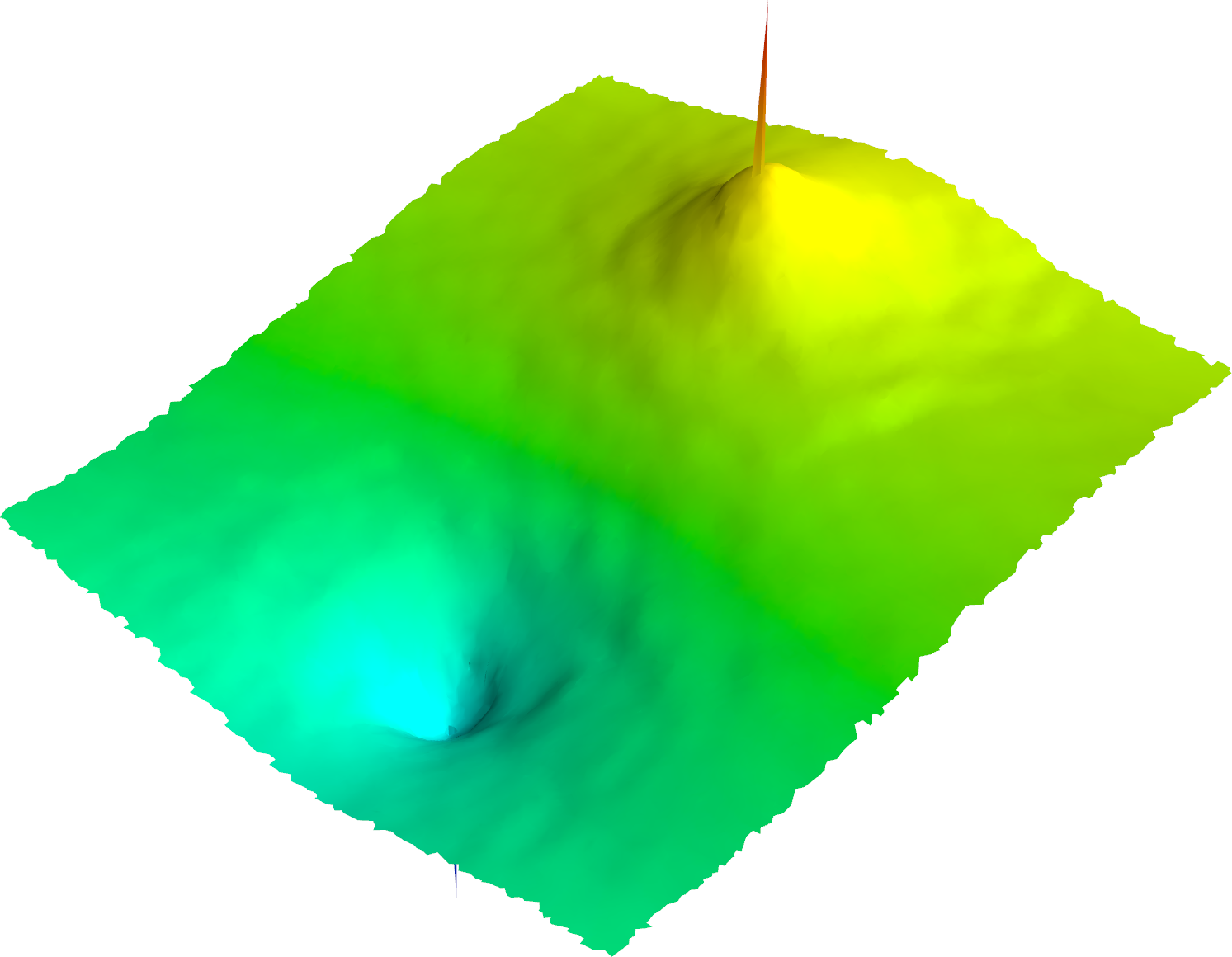}\label{fig:tp_plaplace1}}
\subfloat[$p$-Laplace, $p=2.25$ \cite{flores2022algorithms}]{\includegraphics[height=0.2\textheight]{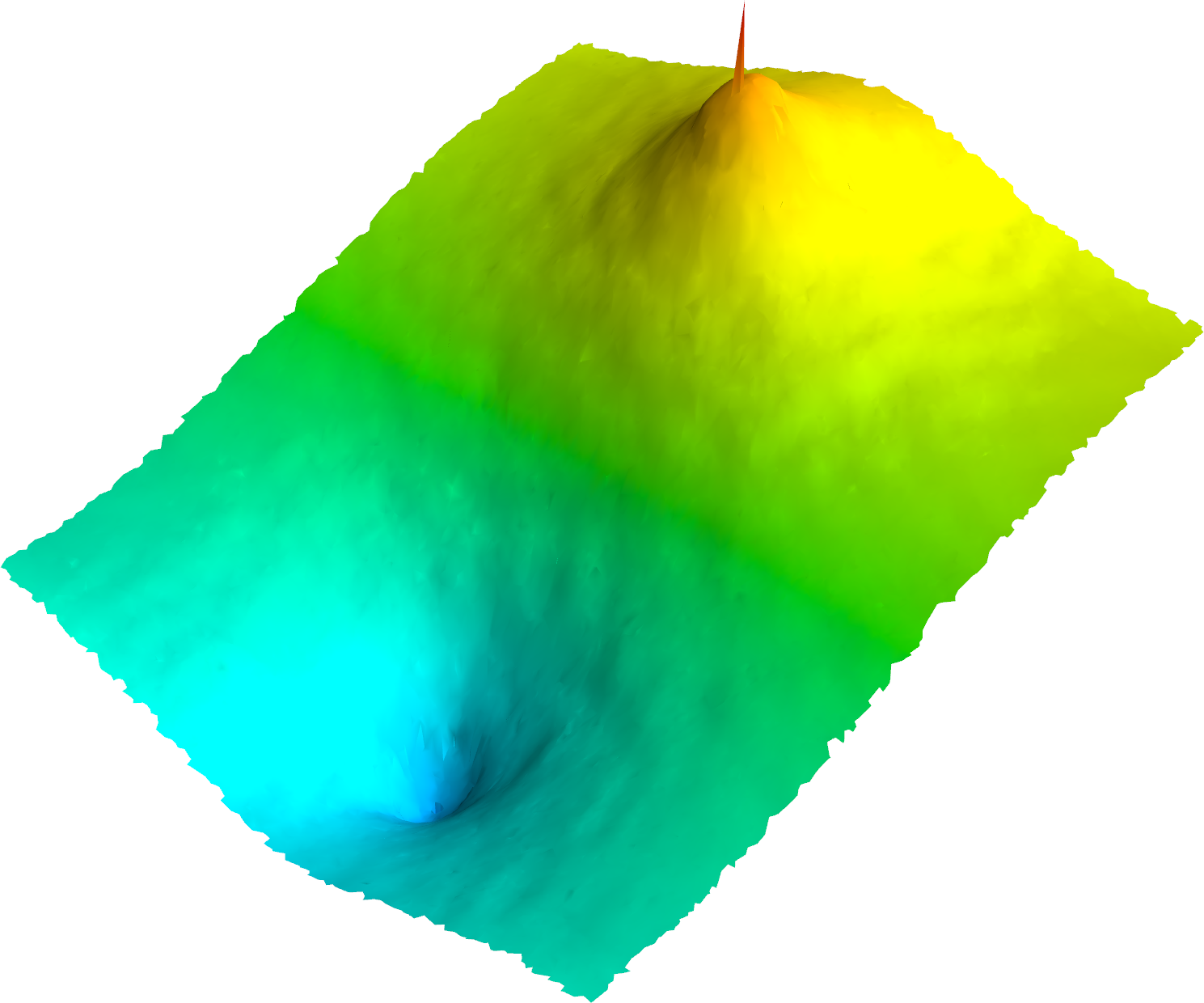}\label{fig:tp_plaplace2}}\\
\subfloat[WNLL \cite{shi2017weighted}]{\includegraphics[height=0.25\textheight]{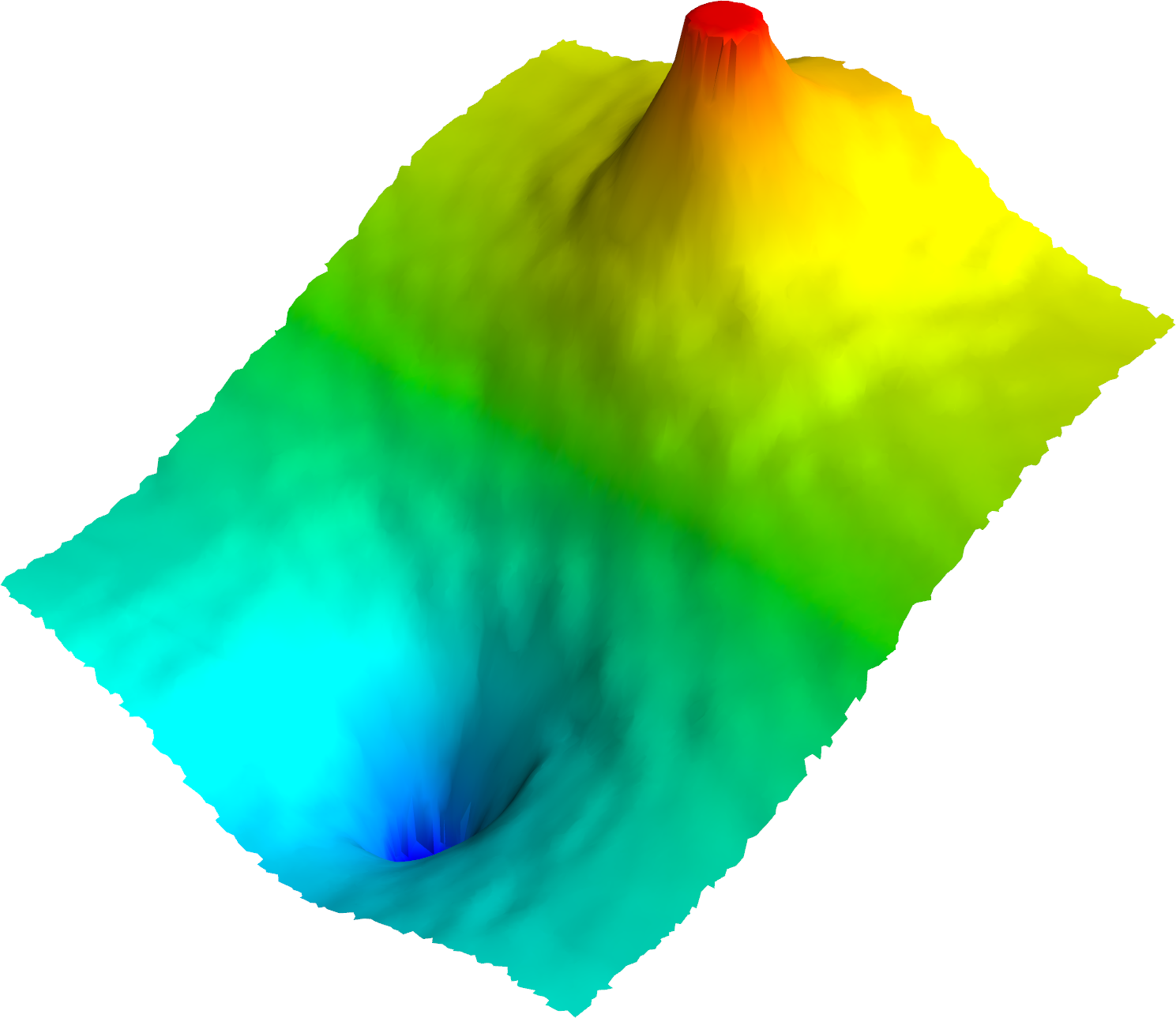}\label{fig:tp_wnll}}
\subfloat[High-order Laplace, $m=2$ \cite{zhou2011semi}]{\includegraphics[height=0.25\textheight]{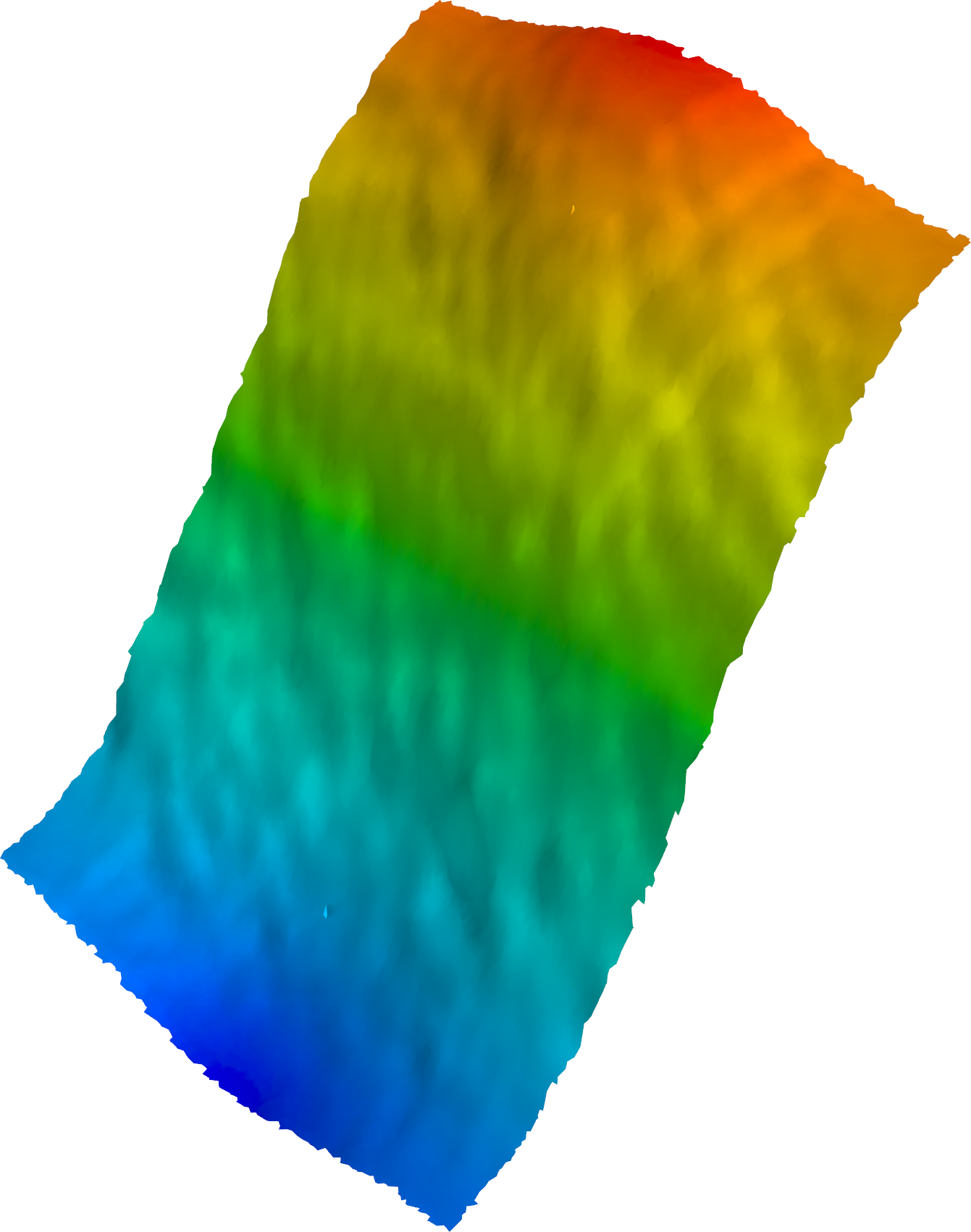}\label{fig:tp_ho}}
\subfloat[PWLL \cite{miller2023active,calder2023poisson}]{\includegraphics[height=0.25\textheight]{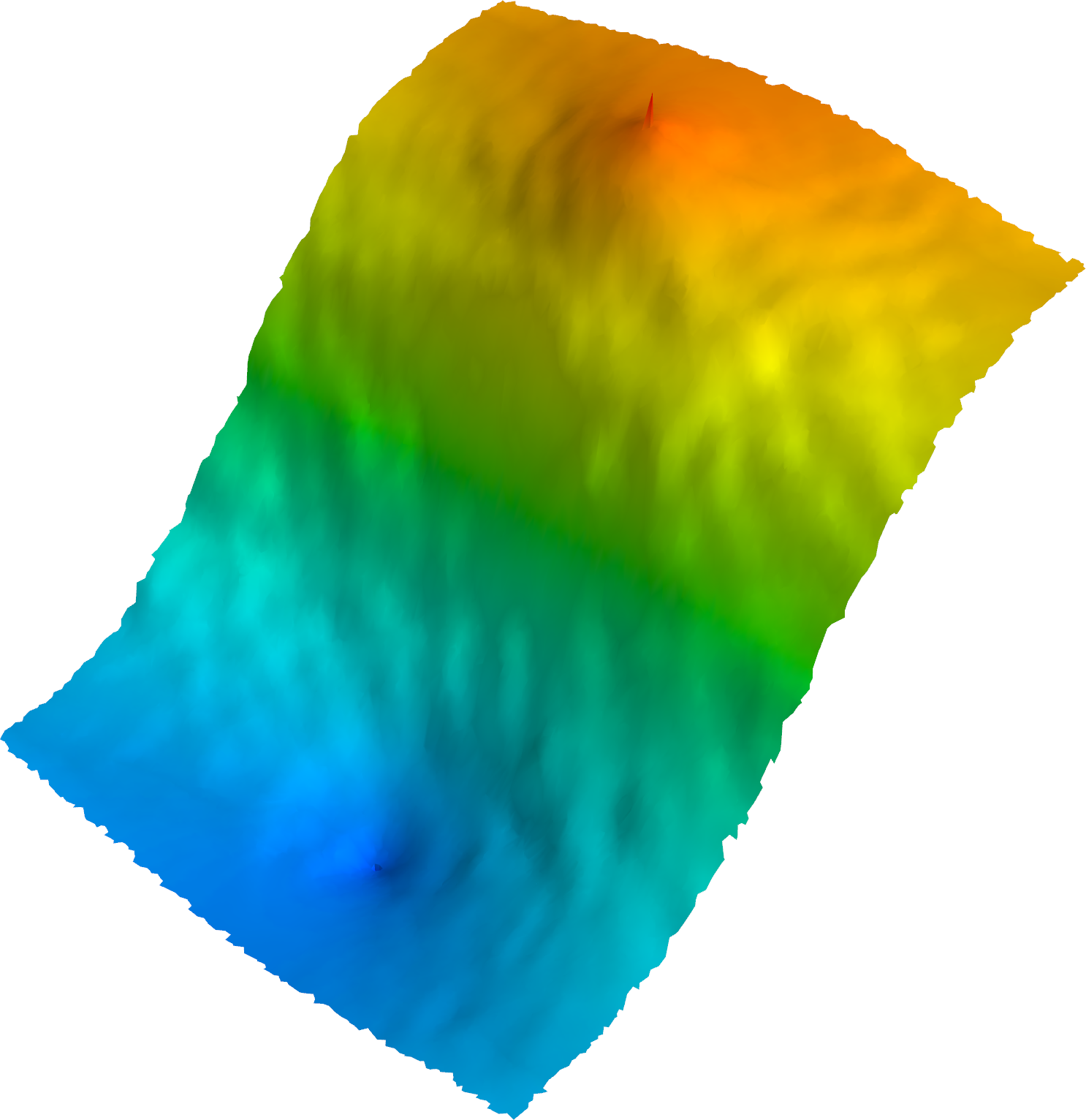}\label{fig:tp_poisson}}
\caption{Comparison of different semi-supervised learning methods on a toy example with two labeled data points with labels $+1$ and $-1$. The graph consists of $n=20,000$ uniformly distributed random variables on the unit box $[0,1]^2$ with geometric Gaussian kernel weights with $\eps=0.05$. \red Here WNLL stands for Weighted NonLocal Laplacian \cite{shi2017weighted}, and PWLL stands for Poisson Weighted Laplace Learning \cite{miller2023active,calder2023poisson}.\nc}
\label{fig:ssl_two_point}
\end{figure}

The last decade has seen a surge of interest in developing graph-based learning algorithms that are well-posed with arbitrarily small labeled data sets, by considering PDE continuum limits that are insensitive to the labeled data set. One natural idea proposed in \cite{elmoataz2015p} is to consider the $p$-Dirichlet energy on the graph, given by 
\begin{equation}\label{eq:pdirichlet_energy}
\sum_{x ,y \in \X} \wij |u(x ) - u(y )|^p.
\end{equation}
In fact, the $p$-Laplacian was suggested in many other works as well \cite{alamgir2011phase, bridle2013p,zhou2005regularization,elmoataz2015p,elmoataz2017game,elmoataz2017nonlocal,hafiene2018nonlocal,jung2016semi,ennaji2023tug}, though not in the low label rate context.  
By a similar argument as above, the continuum limit energy should be the weighted $p$-Dirichlet energy 
\[\int_\Omega \rho^2 |\nabla u|^p \, dx.\]
When $p>d$, we have the Sobolev embedding $W^{1,p}(\Omega) \subset C^{0,\alpha}(\Omega)$ \cite{EvansPDE}, which ensures that solutions are H\"older continuous, spikes cannot form, and the continuum boundary value problem
\begin{equation}\label{eq:plaplace_continuum}
\left\{
\begin{aligned}
\div(\rho^2|\nabla u|^{p-2} \nabla u) &= 0,&& \text{in }  \Omega\setminus \Gamma\\
u &= g,&& \text{on } \Gamma,
\end{aligned}
\right.
\end{equation}
is well-posed without any conditions on $\Gamma$. This well-posedness was established rigorously in the continuum limit in \cite{slepcev2019analysis} when $p >d$ and the number of labeled points is finite and fixed as the number of unlabeled data points tends to infinity. In this case, the authors of \cite{slepcev2019analysis} identified an additional restriction on the length scale $\epsilon>0$ that arises through the nonlocal nature of the graph Laplacian. Namely, in addition to $p>d$, it is necessary that $n \eps^{p} \ll 1$ to ensure that spikes do not form. Since $n\eps^d \gg \log n$ is necessary to ensure connectivity of the graph,\footnote{The quantity $n\epsilon^d$ is the average number of neighbors of any nodes on the graph, up to a constant.} this condition can only be satisfied when $p>d$. The authors of \cite{slepcev2019analysis} also proposed a modification of $p$-Laplace learning for which the condition $n\epsilon^p \ll 1$ is not required, by essentially extending the labels to all nearest neighbors on the graph. 

It is worthwhile to pause and note that the condition $n\epsilon^{p} \ll 1$ comes from a simple balancing of energy between ``spiky'' functions and smooth functions on the graph. If we write \eqref{eq:dirichlet_continuum} for the $p$-Laplacian, we see that the graph $p$-Dirichlet energy of a smooth function scales like $n^2\eps^{p+d}$. On the other hand, the energy of a constant function with spikes at the labeled nodes is
\[\sum_{x \in \Gamma}\sum_{y \in \X} \eta\left( \frac{|x -y |}{\eps}\right)|g(y ) - c|^p \approx C |\Gamma| S^p n\eps^d,\]
where $C$ is a constant depending on $\eta$, $c=u(x )$ is the constant value of $u$, $S$ is the size of the spikes, and $|\Gamma|$ is the number of labeled points, all provided the label rate is sufficiently low and the labeled points are sufficiently far apart. In order to ensure that a smooth interpolating function has less energy than a spiky one, in the setting where $|\Gamma|$ is constant with $n$ and $\eps$ we require
\[n^2 \eps^{p+d} \ll n\epsilon^d \iff n\epsilon^p \ll 1.\]
If we allow the number of labels $|\Gamma|$ to depend on $n$ and $\eps$, then the condition is 
\[n^2 \eps^{p+d} \ll |\Gamma|n\epsilon^d \iff \beta:=\frac{1}{n}|\Gamma| \gg \eps^p.\]
The quantity $\beta$ above is the \emph{label rate} --- the ratio of labeled data points to all data points --- and simply energy balancing arguments show that the label rate $\beta$ should satisfy $\beta \gg \eps^p$ to ensure convergence to a well-posed continuum problem. The case of finite labels as $n\to \infty$ corresponds to $\beta=n^{-1}$, which recovers the condition $n\epsilon^p \ll 1$. In \cite{calder2023rates} it was shown that $\beta \sim \eps^2$ for $p=2$ is the threshold for a spiky versus smooth continuum limit in a setting that allowed the labeled data set to grow as $n\to \infty$. The proofs in \cite{calder2023rates} used the random walk interpretation of the $2$-Laplacian and martingale arguments. A negative result was also established in \cite{calder2023rates} for $p>2$ showing that spikes develop when $\beta \ll \eps^p$. However, a positive result for $p>2$ --- that $\beta \gg \eps^p$ is sufficient to prevent spikes from forming --- is currently lacking, outside of the finite label regime covered in \cite{slepcev2019analysis}. 

As in Section \ref{sec:tug} we can take the limit as $p\to \infty$ of the graph $p$-Dirichlet energy. It is more illustrative to do this with the necessary conditions for minimizing the graph $p$-Dirichlet energy \eqref{eq:pdirichlet_energy}, which is the graph $p$-Laplace equation
\begin{equation}\label{eq:plaplace_graph}
\sum_{y \in \X} \wij |u(x ) - u(y )|^{p-2}(u(x ) - u(y )) = 0
\end{equation}
for each $x \in \X\setminus \Gamma$.  To take the limit as $p\to \infty$, we separate the terms in the summation above by sign, and take the $p^{\rm th}$ root to obtain 
\[\left(\sum_{y \, : \, \gij >0 }^n \wij (u(x ) - u(y ))^{p-1}\right)^{1/p} = \left(\sum_{y \, : \, \gji  >0 }^n \wij (u(y ) - u(x ))^{p-1}\right)^{1/p}.\]
where $\gij = \wij (u(x ) - u(y ))$. The terms with $\gij =0$ do not contribute and can be neglected. We assume neither sum above is empty, otherwise all terms in \eqref{eq:plaplace_graph} are zero, which is trivial.  Sending $p\to \infty$ yields the equation
\[\max_{y \, : \, \gij > 0} (u(x ) - u(y )) = \max_{y \, : \, \gji > 0} (u(y ) - u(x )).\]
The maximums above are unchanged by replacing the conditions $\gij >0$ and $\gji >0$ with $\wij  > 0$, as the graph is symmetric so $\wij =w_{y  x}$. Making this replacement, dividing by two on both sides, and rearranging yields
\begin{equation}\label{eq:graph_infty_laplacian}
\L_\infty u(x ) := u(x ) - \frac{1}{2}\left( \max_{N_x } u +  \min_{N_x } u \right) = 0,
\end{equation}
where $\L_\infty$ is called the \emph{graph $\infty$-Laplacian}, and we recall that $N_x $ are the graph neigbhors of $x $ defined in \eqref{eq:neigh}. Notice the maximum and minimum above are over graph neighbors, which satisfy $\wij  > 0$, but that the weights in the graph do not enter directly into the operator otherwise. If we were to replace $\wij $ by $\wij ^p$ in \eqref{eq:plaplace_graph}, then the weights would appear in the graph $\infty$-Laplacian, which is more informative of the graph structure, and was the approach taken in \cite{calder2018game}. We use the unweighted $\infty$-Laplacian here since it connects more directly to the tug-of-war interpretation of the graph $p$-Laplacian, discussed below in this section. It turns out, similar to the continuum $\infty$-Laplace equation, that graph $\infty$-harmonic functions also solve $L^\infty$-type variational problems on graphs (see, for example, \cite{kyng2015algorithms,luxburg2004distance}).

The graph $\infty$-Laplacian was first used for semi-supervised learning in \cite{luxburg2004distance,kyng2015algorithms}, where it is called \emph{Lipschitz learning}. In \cite{calder2019consistency} it was shown rigorously using PDE-continuum limits that solutions of the graph $\infty$-Laplace equation converged to $\infty$-harmonic functions in the continuum, even for finite isolated labeled data points. This established well-posedness of Lipschitz learning with arbitrarily few labeled examples. Subsequent work \cite{bungert2023ratio,bungert2022uniform} established convergence rates of graph $\infty$-harmonic functions to the continuum. As in Section \ref{sec:tug}, specifically \eqref{eq:homo_p}, it is natural to define a $p$-Laplacian operator on the graph that is a combination of the $2$-Laplacian and $\infty$-Laplacian. Here we use the definition
\begin{equation}\label{eq:game_theoretic}
\Lp = \alpha \Lrw + (1-\alpha) \L_\infty,
\end{equation}
where $\alpha = 1/(p-1)$ and $\alpha=1$ when $p=\infty$. This type of graph $p$-Laplacian is often called the \emph{game-theoretic} $p$-Laplacian due to its connection to tug-of-war games, as described in Section \ref{sec:tug}. The $p$-Laplace operator appearing earlier in \eqref{eq:plaplace_graph}, as the necessary conditions for the $p$-Dirichlet energy, is usually called the \emph{variational} $p$-Laplacian, and is a different operator on the graph, even though they agree in the continuum. The game-theoretic $p$-Laplacian on graphs originally appeared in \cite{manfredi2015nonlinear}, and was proposed for semi-supervised learning in \cite{calder2018game}, although the latter work considered a weighted version of the $\infty$-Laplacian. In \cite{calder2018game} it was shown that the game-theoretic $p$-Laplacian is well-posed with arbitrarily few labeled examples when $p>d$\footnote{The definition of $\alpha$ differs by a constant in \cite{calder2018game} so that $p$ has the correct continuum interpretation.}, by showing that the continuum limit was the well-posed continuum $p$-Laplace equation with $p>d$. An interesting and important detail is that the condition $n\epsilon^{p}\ll 1$ is not required by the game-theoretic $p$-Laplacian, due to the strong regularization provided by the $\infty$-Laplacian. 

As in Section \ref{sec:tug}, we can develop a tug-of-war game interpretation for the game-theoretic graph $p$-Laplacian. Any function $u$ on the graph satisfying $\Lp u = 0$ also satisfies, by rearranging \eqref{eq:game_theoretic}, the equation
\begin{equation}\label{eq:graph_tug}
u(x ) = \frac{\alpha}{d_{x }}\sum_{y \in \X} \wij u(y ) + \frac{1-\alpha}{2}\left( \max_{N_x } u+ \min_{N_x } u\right),
\end{equation}
which is the discrete graph version of the continuous identity \eqref{eq:mvp_pharm}, without the $O(\epsilon)$ error term. As we did in Section \ref{sec:tug}, we can define a discrete stochastic process on the graph that produces the martingale property. We define $X_1,X_2,\dots$ so that, given $X_k$, the choice of $X_{k+1}$ is with probability $\alpha$ a random walk step from $X_k$, with probability $\frac{1-\alpha}{2}$ the vertex $y $ maximizing $u(y )$ over neighbors $y $ of $X_k$, and with probability $\frac{1-\alpha}{2}$ the corresponding minimizing vertex. For the same reasons as in Section \ref{sec:tug}, any graph $p$-harmonic function, satisfying $\Lp u(x )=0$ for all $i$, is a martingale when applied to the process $X_k$, in the sense that
\[\E[u(X_{k+1}) \, | \, X_k] = u(X_k).\]
This allows us to apply martingale techniques to study properties of the solution to the game-theoretic $p$-Laplacian on the graph. The variational graph $p$-Laplacian \eqref{eq:plaplace_graph} does not have any such stochastic tug-of-war interpretation.  To our knowledge, no existing works have used the tug-of-war interpretation of the graph $p$-Laplacian to study properties of semi-supervised $p$-Laplacian learning.

Figures \ref{fig:tp_plaplace1} and \ref{fig:tp_plaplace2} show the results of the game-theoretic $p$-Laplacian applied to the toy two labeled point problem, illustrating how the $p$-Laplacian corrects the spike phenomenon in Laplace learning. Note we only need $p>2$ here since $d=2$. Many other approaches have been proposed for correcting the degeneracy in Laplace learning. In \cite{shi2017weighted}, the authors proposed to reweight the edges of the graph that connect to labels more strongly to discourage spike formation, and to use the ordinary graph Laplacian on the reweighted graph. The method is called the Weighted Nonlocal Laplacian (WNLL). Figure \ref{fig:tp_wnll} shows an example of the WNLL on the two-point problem, where we can see that the method essentially just extends the labels to nearest neighbors on the graph. However, it was shown in \cite{calder2020properly} that the WNLL method remains ill-posed at arbitrarily low label rates. In \cite{calder2020properly}, the authors proposed the \emph{properly weighted} graph Laplacian, which reweights the graph in a non-local way that is \emph{singular} near the labeled data, so that the continuum limit is well-posed with arbitrarily few labels. In \cite{zhou2011semi}, it was proposed to use a higher order Laplace equation of the form $\L^m u = 0$ for problems with very few labels. The idea here is that in the continuum, the variational formulation of the equation would involve $u\in H^{m}(\Omega)$, and when $m>\frac{d}{2} $, the Sobolev embedding  $H^m \subset C^{0,\alpha}(\Omega)$ ensures that spikes do not form and the continuum limit is well-posed. Figure \ref{fig:tp_ho} shows the two-point problem with higher order Laplacian regularization. The well-posedness of higher order Laplace learning with finite labeled data in the continuum limit was only very recently addressed in \cite{weihs2023consistency}. Finally, there is recent work on using Poisson equations on graphs for semi-supervised learning in \cite{calder2020poisson,calder2023poisson,miller2023active}, some of which connects to the earlier work on the properly weighted Laplacian \cite{calder2020properly}. In particular, the Poisson Weighted Laplace Learning (PWLL) method from \cite{miller2023active,calder2023poisson} involves reweighting the graph using the solution of a graph Poisson equation, whose singularities are sufficient to invoke the results in \cite{calder2020properly}, though the theory for these methods is currently still under development. Figure \ref{fig:tp_poisson} shows the behavior of PWLL on the two point problem. The volume constrained MBO method based on auction dynamics developed in  \cite{jacobs2018auction} has also proven to be effective at low label rates, though we are not aware of any theory to explain this. Other methods at low label rates include Hamilton-Jacobi equations on graphs \cite{calder2022hamilton} and the \emph{centered kernel method} \cite{mai2018random,mai2021consistent,maicouillet2018random}. \red We also mention here the work on MBO methods\footnote{MBO stands for Merriman-Bence-Osher, who originally proposed threshold dynamics approaches for numerically approximating mean curvature motion \cite{merriman1994motion}.} for graph-based semi-supervised learning \cite{merkurjev2013mbo,hu2013method,garcia2014multiclass,boyd2018simplified,merkurjev2018semi,merkurjev2014diffuse} which consider approximations of the $p=1$ Laplacian for graph-based learning. The MBO methods perform semi-supervised learning on graphs by alternating diffusion and thresholding to label vectors until convergence, which has the effect of finding the partition of the graph with the smallest perimeter that fits the labels correctly.  \nc

\subsection{Consistency versus well-posedness}

\begin{figure}[!t]
\centering
\subfloat[Data]{\includegraphics[width=0.32\textwidth]{figs/fully_preDF.pdf}}
\subfloat[Laplace \cite{zhu2003semi} (77\perc)]{\includegraphics[width=0.32\textwidth]{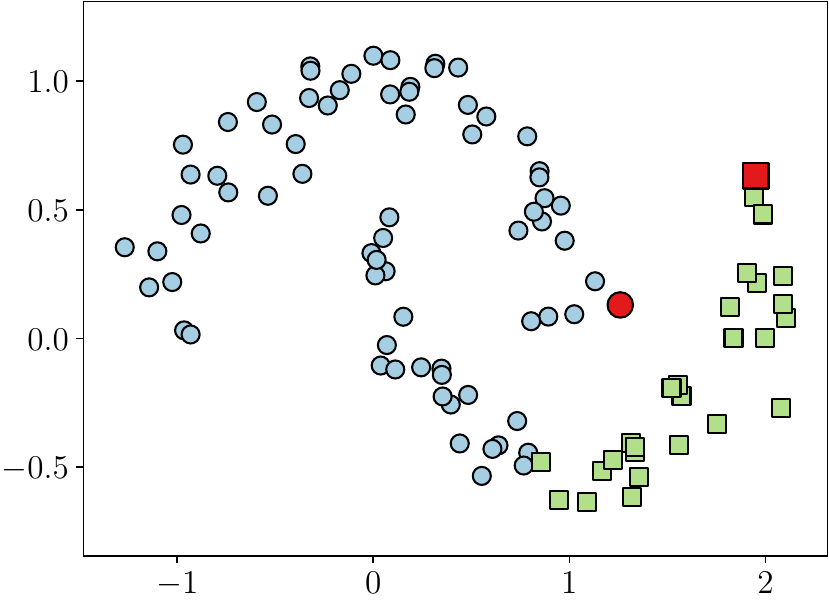}}
\subfloat[$p$-Laplace \cite{flores2022algorithms} (79\perc)]{\includegraphics[width=0.32\textwidth]{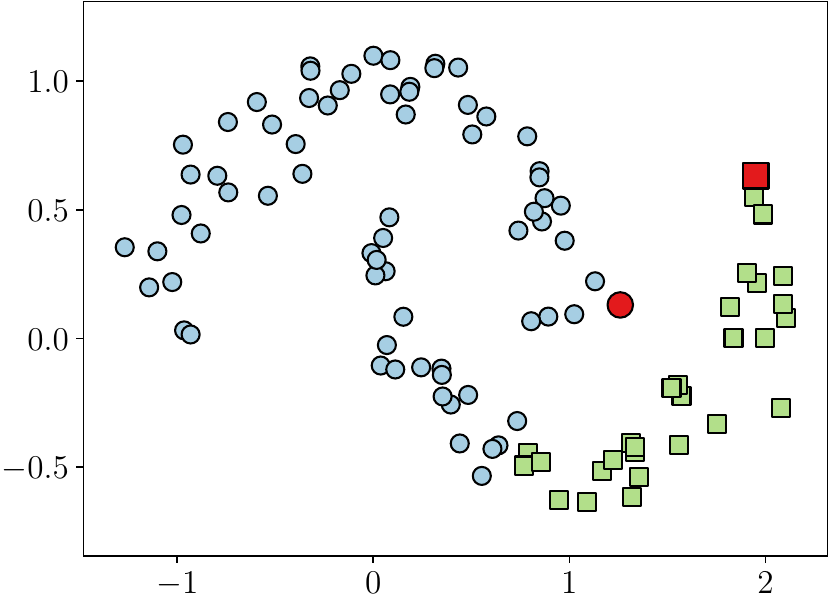}}\\
\subfloat[WNLL \cite{shi2017weighted} (61\perc)]{\includegraphics[width=0.32\textwidth]{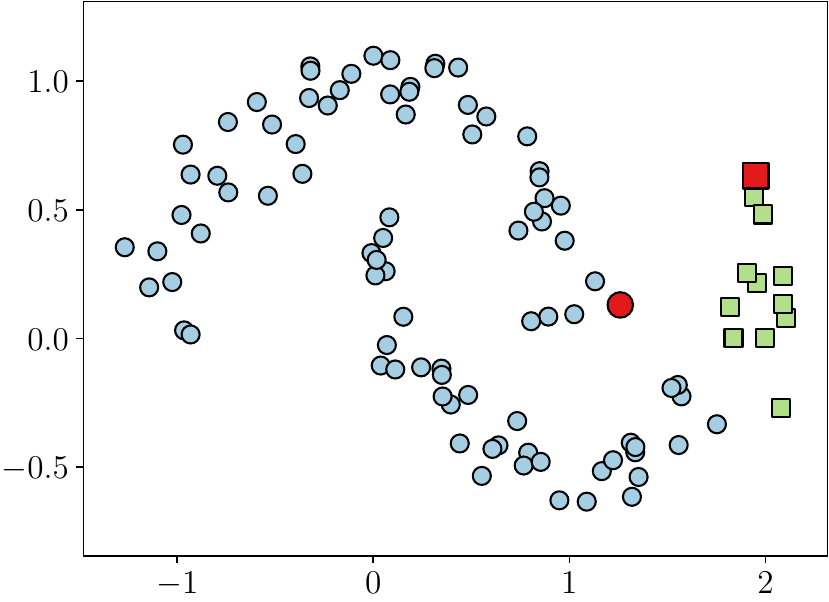}}
\subfloat[High-order Laplace \cite{zhou2011semi} (94\perc)]{\includegraphics[width=0.32\textwidth]{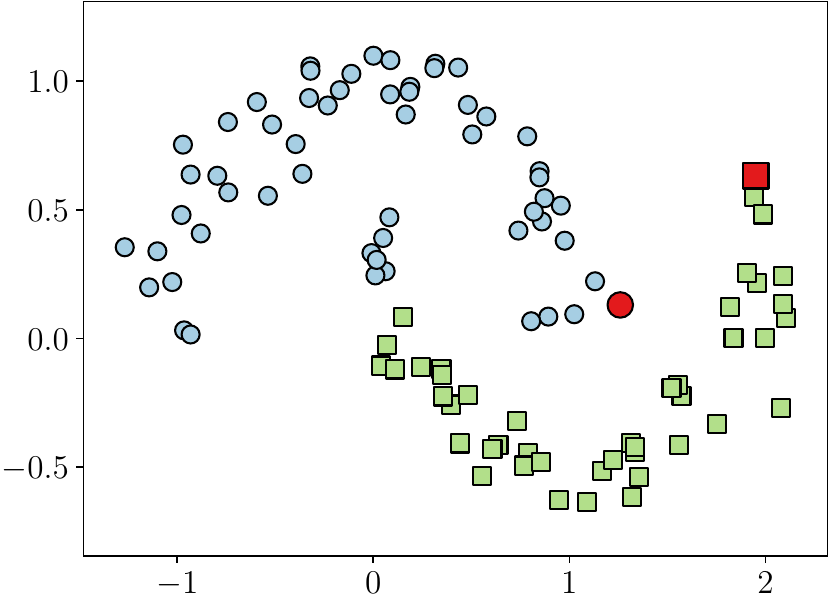}}
\subfloat[Poisson \cite{calder2020poisson} (100\perc)]{\includegraphics[width=0.32\textwidth]{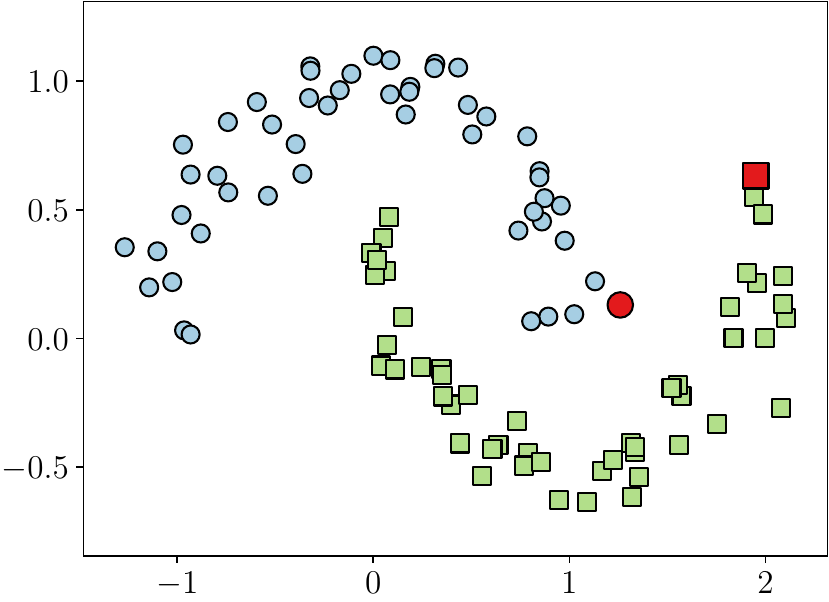}}
\caption{Classification results using different graph-based semi-supervised learning algorithms on the two-moons data set. \red There are two classes, the upper moon and the lower moon, and the label function is binary $g(x)\in \{0,1\}$. \nc The red points are the only labeled data points. The $p$-Laplace method used $p=3$, while high order Laplace learning used order $m=4$, both of which gave the best results over a reasonable search.}
\label{fig:ssl_comp}
\end{figure}

Much of the work described above is aimed at establishing continuum PDE or variational descriptions of graph-based machine learning algorithms, which can, for example, prove that a method like $p$-Laplace learning provides a well-posed method for propagating labels on graphs that does not develop spikes or other degeneracies. Many other works have also considered the continuum limit of the graph Laplacian, its spectrum, and regularity properties \cite{calder2022improved, dunlop2020large,garcia2020error,hein2007graph,hein2005graphs,calder2022Lip}. However, this abundance of work fails to address the problem of whether the machine learning methods work well or not; that is, they say little or nothing about whether the predicted label is correct! This is the question of \emph{consistency}, which is more important and difficult, and thus less often studied. 

As an example, in Figure \ref{fig:ssl_comp} we show the results of graph-based semi-supervised learning applied to the same two-moons example as in Figure \ref{fig:ssl_vs_fully}. The $p$-Laplace and higher order Laplace methods are the only methods with rigorous well-posedness results with very few labels, and Poisson learning is expected to enjoy such results, which is the subject of a forthcoming paper by the first author. Laplace learning and WNLL are ill-posed in the low-label regime. The best performing method is Poisson learning \cite{calder2020poisson}, followed closely by high-order Laplace, which illustrates that well-posedness does not always go hand in hand with good consistency properties, or at least that the relation between consistency and having a well-posedness continuum limit is not clear or well-understood. 

To address the consistency problem, one has to make an assumption on the underlying cluster structure of the graph, or the labeling function. Then the goal is to prove that the machine learning algorithm can identify the clusters or correct labels, under reasonable assumptions on the model and parameters. This kind of consistency analysis was carried out in the context of spectral clustering in \cite{hoffmann2022spectral}, in which the probability density $\rho$ was assumed to be highly concentrated on disjoint clusters. Spectral clustering uses the eigenvectors of the graph Laplacian for clustering \cite{von_luxburg_tutorial_2007} and are closely related to spectral methods in data science \cite{donoho2003hessian,belkin2003laplacian,coifman2006diffusion}. Another related work is \cite{trillos2021geometric}, which assumes a well-separated mixture model for the data and shows that the spectral embedding reflects the clusters in the data \cite{trillos2021geometric}. Aside from these examples, there is a lack of consistency results for many graph-based semi-supervised learning methods, which should be viewed as an opportunity to utilize powerful PDE-continuum limit tools to establish important and interesting results about machine learning algorithms. We begin such a study in this paper.

\red

\subsection{Deep semi-supervised learning}

There is a growing body of work on deep semi-supervised learning, which includes many methods based on graph neural networks; see \cite{yang2022survey} for a survey. The data science applications are somewhat different here; in addition to a graph, most problem formulations assume that each \emph{node} $x\in \X$ in the graph has an associated feature vector $v(x)\in \R^k$. An example application would be a citation network, where each node in the graph is an academic paper, and the edges represent citations between papers. The feature vector $v(x)$ for paper $x$ may contain summary statistics of key words appearing in the abstract of the paper. An example task is to classify the subject of each paper, using both the graph structure and the feature vectors, which in general provide complimentary information.\footnote{A standard test data set called \emph{PubMed} \cite{yang2016revisiting} uses exactly this setup.}

Let us briefly describe how some graph neural networks work for semi-supervised learning. A standard feed-forward neural network, or multi-layer perceptron, iteratively composes linear functions with nonlinear activation function, $u_{k+1} = \sigma(\Theta_k u_k + b_k)$, where $u_k\in \R^{n_k}$ is the input to the $k^{\rm th}$ layer, $\Theta_k\in \R^{n_{k+1}\times n_k}$ is the weight matrix for the $k^{\rm th}$ layer, and $b_k\in \R^{n_{k+1}}$ the bias vector. A common choice for the activation $\sigma$ is the rectified linear unit (ReLU) $\sigma(t)=\max\{t,0\}$, but other choices are possible. The input to the network is $u_1$, and the output is $u_L$ after $L$ layers. 

Graph neural networks incorporate the graph information into the neural network architecture.  A very popular approach is the graph convolutional neural network (GCN) proposed in \cite{kipf2016semi}, which combines the standard feed forward neural network with the graph structure as follows:
\begin{equation}\label{eq:gcn_layer}
u_{k+1}(x) = \sigma\left( \frac{1}{d_x}\sum_{y\in \X}w_{xy}(\Theta_k u_k(y) + b_k) \right),
\end{equation}
where $u_k:\X\to \R^{n_k}$ and $\Theta_k,b_k$ are as before.\footnote{To be precise, the method in \cite{kipf2016semi} divides by $\sqrt{d_xd_y}$ in the sum, instead of $d_x$, and adds self-loops to the graph.} The inputs are the node feature vectors $u_1(x)=v(x)$, and the output after $L$ layers is, after thresholding, the network's prediction of the label for each node in the graph. This output is fed into a loss function that measures how well the network fits some training data, and the weights $\Theta_k$ and biases $b_k$ are tuned by minimizing the loss function with gradient descent.

We note that the GCN layer \eqref{eq:gcn_layer} is essentially a combination of a feed-forward neural network with averaging over neighborhoods on the graph. These types of networks are also called \emph{message passing} graph neural networks. The reader can compare this to label propagation, see \eqref{eq:mvp_iterate}, and notice that when $\sigma$ is the identity $\sigma(t)=t$, $\Theta_k=I$, and $b_k=0$, this appears to be exactly label propagation. The difference is that with graph neural networks, the operation acts on the \emph{feature vectors} of the nodes, while label propagation acts directly on the labels, since feature vectors are not available.  The graph convolutional layer \eqref{eq:gcn_layer} can be viewed as a special case of a more general class of of graph neural networks based on a spectral definition of convolution  \cite{defferrard2016convolutional}. We also mention an important variant called graph attention networks \cite{velivckovic2017graph}, which allows the weights in \eqref{eq:gcn_layer} to be \emph{learned} using an \emph{attention mechanism}, in a similar way that transformers utilize the attention mechanism in large language models \cite{vaswani2017attention}.

The analysis of graph neural network models like the GCN given in \eqref{eq:gcn_layer} would be substantially different from the analysis in this paper. The important problems concern how the \emph{learned} weights $\Theta_k$ and biases $b_k$ interact with the graph structure during training, especially for deeper networks, and whether the network relies more on the graph structure or the node features for classification. Although there is ostensibly a connection to label propagation \eqref{eq:mvp_iterate}, this is largely superficial, since GCNs \eqref{eq:gcn_layer} usually only have a handful of layers, say $2$ or $3$, while label propagation is often run for hundreds of steps or more, in order for the method to \emph{converge} to the solution of Laplace's equation. Another difference with the setting of this paper is that many recent graph neural network approaches are aimed at tackling heterophily in graphs, in which adjacent nodes may very often belong to different classes \cite{yan2022two,zheng2022graph,li2022finding,zhu2020beyond}. In this case the semi-supervised smoothness assumption does not hold, and the question is how to combine information from both the feature vectors and graph structure, when neither on their own is highly informative. 

There are, nevertheless, some settings where the analysis in this paper may be relevant in deep graph neural networks. Many recent works have shown that one can obtain effective methods by decoupling the graph diffusion or averaging in \eqref{eq:gcn_layer} from the neural network weights $\Theta_k$ and biases $b_k$ \cite{sellars2022laplacenet,gasteiger2018predict,dong2021equivalence,azad2022learning,huang2020combining}. One way to do this, proposed in \cite{gasteiger2018predict} is to first run the feature vectors $v(x)$ for each node through a standard feed-forward neural network to make label predictions for each node $x$, and then to diffuse those predictions over the graph before feeding them into a loss function. In the case of \cite{gasteiger2018predict} the authors used the PageRank algorithm \cite{gleich2015pagerank}, which is essentially a graph reaction-diffusion equation \cite{yuan2022continuum}, and the label predictions entered through the source term. The particular choice of PageRank is not essential, and other works have shown that various forms of graph diffusion can easily substitute \cite{sellars2022laplacenet,azad2022learning}. Thus, improvements in graph-based semi-supervised learning algorithms, especially theoretical guarantees, can have a direct impact on deep semi-supervised learning methods that decouple the graph from the neural network.

\nc


%% file: sections/results.tex

\section{Consistency results for the \texorpdfstring{$p$}{p}-Laplacian}
\label{sec:results}

We present here a general framework for proving consistency results for graph-based semi-supervised learning with the game-theoretic $p$-Laplacian. The framework is based on the stochastic tug-of-war interpretation of the $p$-Laplacian on a graph, that was introduced in Section \ref{sec:ssl}. We illustrate how to use the framework to prove some preliminary consistency results on geometric graphs and stochastic block model graphs, and then highlight some open research questions for future work. 

As in Section \ref{sec:ssl} let $\X=\{x_1,\dots,x_n\}$ denote our data points, which are the vertices of the graph, and let $W=(w_{xy})_{x,y\in \X}$ be a symmetric (i.e., $w_{xy}=w_{yx}$) weight matrix which encodes the graph structure. We assume the graph is connected. We let $d_x=\sum_{y\in \X}w_{xy}$ be the degree of vertex $x\in \X$, and we recall that $N_x\subset \X$ denotes the graph neighbors of $x$, as defined in \eqref{eq:neigh}. Let $\ell^2(\X)$ denote the space of functions $u:\X\to \R$, and define the game-theoretic $p$-Laplacian on the graph $G$ is the operator $\Lp:\ell^2(\X)\to \ell^2(\X)$ defined in \eqref{eq:game_theoretic} in Section \ref{sec:ssl}, where $\alpha = 1/(p-1)$ and $p\geq 2$. Let us denote by $\Gamma\subset \X$ the labeled data set and $g:\Gamma\to\R$ the labels. The game-theoretic $p$-Laplacian regularized semi-supervised learning problem is given by the boundary value problem
\begin{equation}\label{eq:ssl_pLap}
\left\{\begin{aligned}
\Lp u &= 0,&&\text{in }\X\setminus \Gamma\\ 
u &=g,&&\text{on }\Gamma.
\end{aligned}\right.
\end{equation}
Since the graph is connected, the $p$-Laplace equation \eqref{eq:ssl_pLap} can be shown to have a unique solution via the Perron method \cite{calder2018game}.

In this section we address the question of \emph{consistency} of $p$-Laplacian regularization. That is, supposing that $g$ is the restriction to $\Gamma$ of a label function on the whole graph $g:\X\to \R$ that is in some sense smooth with respect to the geometry of the graph, we aim to give conditions under which the solution $u$ of \eqref{eq:ssl_pLap} is close to $g$ at the unlabeled vertices $\X\setminus \Gamma$. In other words, consistency asks the question: under what conditions on the underlying label function $g$, the graph structure $W$, and the label set $\Gamma$ does the solution of $p$-Laplacian regularization make the correct label predictions? \red In this paper, we consider the \emph{noise-free} setting, where we observe the true labels without any noise or corruption. It would be interesting to consider the setting of noisy labels in future work; we discuss this more in Section \ref{sec:conc}. \nc

In consistency of graph-based learning, it is essential that the underlying label function $g$ is in some way regular or smooth with respect to the graph structure. In other words, there should be some degree of correlation between the labels of vertices and their nearby neighbors in the graph topology. Otherwise graph-based learning is not useful and would not be expected to yield better results than fully supervised learning that ignores the graph structure. In this section, we present results for two different types of graph structures. In Section \ref{sec:e-graph} we present results for geometric graphs, where the vertices are sampled from a Euclidean domain and nearby points are connected by edges. In this case we assume that $g$ has some degree of regularity with respect to the underlying Euclidean geometry. In Section \ref{sec:stochasticblockmodels} we present results for classification on stochastic block model graphs, where there is no \red continuous geometric structure (SBM graphs can be thought of as having a discrete geometric structure, but this is not directly related to continuous PDEs). \nc In this case the label function is concentrated on the blocks, and only changes between blocks where there are relatively few edges. 

\subsection{Results on general graphs}

Before presenting our main results in Section \ref{sec:e-graph} and Section \ref{sec:stochasticblockmodels}, we present our general stochastic tug-of-war framework on general graphs in Section \ref{sec:RGG}, which amounts to selecting sub-optimal strategies in the tug-of-war game to produce tractable upper and lower bounds on the solution $u$ of the $p$-Laplace equation \eqref{eq:ssl_pLap}. It is interesting to remark that we do not explicitly use the symmetry assumption $w_{xy}=w_{yx}$ in our results, and our results hold with minor modifications to the quantities involved for directed (i.e., nonsymmetric) graphs.




\subsubsection{The tug-of-war lemma}\label{sec:RGG}

In this section we present our main tug-of-war result (Lemma \ref{lem:main_tug}), which uses the stochastic tug-of-war game and martingale techniques to bound the difference $u(x)-g(x)$, where $u$ solves \eqref{eq:ssl_pLap}. These techniques are later applied to geometric graphs in Section \ref{sec:e-graph}, and to stochastic block model (SBM) graphs in Section \ref{sec:stochasticblockmodels}.    The tug-of-war game for the $p$-Laplacian is based on the following dynamic programming principle associated with the $p$-Laplacian, which was discussed and established in Section \ref{sec:ssl}. 
\begin{proposition}[Dynamic Programming Principle]\label{prop:dpp}
If $u\in \ell^2(\X)$ and $x\in \X$ such that  $\Lp u(x)=0$ then
\begin{equation}\label{eq:dpp}
u(x) = \frac{\alpha}{d_x}\sum_{y\in \X}w_{xy}u(y) +\frac{1-\alpha}{2}\left(\min_{N_x}u+ \max_{N_x}u \right),
\end{equation}
where we recall that $\alpha = 1/(p-1)$ and $p\geq 2$.
\end{proposition}
We call equation \eqref{eq:dpp} a dynamic programming principle (DPP), because it expresses $u$ at a point via $u$ at nearby points, and is closely related to the DPPs that appear in optimal control problems (see, e.g., \cite{bardi1997optimal}). In fact, the DPP arises from the stochastic two player tug-of-war game described in Section \ref{sec:ssl}. The game is played between two players Paul and Carol. The game ends when the token of the game arrives at a point $x\in \Gamma$, and Paul pays Carol $g(x)$. Thus, Paul wants the game to end where $g$ is smallest, while Carol wants the opposite. At each step of the game, with probability $\alpha$ the token moves from its current position $x$ to a neighbor $y\in N_x$ via a single step of a random walk on the graph (i.e., the next position is chosen randomly according to the distribution $\P(y)=w_{xy}/d_x$).  With probability $(1-\alpha)/2$ Paul chooses the next position of the token, and with probability $(1-\alpha)/2$ Carol chooses the next position. We will use this interpretation in the proofs of our main results, without explicitly writing down the game or the appropriate spaces of strategies.



In order to use the tug-of-war game, we define a stochastic process that corresponds to a \emph{sub-optimal} strategy for Paul.  We recall that $\alpha=1/(p-1)$ and pick any $p\geq2$. Let $u\in \ell^2(\X)$ and $x\in \X$, and consider the following stochastic process defined on $\X$ associated with $u$ and an initial point $x\in \X$. We define the sequence of random variables $X_0,X_1,X_2$ as follows. We first set $X_0=x$. Then, conditioned on $X_i$, we choose $X_{i+1}$ as follows: first, if $X_i\in \Gamma$, then we set $X_{i+1}=X_i$, so the process gets \emph{stuck} when it hits $\Gamma$. If $X_i\not\in \Gamma$, then we choose $X_{i+1}$ by the following:
\begin{enumerate}
\item With probability $\alpha$ we take an independent random walk step, that is 
\[\P(X_{i+1} = y \, | \, X_i = x) = \frac{w_{xy}}{d_X}.\]
\item With probability $\tfrac{1}{2}(1-\alpha)$ we choose $\displaystyle X_{i+1}\in \argmax_{y\in N_{X_i}}u$.
\red \item With probability $\tfrac{1}{2}(1-\alpha)$ we choose $\displaystyle X_{i+1}\in \argmin_{y\in N_{X_i}\cap \Gamma} g$, when $N_{x_i}\cap \Gamma \neq \varnothing$, and $\displaystyle X_{i+1}\in \argmin_{y\in N_{X_i}} u$, when $N_{x_i}\cap \Gamma=\varnothing$.
\end{enumerate}
All the probabilistic ingredients are chosen independently. When there are multiple choices for $X_{i+1}$ in steps 2 and 3, the particular choice is irrelevant to the analysis, and any concrete choice will do (e.g., say, choose the vertex $x_i$ with the smallest index). The stochastic process $X_0,X_1,X_2,\dots,$ defined above can be interpreted as a realization of the two-player game where Carol plays optimally and our choice of strategy for Paul is to move to $\Gamma$ as soon as possible, \red which is the \emph{sub-optimal} part of his strategy.  \nc



We now establish the sub-martingale property corresponding to our stochastic process, \red which is a key ingredient in our analysis. 
\begin{lemma}\label{lem:submart}
Let $u$ be the solution to \eqref{eq:ssl_pLap}. Then the random variables $Z_i=u(X_i)$ form a sub-martingale with respect to $X_i$, that is $\E[Z_{i+1}\, |\, X_i] \geq Z_i$. 
\end{lemma}
\begin{proof}
Conditioned on $X_i\in \Gamma$ we have $\E[u(X_{i+1}) \, | \, \F_i]=u(X_i)$, since $X_{i+1}=X_i$.
Conditioned on $X_i\not\in \Gamma$ we further condition on steps 1--3 in the definition of $X_{i+1}$ and use $u=g$ on $\Gamma$ to obtain 
\begin{equation}\label{eq:step}
\E[u(X_{i+1}) \, |\, \F_i]=\frac{\alpha}{d_{X_i}}\sum_{y\in \X}w_{X_i,y}u(y) + \frac{1-\alpha}{2}\left(\max_{N_{X_i}}u + \min_{V}u(y)\right),
\end{equation}
where $V=N_{x_i}$ when $N_{X_i}\cap \Gamma=\varnothing$ and $V = N_{X_i}\cap \Gamma$ when $N_{X_i}\cap \Gamma\neq \varnothing$.  Using Proposition \ref{prop:dpp}  we have
\[\E[u(X_{i+1}) \, |\, \F_i]\geq\frac{\alpha}{d_{X_i}}\sum_{y\in \X}w_{X_i,y}u(y) + \frac{1-\alpha}{2}\left(\max_{N_{X_i}}u + \min_{N_{X_i}}u\right) = u(X_i),\]
which completes the proof.
\end{proof}
\nc

\red We now give the main tug-of-war lemma, which uses the sub-martingale property to bound $u(x)$.
\begin{lemma}\label{lem:main_tug}
Let $u$ be the solution to \eqref{eq:ssl_pLap}. Define the stopping time
\begin{equation}\label{eq:stopping}
\tau = \inf\{i\geq 0\, : \, X_i\in \Gamma\}.
\end{equation}
Then for any $x\in \X$ it holds that
\begin{equation}\label{eq:stopping_bound}
u(x)  \leq  \E[g(X_\tau) \, | \, X_0=x]. 
\end{equation}
\end{lemma}
\begin{proof}
By Lemma \ref{lem:submart}, $Z_i=u(X_i)$ is a sub-martingale, and so by Doob's Optional Stopping Theorem for sub-martingales, see e.g. \cite{williams1991probability}, we have
\begin{equation}\label{Doob}
u(x) =\E[Z_0 \, | \, X_0=x] \leq \E[Z_\tau\, | \, X_0=x] = \E[u(X_\tau)\, | \, X_0=x] =\E[g(X_\tau)\, | \, X_0=x],
\end{equation}
since $X_\tau \in \Gamma$. 
\end{proof}

Let us give a brief preview of how Lemma \ref{lem:main_tug} is used. We can subtract $g(x)$ from both sides of \eqref{eq:stopping_bound} to obtain
\[ |u(x)-g(x)|  \leq  \E[|g(X_\tau) - g(X_0)| \, | \, X_0=x]. \]
If $g$ is Lipschitz continuous, which is a reasonable assumption on a \emph{continuous geometric} graph, then the right hand side is $O(|X_\tau - X_0|)$, and so bounding $|u(x)-g(x)|$ amounts to estimating the stopping time $\tau$. In this context, it is not especially important that Paul chooses the point in $N_{x_i}\cap \Gamma$ that minimizes $g$, and the game could be modified so Paul picks, say, a random boundary point. However, for SBM graphs, the arguments are different, and the key is to allow the \emph{noise} to exit the game, since the noise sees the block structure. In the SBM setting, it is important for Paul to choose the minimizer of $g$ over $N_{x_i}\cap \Gamma$ since $g$ cannot be Lipschitz in any sense on an SBM graph. The interested reader can skip to Section \ref{sec:stochasticblockmodels} for the SBM analysis, which is independent of much of the geometric results below. 

Throughout the rest of the paper, we make the strong, but simplifying, assumption that every vertex $x\in \X$ has a neighbor $y\in N_x$ with $y\in \Gamma$, meaning: 
\begin{itemize}
\item[\as1] We assume that for every $x\in \X$
\begin{equation*}
\Gamma\cap N_x \neq \varnothing.
\end{equation*}
\end{itemize}
Assumption \as1 is a simplifying assumption that makes much of the analysis tractable. Essentially, it allows us to estimate the stopping time $\tau$ defined in \eqref{eq:stopping} in many different settings by allowing Paul to end the game quickly. Depending on the graph model, assumption \as1 is a fairly strong assumption on both the \emph{label rate} and the \emph{graph structure}. In Sections \ref{sec:e-graph} and \ref{sec:stochasticblockmodels} we will give conditions under which \as1 holds for geometric and SBM graphs. The results in this paper are preliminary, and we discuss the importance of relaxing assumption \as1  in Section \ref{sec:conc}. This seems to us to be a nontrivial task that will require fine estimates on martingale stopping times.  \nc

\subsubsection{General consistency results}\label{main_r}

Here, we use the martingale lemma, Lemma \ref{lem:main_tug}, to establish some general consistency results, \red which will rely on a Lipschitz or bounded derivative condition on the label function $g$. Thus, for $g:\X\to \R$, we need to construct a notion of gradient that is consistent with the gradient in the continuum when working with geometric graphs with a length scale $\eps$. Given a graph described by a weight matrix $W$, we let $\diam(W)$ denote the \emph{diameter} of the graph, which is defined as the largest number of hops required to get from any node to any other node in the graph. When the graph is a random geometric graph with bandwidth $\eps>0$, then each hop travels at most distance $\eps$, and so we would expect to have $\diam(W)= O(\eps^{-1})$. This motivates the definition of the \emph{graph length scale}.
\begin{definition}\label{def:graph_length_scale}
Given a weighted graph $W$, the \emph{graph length scale} $\eps_W$ is defined as 
\[\eps_W = \frac{1}{\diam(W)}.\]
\end{definition}

We can now define the gradient of $g$ as 
\[\nabla g(x,y) = \frac{g(x)-g(y)}{\eps_W},\]
for any $x,y\in \X$ connected by an edge, so $w_{xy}>0$.  We also define
\[\|g\|_\infty = \max_{x\in \X}|g(x)|,\]
and
$$\|\nabla g\|_\infty =\max_{x,y \in X}\left\{\left| \nabla g(x,y)\right|\mathbbm{1}_{w_{xy}>0} \right\}.$$
This is to say, the norm $\|\nabla g\|_\infty$ is the maximal absolute difference of $g$ across all edges in the graph divided by the graph length scale $\eps_W$.

\nc 

Define the smallest \emph{weighted} ratio of labeled neighbours to neighbours to be
\begin{equation}\label{delta1}
\delta:=\min_{x\in X}\frac{\Sigma_{y\in \X}w_{xy}\mathbbm{1}_{y\in \Gamma}}{\Sigma_{z\in \X}w_{xz}}.
\end{equation}
Note that if \as1 holds, then $\delta>0$. 
 
The following theorem gives our first consistency result on general graphs. 
\begin{theorem}\label{th1}
Assume \as1 holds. 
Then, for any $x\in\X,$ 
\red
\begin{equation}\label{u_minus_g}
|u(x) - g(x)|\leq \left(\frac{2\log(\|g\|_\infty\|\nabla g\|_\infty^{-1}\eps_W^{-1})}{1-\alpha +2\alpha \delta}+2\right)\|\nabla g \|_\infty\eps_W.
\end{equation}
\nc
\end{theorem}

\begin{proof}
Throughout the proof let us write $\zeta = \|\nabla g\|_\infty\red\eps_W\nc$ for simplicity. Define the stopping time $\tau$ as in \eqref{eq:stopping}. By Lemma \ref{lem:main_tug}, we have that 
$$u(x) - g(x) \leq \E[g(X_\tau) -g(x) \, | \, X_0=x].$$
Let us estimate the right hand side above by fixing some $k$ (which we will choose shortly) and conditioning on $\tau>k$ or $\tau \leq k$. For simplicity of notation we will drop the conditioning $X_0=x$ below.  Then we have
\begin{align}\label{ine:1}
u(x) - g(x) \red  \leq \nc \E[g(X_\tau) - g(x) \, | \, \tau \leq k]\P(\tau \leq k) + \E[g(X_\tau) - g(x) \, | \, \tau > k]\P(\tau>k).
\end{align}
First, we estimate the probability $\P(\tau > k).$
Since each vertex has a labeled neighbour (by Assumption \as1), each step has probability at least $\tfrac{1}{2}(1-\alpha)+\alpha\delta$ of hitting $\Gamma$ and exiting, so a general estimate is
\begin{equation}\label{prob_exit}
\P(\tau > k) \leq \left[ 1 - \tfrac{1}{2}(1-\alpha)-\alpha\delta \right]^k \leq e^{-\frac{k}{2}(1-\alpha)-\alpha\delta k}.
\end{equation}

For any $k$, we can use a telescoping series to obtain
\[g(X_k) - g(x) = \sum_{j=0}^{k-1} (g(X_{j+1}) - g(X_j)) \leq \sum_{j=0}^{k-1}\|\nabla g\|_\infty\red \eps_W\nc = k\zeta.\]
It follows that
\begin{equation}\label{ine:11}
\E[g(X_\tau) - g(x) \, | \, \tau \leq k]  \leq k\zeta.
\end{equation}
Substituting \eqref{prob_exit} and \eqref{ine:11} in \eqref{ine:1}, we have that 
\begin{align}\label{diff_ug}
u(x) - g(x)&\leq k\zeta \P(\tau \leq k) + \E[g(X_\tau) - g(x) \, | \, \tau > k]e^{-\frac{k}{2}(1-\alpha)-\alpha\delta k}\\ \nonumber
& \leq k\zeta + 2\|g\|_\infty e^{-\frac{k}{2}(1-\alpha)-\alpha\delta k}.
\end{align}
%
We want to balance the two terms in the right hand side of \eqref{diff_ug}, so we choose $k$ so that
$$\zeta =\|g\|_\infty e^{- \frac{(1-\alpha)k}{2}-\alpha\delta k},$$ 
or equivalently  
$$k= \frac{2\log( \|g\|_\infty/\zeta)}{1-\alpha+2\alpha\delta}.$$
Therefore we have
\begin{align}\label{ine:2}
u(x) - g(x) 
&\leq \zeta(k +2) = \left(\frac{2\log( \|g\|_\infty/\zeta)}{1-\alpha+2\alpha\delta}+2\right)\zeta.
\end{align}
To estimate $g-u$, we use a similar argument, where we replace $u$ by $-u$ and $g$ by $-g$. This concludes the proof.
\end{proof}

It turns out we can use the same argument to establish a gradient estimate on the solution $u$ of the $p$-Laplace equation, \red which shows that the gradient of $u$ is controlled by the gradient of $g$. \nc
\begin{theorem}\label{th3}
Assume \as1 holds, and let $x$ and $y$ be two neighbors on the graph. Then, 
\red$$\|\nabla u\|_\infty\leq \left(\frac{4\log( \|g\|_\infty\|\nabla g\|_\infty^{-1}\eps_W^{-1})}{1-\alpha+2\alpha\delta}+5\right)\|\nabla g\|_\infty.$$\nc
\end{theorem}
\begin{proof}
We apply the triangle inequality
$$|u(x)-u(y)| \leq |u(x)-g(x)|+|u(y)-g(y)| + |g(x)-g(y)| $$
and use the bound in Theorem \ref{th1} twice, noting that $g(x)-g(y)=\nabla g(x,y)\red\eps_W\nc$.
\end{proof}

\subsubsection{Vertex classification consistency results}

We now turn to vertex classification results, where the label function $g:\X\to \R$ takes constant values, and hence it will have sharp transitions across edges in the graph. Thus, $\|\nabla g\|_\infty$ may not be small, and so Theorems \ref{th1} and \ref{th3} are not applicable.  We will focus on binary classification, where every point $x\in \X$ has a label zero or a label one, so $g:\X\to \{0,1\}$. To use the $p$-Laplace equation \eqref{eq:ssl_pLap}, we solve the equation with the binary values for the boundary condition $g:\Gamma \to \{0,1\}$, and threshold the solution $u(x)$ at $\frac{1}{2}$ to obtain a binary label prediction for each vertex $x\in \X\setminus \Gamma$.  Thus, we make the following definition.
\begin{definition}
Given $g:\X\to \{0,1\}$, a point $x \in \X$ is classified \textit{correctly} if $g(x)=\one_{u(x) \geq \frac{1}{2}}$, where $u$ is the solution to \eqref{eq:ssl_pLap}.
\end{definition}


In this section, we focus on identifying which points in the graph are classified correctly. This requires a few definitions.
\begin{definition}
A path on the graph is a sequence of edges which joins a sequence of vertices.
The length of a path is the number of edges the path consists of.
\end{definition}

\begin{definition}
The distance $\dist(x,y)$ between two vertices $x,y\in \X$ is the length of the shortest path connecting the vertices. The distance between a vertex $x$ and a set of vertices $A\subset \X$, denoted $\dist(x,A)$, is the smallest distance between $x$ and any of the vertices $y\in A$.
\end{definition}

For $i=0,1$ we define
\begin{equation}\label{eq:xi}
\X_i = \{x \in \X : g(x)=i \}
\end{equation}
and note that $\X=\X_0\cup\X_1$. For any integer $m \geq 0$ and $i=0,1$, we define 
$$A_m^i = \{x \in \X_i: \dist(x,\X_{1-i}) \leq m \},$$
and note that $A_m= A_m^0 \cup A_m^1$. That is to say, $A_m^0$ is the set of points from class $0$ that are at most $m$ steps away from the set of class $1$ points, and vice-versa for $A_m^1$. This means that when $m$ is small, the set $A_m$ contains points that are sufficiently close to the boundary between the two classes that we may expect data points in $A_m$ to be misclassified. 

Observe that for any $x\in\X$ there is a positive probability of at least $\tfrac{1-\alpha}{2} + \alpha \delta=\tfrac{p-2+2\delta}{2(p-1)}$ of hitting the boundary $\Gamma.$ Even when $p=2,$ this probability is positive, since it is exactly $\delta>0$. We use this to establish the classification bound below, which shows that whenever a vertex $x$ is sufficiently far, in terms of graph distance, away from the true decision boundary, it will be classified correctly.
\begin{theorem}\label{th2}
Assume \as1. Let  $\k$ be the smallest integer strictly larger than 
\begin{equation}\label{k_defn}
\frac{2\log (2)}{1-\alpha+2\alpha \delta}= \frac{2(p-1)\log(2)}{p-2+2\delta}.
\end{equation}
Then every $x\not\in A_\k$ is classified correctly. 
\end{theorem}
\begin{proof}	
Let $x\in \X\setminus A_\k$ and assume $X_0=x$. Define the stopping time $\tau$ as in \eqref{eq:stopping}. For any point $x\in\X$ we have that 
\begin{equation}\label{ine:3}
u(x) - g(x)\leq \E[g(X_\tau) - g(x) \, | \, \tau \leq \k]\P(\tau \leq \k) + \E[g(X_\tau) - g(x) \, | \, \tau > \k]\P(\tau>\k).
\end{equation}
We want to classify $x$ correctly whether it belongs to $\X_0$ or $\X_1.$ Recall that the labels of vertices  in  $\X_0$ are zeros and in $\X_1$ are ones. 

Since $x \notin A_{\k}$, the labels of vertices reached by the stopping time $\tau \leq {\k}$ can only be labels from the same-class, so $g(X_\tau)= g(x).$ Thus, when  $\tau \leq {\k}$ the first term on the right hand side of \eqref{ine:3} is zero.
Then, using that $g$ is between zero and one, and that equation \eqref{prob_exit} holds, we obtain 
\begin{equation}\label{ine:4} 
u(x) - g(x)\leq  \E[g(X_\tau) - g(x) \, | \, \tau > {\k}]\P(\tau>{\k}) \leq  \P(\tau > {\k})\leq e^{-\frac{{\k}}{2}(1-\alpha)-\alpha\delta {\k}}
\end{equation}
The analogous estimate on $g-u$ is similarly obtained, and so we have 
$$|u(x)-g(x)|\leq e^{-\frac{{\k}}{2}(1-\alpha)-\alpha\delta {\k}}.$$
For \textit{correct} classification of $x$, we need precisely $|u(x) - g(x)|< \frac{1}{2},$  which is equivalent to
$$ e^{-\frac{{\k}}{2}(1-\alpha)-\alpha\delta {\k}} <\frac{1}{2},$$
and is satisfied by the choice of $\k$ in \eqref{k_defn}

\end{proof}

\subsection{Geometric graphs}\label{sec:e-graph}

In this section we specialize our results to geometric graphs, which include the very commonly used random geometric $\eps$-graphs and $k$-nearest neighbor ($k$-NN) graphs. The purpose of this section is to provide conditions under which  $\eps$-graphs and $k$-NN graphs, along with other graphs with geometric $\eps$-scaling, fulfill the general assumption \as1 in Section \ref{sec:RGG}, so that we can apply the theory developed in Section \ref{sec:RGG} to these graphs.

\subsubsection{Main assumptions}
In order to construct $\eps$-graphs and $k$-NN graphs and construct the labeled data set $\Gamma$, we need the following assumptions. Note that assumptions \as2,\as3,\as4 also appear in \cite{calder2023rates}.

\begin{itemize}
\item[\as2] Let $d\geq 2$ and let $\Omega \subset \R^d$ be an open, bounded, and connected domain with a Lipschitz boundary. Let $x_1, x_2, \dots, x_n$ be \textit{i.i.d.} with continuous density $\rho$ satisfying $$0 < \rho_{\min} := \inf_{x\in \Omega}\rho(x) \leq \sup_{x \in \Omega}p(x) =: \rho_{\max}< \infty.$$ We define $\X = \{x_1, x_2, \dots, x_n \}.$
\item[\as3] We construct $\Gamma$ in the following way. Let $g \in \Lip(\bar{\Omega})$, where $\Omega$ is defined as in \as2. Let $\beta\in (0,1)$. For each $x\in \X$, define an independent random variable $z(x)\sim Bernoulli(\beta)$ and assign $x\in\Gamma$ whenever $z(x)=1$.  Each training data point $x\in \Gamma$ is assigned a label $g(x)$.
\item[\as4] We define an interaction potential $\eta: [0, \infty) \to [0,\infty)$. The interaction potential $\eta(\cdot)$ has support on the interval $[0,1]$, it is non-increasing, and $\eta$ restricted to $[0,1]$ is Lipschitz continuous. For $t\geq 1$ we extend $\eta$ so that $\eta(t) = 0.$ Moreover, $\eta(\tfrac{1}{2})>0$ and $\eta$ integrates to $1:$
$$ \int_{\R^d} \eta(|x|)dx =1.$$
We define $\eta_\eps(t) = \frac{1}{\eps^d}\eta(t/\eps).$
\item[\as5] (Geometric graphs of length $\eps$) There exist positive constants $c_1, c_2, \eps_{1}, \eps_{2}$ such that for every $x, y \in \X$ we have that
$$c_1 \eps_{1}^{-d}\mathbbm{1}_{|x-y|<\eps_{1}}\leq w_{xy} \leq c_2\eps_{2}^{-d}\mathbbm{1}_{|x-y|<\eps_{2}} .$$
\end{itemize}
We note that assumptions \as2 and \as4 are used in the $\eps-$graph construction (see Section \ref{ss_eps_graphs}).  In Lemma \ref{A5_eps} we prove that assumption \as5 holds for $\eps$-graphs.  Assumption \as3 has to do with how the labeled data points are chosen, and we note that the parameter $\beta\in (0,1)$ is the \emph{label rate}.

Assumption \as2 is satisfied for $k$-NN-graphs by construction. In Lemma \ref{A5_kNN} we prove that assumption \as5 holds with high probability for $k$-NN graphs.  Applying \as5, we prove \as1 holds with high probability, which means that Theorems \ref{th1} and \ref{th3} hold.
For $x\in \X$ we define $d_n(x)$ and $p_n(x)$ as follows:
\begin{align}
&d_{n}(x) = \Sigma_{y \in \X}w_{xy},\\ 
&p_{n}(x) =\Sigma_{y \in \X}w_{xy}\mathbbm{1}_{y\in\Gamma}. \label{p_n}
\end{align}
Then, as before, we define
\begin{equation}\label{delta}
 \delta= \min_{x\in \X}\frac{p_n(x)}{d_n(x)} =\min_{x\in \X}\frac{\Sigma_{y\in \X}w_{xy}\mathbbm{1}_{y\in\Gamma}}{\Sigma_{z\in \X}w_{xz}} .
 \end{equation}
 
\subsubsection{The \texorpdfstring{$\eps$}{epsilon}-Graphs}\label{ss_eps_graphs}

Let $0 < \epsilon \leq 1$. We construct our $\eps$-graph as follows.
\begin{definition}
We define an $\eps$-graph according to the following rules. The vertices $\X= \{x_1, \dots x_n \}$ are constructed according to \as2 and the labeled set $\Gamma$ is chosen according to \as3. We assume \as4 and the symmetric  edges $w_{xy} = w_{yx}$ and are obtained as follows: 
\begin{equation}
w_{xy} = 
\eta_{\eps}(|x-y|)
\end{equation}
\end{definition}

\begin{lemma}\label{A5_eps}
$\eps$-Graphs satisfy assumption \as5 with constants $c_1=2^{-d}\eta(\tfrac{1}{2}), c_2=\sup{\eta}, \\ \eps_1=\tfrac{\eps}{2},$ and $\eps_2=\eps$.
\end{lemma}
\begin{proof}
By their construction, $\eps-$graphs have the following property:
\begin{align}
w_{xy}=\eps^{-d}\eta(\frac{|x-y|}{\eps}) \in&
[\eps^{-d}\eta(\tfrac{1}{2})\mathbbm{1}_{|x-y|<\frac{\eps}{2}}, \eps^{-d}\sup{\eta}\mathbbm{1}_{|x-y|<\eps}]\\
\equiv & [(\eps/2)^{-d}(2^{-d}\eta(\tfrac{1}{2}))\mathbbm{1}_{|x-y|<\frac{\eps}{2}}, \eps^{-d}\sup{\eta}\mathbbm{1}_{|x-y|<\eps}].
\end{align}
so \as5 holds, with the constants stated in the Lemma.
\end{proof}

\subsubsection{The k-Nearest Neighbors Graphs}

We follow \cite{calder2022improved} in constructing a $k$-NN graph. We fix a number  $k\in \N,$ such that $1\leq k \leq n-1$.  The number of neighbors $N_\eps(x)$ of $x$ in an $\eps$-ball is
$$ N_\eps(x):=\sum_{j: 0<|x_j-x| \leq \eps} 1$$ and the distance to the $k$-nearest neighbor is
$$\eps_k(x):=\min\{\eps>0~: N_\eps(x)\geq k\}.$$

\begin{definition}
	We define a relation $\sim_k$ on $\X\times \X$ by declaring
	$$x_i \sim_k x_j$$
	if $x_j$ is among the $k$ nearest neighbors of $x_i$. 
\end{definition}

\begin{definition}
For the \textit{symmetric} k-nearest neighbor graph we construct the vertices according to \as2. If $x_i \sim_k x_j$ \textbf{or} $x_j \sim_k x_i$, we place an edge between $x_i$ and $x_j$. The edges are symmetric, meaning $w_{xy}=w_{yx}$ and
$$w_{xy}=\eta_{\max\{\eps_k(x), \eps_k(y)\}}(|x-y|). $$
\end{definition}


\begin{definition}
In the \textit{mutual} k-nearest neighbor graph we construct the vertices according to \as2. If $x_i \sim_k x_j$ \textbf{and} $x_j \sim_k x_i$, we place an edge between $x_i$ and $x_j$. The edges are symmetric, meaning $w_{xy}=w_{yx}$ and
$$w_{xy}=\eta_{\min\{\eps_k(x), \eps_k(y)\}}(|x-y|). $$
\end{definition}

In the following lemma, we verify that \as5 holds for $k$-NN graphs.
\begin{lemma}\label{A5_kNN} 
	There exists a $c>0,$ such that symmetric k-NN graphs and mutual k-NN graphs  satisfy assumption \as5 with probability $1-4\exp(-\tfrac{c}{2}(k/n)^{2/d}k),$ provided $k$ is not too large, meaning
	$$1\leq k\leq c n \eps^{d}$$ holds. 
\end{lemma}
\begin{proof}
	Consider $x\in\X$ --- a vertex in the $k-$NN graph. 
	We set
	$$\gamma = \tfrac{1}{\sqrt{2}}(k/n)^{1/d}.$$
	Let $\alpha_d$ be the volume of the $d$-dimensional Euclidean unit ball. 
	
	We apply Lemma 3.8 from \cite{calder2022improved} to estimate the expected value of the number of neighbors: we have that 
	$$\P(|\alpha_d \rho(x)n\eps_k(x)^d-k|\geq \tfrac{1}{2} k)\leq 4\exp(-c\gamma^2 k).$$
	Thus, with probability $1-4\exp(-c\gamma^2k)$ we obtain
	\begin{equation}
	\alpha_d \rho(x)n\eps_k^d \in \big[\tfrac{k}{2}, \tfrac{3k}{2}\big], \nonumber
	\end{equation}
	or equivalently
	$$ \left(\frac{k}{2\alpha_d\rho(x)n}\right)^{1/d} \leq \eps_k \leq \left(\frac{3k}{2\alpha_d \rho(x)n}\right)^{1/d}.$$	
	 Denote $m:=\inf_{x\in X}{(\frac{k}{2\alpha_d\rho(x)n})^{1/d}}$ and
	 $M:=\sup_{x\in X}(\frac{3k}{2\alpha_d \rho(x)n})^{1/d}.$ Thus, with probability at least $1-4\exp(-c\gamma^2k)$, we have:
	 $$m\leq \eps_k \leq M .$$ 
  In order to obtain the constants in \as5, we bound $w_{xy}$ from below and above.
	 Note that $\eps_k$ comes from the vertex $x$ or the vertex $y$, and that $\eta$ is a non-increasing function. 
  
  We start with the lower bound of $w_{xy}$:	 
	  \begin{align}
	 	w_{xy} = \frac{1}{\eps^d_k}\eta\left(\frac{|x-y|}{\eps_k} \right)  \nonumber
	 	&= \frac{1}{\eps^d_k}\eta\left(\frac{|x-y|}{\eps_k} \right)\mathbbm{1}_{|x-y|\leq\eps_k} \\ \nonumber
	 	&= \frac{1}{\eps^d_k}\eta\left(\frac{|x-y|}{\eps_k} \right)\mathbbm{1}_{|x-y|\leq m/2} +  \frac{1}{\eps^d_k}\eta\left(\frac{|x-y|}{\eps_k} \right)\mathbbm{1}_{|x-y|\in[m/2,\eps_k]} \\ \nonumber
	 	&\geq \frac{1}{\eps^d_k}\eta\left(\frac{m/2}{\eps_k} \right)\mathbbm{1}_{|x-y|\leq m/2} \\ \nonumber
	 	&\geq \frac{1}{\eps^d_k}\eta\left(\frac{1}{2} \right)\mathbbm{1}_{|x-y|\leq m/2} \\ \nonumber
	 	&= \eta\left(\frac{1}{2} \right) \left(\frac{m}{2\eps_k} \right)^d\left(\frac{m}{2}\right)^{-d}\mathbbm{1}_{|x-y|\leq m/2} \\ \nonumber
	 	&\geq \eta\left(\frac{1}{2} \right) \left(\frac{m}{2M} \right)^d\left(\frac{m}{2}\right)^{-d}\mathbbm{1}_{|x-y|\leq m/2}. \\ \nonumber	 	
	 \end{align}
	 We choose $\eps_1=\frac{m}{2}$ and $c_1=\eta(\frac{1}{2})(\frac{m}{2M})^d.$
	 
	 Next, we compute the upper bound for the edge weight:
	 \begin{align}
	 	w_{xy} = \frac{1}{\eps^d_k}\eta\left(\frac{|x-y|}{\eps_k} \right)  \nonumber
	 	&= \frac{1}{\eps^d_k}\eta\left(\frac{|x-y|}{\eps_k}  \right)\mathbbm{1}_{|x-y|\leq\eps_k} \\ \nonumber
	 	&\leq \eps_k^{-d}\max\eta \mathbbm{1}_{|x-y|\leq\eps_k}\\  \nonumber
	 	&= M^{-d} \left( \frac{M}{\eps_k} \right)^{d}\max\eta \mathbbm{1}_{|x-y|\leq\eps_k} \\ \nonumber
	 	&\leq M^{-d} \left( \frac{M}{m} \right)^{d}\max\eta \mathbbm{1}_{|x-y|\leq M}\\  \nonumber
	 	&\leq M^{-d} \mathbbm{1}_{|x-y|\leq M}\left\{ \left(\frac{M}{m} \right)^{d}\max\eta \right\}. \nonumber
	 \end{align}
	 Now, we choose $c_2=\left\{ \left(\frac{M}{m} \right)^{d}\max\eta \right\}$ and $\eps_2=M$.
	 This establishes all the constants in \as5, which concludes the proof.
\end{proof}

\subsubsection{General geometric graphs}

Now that we have shown that $\eps$-graphs and $k$-NN graphs satisfy \as5, we proceed to study general geometric graphs, which are essentially only assumed to satisfy \as5. This covers a wide variety of geometrically constructed graphs, including, but not limited to, $\eps$-graphs and $k$-NN graphs.

\begin{theorem} \label{th20}
Assume that \as2,\as3, \as4,\as5 hold. Consider the labeling rate $\beta$ and $\delta$ defined as in \eqref{delta}. Then, there exist positive constants $ c_3, c_4,$ and $c$ such that $\delta \geq c\beta$ and \as1 
hold with probability at least 
$$1-n\exp(-c_3n\eps_2^d)-n\exp(-c_4\beta n\eps_1^d).$$
Here $c_3, c_4,$ and $c$ only depend on $\eta, d, \Omega$, and $\rho.$ 
\end{theorem}

\begin{proof}
	Fix $x\in\X$. Because of \as5, we have that:
	\begin{align}
		d_n(x)=\sum_{y\in N_x}w_{xy}  &\leq \sum_{y\in \X\setminus x}c_2\mathbbm{1}_{|x-y|\leq {\eps}_2}{\eps}_2^{-d} :=\Phi \nonumber
	\end{align}
	Let us consider all the vertices in $\X\setminus x$ and treat them as random variables, calling them $Y_1,\dots, Y_{n-1}.$ Define $\psi_1(Y_i) =c_2\mathbbm{1}_{|x-Y_i|\leq {\eps}_2}\eps_2^{-d}$ and then it follows that 
	$$\Phi=\sum_{i=1}^{n-1} \psi_1(Y_i)=\sum_{i=1}^{n-1} c_2\mathbbm{1}_{|x-Y_i|\leq {\eps}_2}{\eps}_2^{-d}.$$
    Thus,
	$$\E[\Phi]=\sum_{Y_i \in \X\setminus x}\E[ c_2\mathbbm{1}_{|x-Y_i|\leq {\eps}_2}{\eps}_2^{-d}] = (n-1)c_2\int_{\Omega\cap B(x, \eps_2)} {\eps}_2^{-d}\rho(y)dy.$$
	We apply Bernstein's inequality (30) in \cite{calder2018game} to $\Phi$ obtaining that 	
	$$\P(|\Phi-\E(\Phi)|>c_d(n-1))\leq 2\exp\Big(\frac{-(n-1)c_d^2}{2(\sigma^2+\tfrac{1}{3}bc_d)} \Big)$$
	for $\sigma^2=Var(\psi_1(Y_i))$ and $b=\sup|\psi_1(Y_i)|.$ We observe that $b=\sup|c_2\mathbbm{1}_{|x-Y_i|\leq {\eps}_2}\eps_2^{-d}|=c_2\eps_2^{-d}$ and also 
 $$\sigma^2=Var(c_2\mathbbm{1}_{|x-Y_i|\leq {\eps}_2}\eps_2^{-d})=c_2^2\eps_2^{-2d}Var(\mathbbm{1}_{|x-Y_i|\leq {\eps}_2})\sim C\eps_2^{-d},$$ because $Var(\mathbbm{1}_{|x-Y_i|\leq {\eps}_2})$ is of order $\eps_2^d$ as we integrate over a ball of radius $\eps_2.$ 
 Thus, $\sigma^2$ and $b$ depend on $\Omega, \rho, c_2$, as well as on $\eps_2$ explicitly. 
 
 We union-bound, obtaining that there exists  two positive constants $c_5$ 
 and $c_3$ for which  
	$$\Phi \leq c_d (n-1)+(n-1)c_2\int_{\Omega\cap B(x, \eps_2)} {\eps}_2^{-d}\rho(y)dy=c_5(n-1)$$
	holds with probability at least $1-n\exp\{-c_3n\eps_2^d\}.$ 
 The same line of reasoning applies to $p_n(x)$ defined in \eqref{p_n}.
We define
 $$\psi_2(Y_i):=c_1\mathbbm{1}_{|x-Y_i|\leq{\eps}_1}{\eps}_1^{-d}\mathbbm{1}_{z(Y_i)=1}$$
 and $\hat{\Phi}=\sum_{i=1}^n\psi_2(Y_i).$ We compute
 
 $$\E[\hat{\Psi}]=\sum_{Y_i \in \X\setminus x}\E[ c_1\mathbbm{1}_{|x-Y_i|\leq {\eps}_1}{\eps}_1^{-d}\mathbbm{1}_{z(Y_i)=1}] = (n-1)c_1\beta \int_{\Omega\cap B(x, \eps_1)} {\eps}_1^{-d}\rho(y)dy.$$
 Thus there exists two positive constants $c_6$ and $c_4$ such that with probability $1-n\exp(-c_4\beta n\eps_1^d)$ we have
$$p_n(x)\geq(n-1)\beta c_6.$$
 
 Thus, there exist constants $c_5, c_6$ such that 
 \begin{equation}\label{d_n}
     d_{n}(x) < c_5(n-1)
 \end{equation} 
 and
 \begin{equation}\label{p_n_ineq}
     p_{n}(x) \geq c_6\beta (n-1)
 \end{equation}
 hold for all all $x$ with probability $1-n\exp(-c_3n\eps_2^d)-n\exp(-c_4\beta n\eps_1^d).$
 
We use the results we just obtained for \eqref{d_n} and \eqref{p_n_ineq} to estimate that
\begin{equation}
	\delta= \min_{x\in\X}\frac{p_{n}(x)}{d_{n} (x)}\geq \min\frac{c_5 (n-1)\beta}{c_6(n-1)}\geq c\beta.
\end{equation}

We now focus on proving \as1. 
	Pick any vertex $x \in X.$	Since by definition $$p_{n}(x) =\Sigma_{y \in \X}w_{xy}\mathbbm{1}_{z(y)=1}$$
	is a sum of non-negative quantities, and by \eqref{p_n_ineq} there exists a neighbor $x_\ell$ of $x$ such that 
	$$w_{xx_\ell}\mathbbm{1}_{z(x_\ell)=1} > 0.$$
	Thus, $z_\ell=1$ for some vertex $x_\ell$, therefore $x_\ell$ is labeled, and so $x$ has a neighbor in $\Gamma.$ This is a restatement of \as1. 
	This concludes the proof. 
\end{proof}
	In particular, with probability at least $1-n\exp(-c_3n\eps_2^d)-n\exp(-c_4\beta n\eps_1^d)$ assumption \as1 holds for $k$-NN graphs and $\eps$-graphs. Therefore, Theorem \ref{th1} and Theorem \ref{th2} also hold (with the same probability) for $\eps$-graphs, and $k$-NN graphs.
\begin{corollary}\label{c1}
	For an $\eps$-graph and $g\in Lip,$ we have that 
	$$|u(x) - g(x)|\leq \eps \Lip(g)\Big( \frac{2\log\frac{\sup g}{\eps \Lip(g)}}{1-\alpha +2\alpha\delta}+2\Big)$$
	holds with with probability at least $1-n\exp(-c_3n\eps^d)-n\exp(-c_4\beta n\eps^d).$
\end{corollary}
\begin{proof}
	The statement of Corollary \ref{c1} holds from  Theorem \ref{th1}, the 
	observation $|g(z) - g(y)| \leq \Lip(g) |z-y|$, and the fact that the distance between nearby nodes $|z-y|$ is of size $\eps.$
\end{proof}

\subsection{Stochastic block model graphs}\label{sec:stochasticblockmodels}

In this section we apply our ideas and results from Section \ref{sec:RGG} to the stochastic block model (SBM) graph. \red Since SBM graphs have \emph{discrete} geometric structure, as opposed to the \emph{continuous} geometric structure of a random geometric graph, there is no obvious relationship to a continuum PDE on an SBM graph. Nevertheless, the tug-of-war lemma provides useful results in this setting.\nc

We first define an SBM graph. The edge weights are binary $w_{xy}=1$ if there is an edge between $x$ and $y$, and $w_{xy}=0$ if there is no edge.  We partition the vertices $\X$ of the SBM graph into two disjoint blocks $\X=\X_0\cup \X_1$ according to an underlying binary label function $g:\X\to \{0,1\}$, as in \eqref{eq:xi}.  We define $r,q\in (0,1)$ as the two key probabilities for an SBM graph: $r$ is the probability of an \textit{intra-class} edge and $q$ is the probability of an \textit{inter-class} edge. In other words,
\begin{align*}
	&\P(w_{xy}=1)= r \ \ \ \text{if $x,y\in \X_i$ for $i=0$ or $i=1$}\\
	&\P(w_{xy}=1)= q \ \ \ \text{if $x\in \X_i$ and $y\in \X_{1-i}$  for $i=0$ or $i=1$.}
\end{align*}
By \as3, the probability that any vertex is labeled is $\beta$, thus only a portion $\Gamma \in \X$ is labeled. The rest of $\X$ is unlabeled and the goal is to label the points in $\X_0$ and $\X_1$ correctly. We also define $N_0=|\X_0|$, $N_1=|\X_1|$ and $n =|\X|=N_0+N_1$. Let us set $\Gamma_i = \Gamma\cap \X_i$ so that $\Gamma=\Gamma_0 \cup \Gamma_1$. 

Now, let $u$ be the solution of the $p$-Laplace equation \eqref{eq:ssl_pLap}, and define the two sets $\Y_0$ and $\Y_1$ as follows: we assign $x$ to $\Y_0$ whenever $u(x)\leq \frac{1}{2},$ and we assign  $x$ to $\Y_1$ whenever $u(x)> \frac{1}{2}$. That is, $\Y_0$ is the set of vertices classified as class $0$ by $p$-Laplace learning, and $\Y_1$ is the set classified as class $1$. In order to classify correctly, we need $\Y_0=\X_0$ and $\Y_1=\X_1$.

Consider a vertex $x_i \in \X.$ Let $\gamma_i$ be the ratio of the number of other-class neighbours to the number of neighbors of vertex $x_i$, that is, if $x_i \in \X_j$, $j\in \{0,1\}$, then
\begin{equation}\label{eq:gammai}
\gamma_i = \frac{\sum_{x\in \X_{1-j}}w_{x_i,x}}{\sum_{x\in \X}w_{x_i,x}}.
\end{equation}
For a vertex $x_i$, the quantity $\gamma_i$ is exactly the probability that a random walk step will jump to the opposite block. Let $\beta_i$ be the fraction of neighbors from the same class that are labeled. Thus, if $x_i\in \X_j$ with $j\in \{0,1\}$ then 
\begin{equation}\label{eq:betai}
\beta_i = \frac{\sum_{x\in \X_{j}\cap \Gamma}w_{x_i,x}}{\sum_{x\in \X_j}w_{x_i,x}}.
\end{equation}
The term $\beta_i$ is the probability that the random walk exits at at labeled node in the same class on the next step, conditioned on the event that the random walk remains in the same class. 


We now derive upper and lower bounds for $\gamma_i$ and $\beta_i$.  The lower bound for $\beta_i$ implies that \as1 holds with high probability, but we do not use \as1 directly in this section.
\begin{lemma}\label{lem:sbm_properties}
Let us consider an SBM graph with parameters $N_0,N_1,r,q$, and label rate $\beta$. Let $\sigma_1 \geq 0$ and $0 \leq \sigma_2 < 1$. Then, with probability of at least 
\begin{equation}\label{eq:sbm_prob}
1 - \sum_{j=0}^1N_j\left\{\exp\left(-\frac{qN_{1-j} \sigma_1^2}{2(1 + \tfrac13 \sigma_1)}\right) - 3\exp\left(-\frac{1}{8} \beta rN_j \sigma_2^2 \right)\right\}.
\end{equation}
for any $x_i \in \X_j$, $j\in \{0,1\}$, we have
\begin{equation}\label{eq:gammai_bound}
\gamma_i \leq \frac{(1+\sigma_1)qN_1}{(1+\sigma_1)qN_1 + (1-\sigma_2)r(N_0-1)},
\end{equation}
and
\begin{equation}\label{eq:betai_bound}
\beta_i \geq \frac{1-\sigma_2}{1+\sigma_2}\beta.
\end{equation}
\end{lemma}
\begin{proof}
Pick a vertex $x_i\in \X_0.$
Define the following notation:
\[ Y_j = \begin{cases}
1,&\text{if } x_i \text{ is connected to } x_j \in\X_1\\
0,&\text{otherwise.}\end{cases}\]
\[ Z_j = \begin{cases}
1,&\text{if } x_i \text{ is connected to } x_j \in\X_0 \\
0,&\text{otherwise.}\end{cases}\]
Denote $A=\sum_{j=1}^{N_1}Y_j$, $B=\sum_{j=1}^{N_0}Z_j$ and note that by \eqref{eq:gammai} we have
\[\gamma_i =\frac{A}{A+B} = \frac{1}{1 + C}, \ \ C = \frac{B}{A}.\]
We now apply the Chernoff bounds (see e.g. \cite{boucheron2013concentration}) for $A$ and $B$ to obtain
\begin{align*}
\P(A\geq (1+\sigma_1)qN_1)&\leq \exp\left(-\frac{qN_1 \sigma_1^2}{2(1 + \tfrac13 \sigma_1)}\right) \\
\P\left(|B - r(N_0-1)| \leq \sigma_2r(N_0-1) \right)&\leq 2\exp\left(-\frac{3}{8} r(N_0-1) \sigma_2^2 \right),
\end{align*}
for any $\sigma_1 > 0$ and $0 \leq \sigma_2 < 1$. Therefore
\[C \geq \frac{(1-\sigma_2)r(N_0-1)}{(1+\sigma_1)qN_1}\]
holds with probability at least 
\[1 - \exp\left(-\frac{qN_1 \sigma_1^2}{2(1 + \tfrac13 \sigma_1)}\right) - 2\exp\left(-\frac{3}{8} r(N_0-1) \sigma_2^2 \right).\]
Assuming this event holds, we then have 
\[\gamma_i \leq \frac{1}{1+C} \leq \frac{(1+\sigma_1)qN_1}{(1+\sigma_1)qN_1 + (1-\sigma_2)r(N_0-1)} = \frac{(1+\sigma)qN_1}{(1+\sigma)qN_1 + r(N_0-1)},\]
which establishes \eqref{eq:gammai_bound}, upon union bounding over $\X_0$.

Now, let $V_i$ be \iid~Bernoulli($\beta$) for $i=1,\dots,n$ so that $x_i\in \Gamma$ if and only if $V_i=1$. Then 
\[\beta_i = \frac{1}{B}\sum_{j=1}^n Z_jV_j.\]
Since $Z_jV_j$ is Bernoulli($\beta r$) we have
\[\P\left(\sum_{j=1}^n Z_jV_j \leq (1-\sigma_2)\beta r(N_0-1) \right)\leq \exp\left(-\frac{3}{8} \beta r(N_0-1) \sigma_2^2 \right).\]
Assuming this event holds as well yields
\[\beta_i \geq \frac{(1-\sigma_2)\beta r(N_0-1)}{(1+\sigma_2)r(N_0-1)} = \frac{1-\sigma_2}{1+\sigma_2}\beta,\]
which establishes, with a union bound, \eqref{eq:betai_bound}.  The argument is similar for $x_i \in \X_1$.  The probabilities are simplified by combining like terms and using that $N_i -1 \geq N_i/2$ since $N_i \geq 2$. 
\end{proof}

Now we prove the main result for SBM graphs.

\begin{theorem}\label{thm:sbm} 
Let $2 \leq p<\infty$ and consider an SBM graph with parameters $r,q,N_0,N_1$, and label rate $\beta$. Let $\sigma_1 \geq 0$, $0 \leq \sigma_2 < 1$ and assume that
\begin{equation}\label{eq:rq_lower}
\frac{r}{q} \beta >(1+\sigma)\max\left\{\tfrac{N_0}{N_1-1},\tfrac{N_1}{N_0-1} \right\},
\end{equation}
where
\begin{equation}\label{eq:sigma_def}
1 + \sigma = \frac{(1+\sigma_1)(1+\sigma_2)}{(1-\sigma_2)^2}.
\end{equation}
Then the solution $u$ to \eqref{eq:ssl_pLap} classifies all vertices correctly with probability at least \eqref{eq:sbm_prob}.
\end{theorem}
\begin{remark}\label{rem:sbm_p}
Note that the condition $p<\infty$ ensures that $\alpha>0$ in \eqref{eq:dpp}, and thus there is always a positive probability of taking a random walk step. The proof of Theorem \ref{thm:sbm} relies entirely on the random walk, and allows the two players to balance each other out. We expect this is sub-optimal and is why the value of $p$ does not explicitly enter the condition on $\tfrac{r}{q}$. 
\end{remark}
\begin{proof}[Proof of Theorem \ref{thm:sbm}]
We fix $\sigma_1 \geq 0$ and $0 \leq \sigma_2 \leq \frac{1}{2}$ and assume that the conclusions of Lemma \ref{lem:sbm_properties} hold.  We consider the stochastic tug of war process $X_0,X_1,\dots,$ starting from $X_0=x\in \X_0$. We define the associated random variables $W_k$ by
\[W_k = 
\begin{cases}
0,& \text{if } X_{k} \in \Gamma_0\\
1,& \text{if } X_{k}\in \X_1\\
2,& \text{if } X_{k} \in \X_0 \setminus \Gamma_0.
\end{cases}\]
The random variables $W_k$ indicate whether the process stops at $\Gamma_0$, moves to $\X_1$, or stays in $\X_0$ and does not exit.  The probabilities of each of these transitions depends on the vertex the process is currently at. If $X_{k-1} = x_i\in \X_0$ then 
\begin{align*} 
p_{0,i} := \P(W_k = 0) &= \frac{1-\alpha}{2} + \alpha(1-\gamma_i)\beta_i, \\ 
p_{1,i} := \P(W_k = 1) &\leq \frac{1-\alpha}{2} + \alpha\gamma_i, \\
p_{2,i} := \P(W_k = 2) &\geq \alpha(1-\gamma_i)(1-\beta_i),
\end{align*}
where $\gamma_i$ is defined in \eqref{eq:gammai}, and $\beta_i>0$ is fraction of same-class neighbors of $x_i$ that are labeled.  Let $\tau$ be the stopping time defined by 
\[\tau =\inf \{k: W_{k}=0 \text{ or } W_{k}=1 \}.\]
Let us define
\[p_0 = \min_{1 \leq i \leq n} p_{0,i} = \frac{1-\alpha}{2} + \alpha(1-\gamma_{max})\beta_{min} \ \ \text{ and }  \ \ p_{1}=  \max_{1 \leq i \leq n}p_{1,i} = \frac{1-\alpha}{2} + \alpha\gamma_{max},\]
where we write 
\[\gamma_{max} = \max_{1 \leq i \leq n}\gamma_i \ \ \text{ and } \ \ \beta_{min} = \min_{1\leq i \leq n}\beta_i.\]
By Doob's optional stopping theorem and Lemma \ref{lem:submart} we have $u(x) \leq \E[u(X_\tau)]$. \red Note that the expected value of $\tau$ is finite, since the probability of stopping at each step is $p_{0,i}+p_{1,i}\geq p_0 > 0$. \nc

Let $i_{k}$ denote the vertex index of the stochastic tug-of-war process on the $k^{\rm th}$ step, so that $X_k = x_{i_k}$. Then conditioned on $i_{\tau-1} = i$, the probability of $W_\tau=1$ is exactly $p_{1,i}/(p_{0,i} + p_{1,i})$. Therefore by the law of conditional probability, \red and the fact that $x_{i_{\tau-1}}\in \X_0$ we have
\begin{align*}
u(x) \leq \E[u(X_\tau)] \leq \P(X_\tau \in \X_1) &= \sum_{i=1}^n \frac{p_{1,i}}{p_{0,i} + p_{1,i}} \P(i_{\tau-1}=i) \\
&=\sum_{x_i\in \X_0} \frac{p_{1,i}}{p_{0,i} + p_{1,i}} \P(i_{\tau-1}=i) \leq \frac{p_1}{p_0 + p_1}.
\end{align*}\nc
Since $x\in \X_0$, its correct label is $0$, so in order to classify $x$ correctly we need to show that $u(x) < \frac{1}{2}$, which holds whenever $p_0/p_1 > 1$, or rather, $p_0 > p_1$. This is equivalent to 
\begin{equation}\label{eq:suff_cond}
\beta_{min} > \frac{\gamma_{max}}{1-\gamma_{max}}.\
\end{equation}

By Lemma \ref{lem:sbm_properties} we have
\[\gamma_{max} \leq \frac{(1+\sigma_1)qN_1}{(1+\sigma_1)qN_1 + (1-\sigma_2)r(N_0-1)},\]
and so 
\[1 - \gamma_{max} \geq \frac{(1-\sigma_2)r(N_0-1)}{(1+\sigma_1)qN_1 + (1-\sigma_2)r(N_0-1)}.\]
Therefore
\[\frac{\gamma_{max}}{1-\gamma_{max}} \leq \frac{(1+\sigma_1)qN_1}{(1-\sigma_2)r(N_0-1)}.\]
Also, by Lemma \ref{lem:sbm_properties} we have
\[\beta_{min} \geq \frac{1-\sigma_2}{1+\sigma_2}\beta.\]
Thus, a sufficient condition for \eqref{eq:suff_cond} to hold is 
\[\frac{1-\sigma_2}{1+\sigma_2}\beta >\frac{(1+\sigma_1)qN_1}{(1-\sigma_2)r(N_0-1)},\]
which is equivalent to 
\[\frac{r}{q}\beta > (1+\sigma) \frac{N_1}{N_0-1}, \]
where $\sigma$ is defined in \eqref{eq:sigma_def}.  A similar argument holds for $x\in \X_1$, except with $N_0$ and $N_1$ reversed.  This concludes its proof.
\end{proof}

%% file: sections/numerics.tex
\subsection{Numerical results}
\label{sec:numerics}

\begin{figure}[!t]
\centering
\subfloat[$N_0=N_1=1500$]{\includegraphics[width=0.48\textwidth]{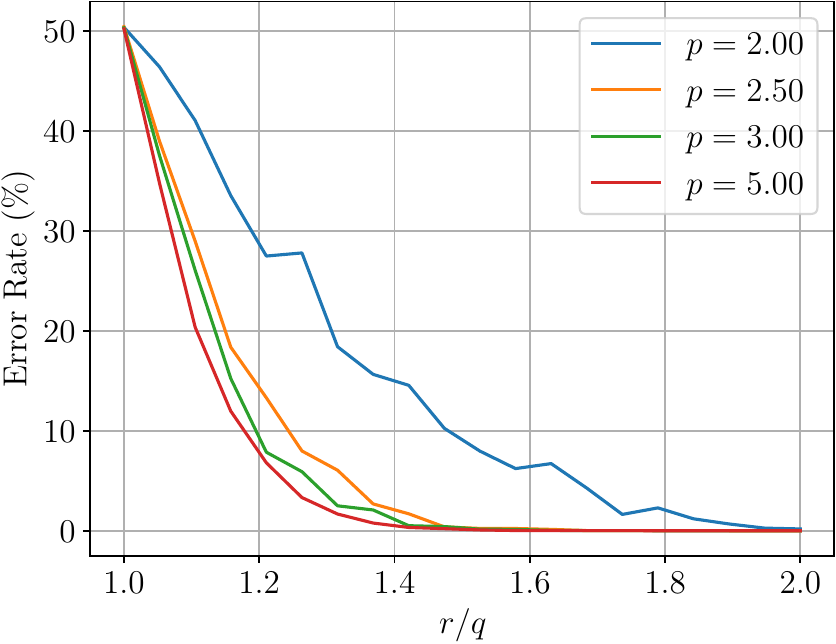}}
\hfill
\subfloat[$N_0=2000,N_1=1000$]{\includegraphics[width=0.48\textwidth]{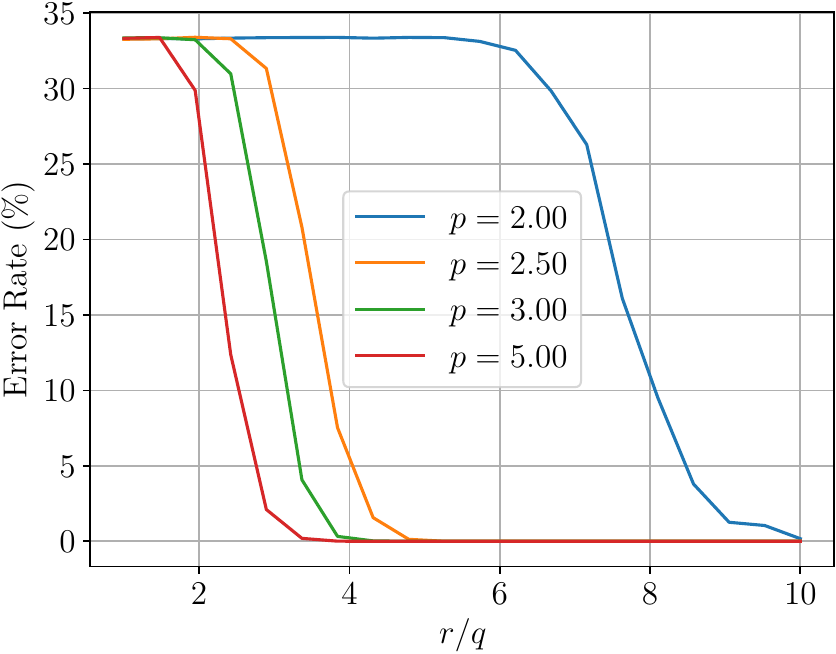}\label{fig:sbmb}}
\caption{Example of the effect of $r/q$ on the classification performance on stochastic block model graphs for various values of $p$. }
\label{fig:sbm}
\end{figure}

We present here some numerical results with synthetic and real data, in order to illustrate the main results established in this section, and the room for improvement in future work.  The Python code to reproduce the results in this section is available on GitHub.\footnote{\url{https://github.com/jwcalder/p-Laplace-consistency}}

Our first experiment is with a stochastic block model graph to illustrate Theorem \ref{thm:sbm}. We take a graph with $n=3000$ vertices and consider the case of equal block sizes $N_0=N_1=1500$ and unbalanced block sizes $N_0=2000,N_1=1000$. We take the intraclass probability to be $r=0.5$, and we vary the interclass probability $q\leq r$. In Figure \ref{fig:sbm} we show the error rates for binary classification using a labeling rate of $\beta=0.2$ as a function of the ratio $r/q$. For each value of $r/q$ we ran $100$ randomized trials and averaged the error rates. For equal size blocks $N_0=N_1$, we see that the error rate decreases rapidly when $r/q > 1$. Theorem \ref{thm:sbm} guarantees that all vertices are classified correctly when $r/q > \beta^{-1}=5$, which is a pessimistic bound in this case, since all values of $p$ classify correctly when $r/q=2$. For unequal block sizes, $N_0/N_1 =2$ in Figure \ref{fig:sbmb}, we see that for larger values of $p$, the error rate decreases quickly when $r/q > 2$, while $p=2$ requires $r/q > 8$ to see a similar decrease, and $r/q=10$ to classify correctly. In this case, Theorem \ref{thm:sbm} guarantees correct classification with high probability when $r/q > 2\beta^{-1} = 10$, which agrees very closely with the $p=2$ result, but is pessimistic for $p > 2$. Thus, our results may be tight for $p=2$, but there is clearly much room for improvement in the range $p>2$. In particular, we currently cannot explain why the results in Figure \ref{fig:sbmb} are so dramatically better as $p$ increases. An analysis of this sort would presumably require a different approach to Theorem \ref{thm:sbm} that exploited the tug-of-war game as well.

To illustrate our results on geometric graphs we turn to real data. We consider the MNIST data set which consists of $70000$ images of handwritten digits between $0$ and $9$ \cite{lecun1998gradient}, with each image a $28\times 28$ grayscale image. We also consider the Cifar-10 data set \cite{krizhevsky2009learning} which contains $60000$ natural images from 10 classes (airplane, automobile, bird, cat, deer, dog, frog, horse, ship, truck). Each Cifar-10 image is a $32\times 32$ pixel color image. To make the problem more similar to our setting, we restrict the data sets from 10 classes down to 2 so we have a binary classification problem. For MNIST we use the 4s and 9s, which are one of the harder pairs of digits to separate, while for Cifar-10 we use the deer and dogs.  Figure \ref{fig:image_sample} shows some example images from each data set.

\begin{figure}[!t]
\centering
\subfloat[MNIST]{\includegraphics[height=0.215\textheight]{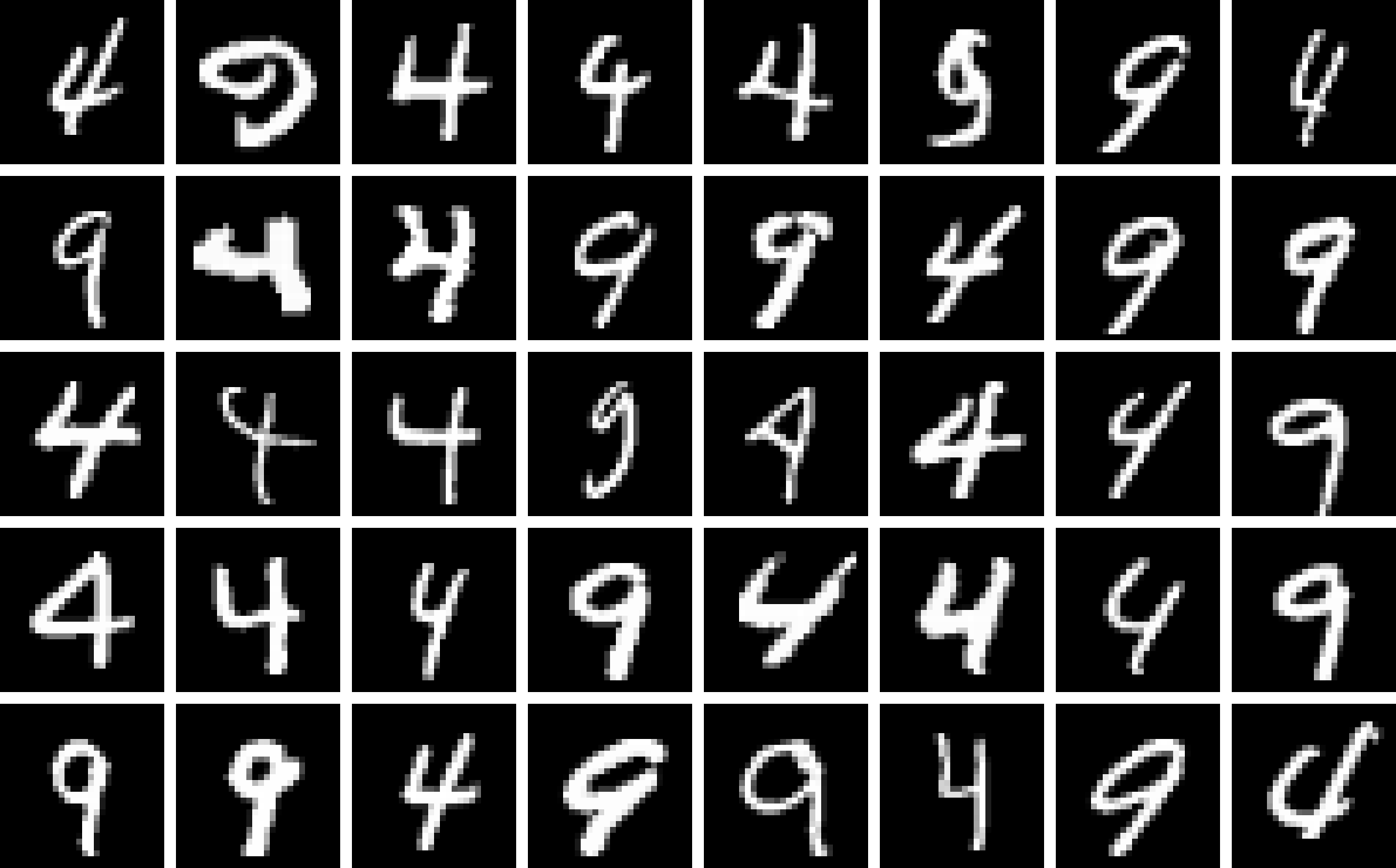}}
\hfill
\subfloat[Cifar-10]{\includegraphics[height=0.215\textheight]{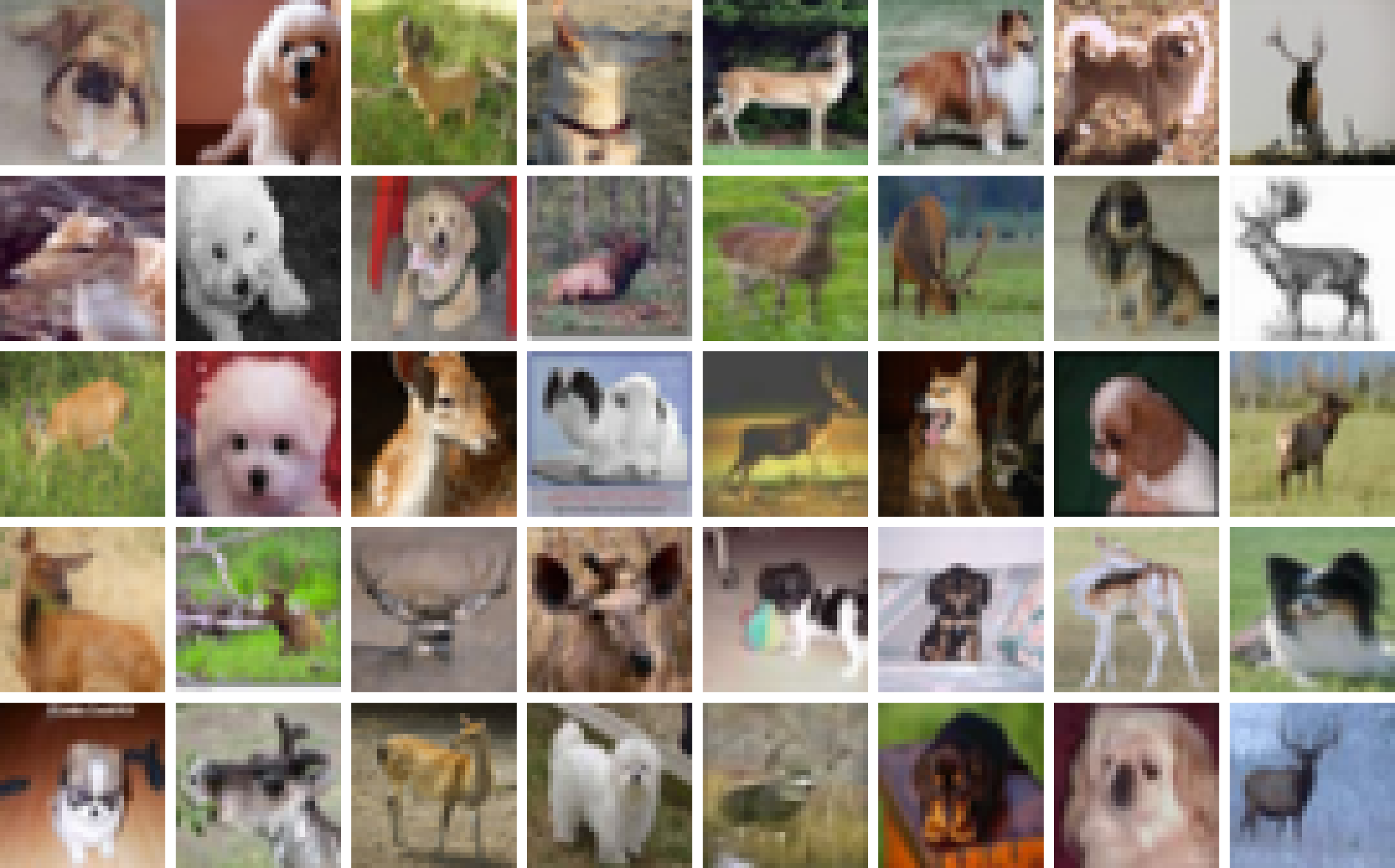}}
\caption{Sample of MNIST 4s and 9s, and Cifar-10 deer and dogs.}
\label{fig:image_sample}
\end{figure}

In order to obtain a good feature embedding for constructing the graph, we used a variational autoencoder \cite{kingma2013auto} for MNIST and the contrastive learning SimCLR method \cite{chen2020simple} for Cifar-10. Both methods are unsupervised deep learning algorithm that learn feature embeddings, or representations, of image data sets that maps the images into a latent space identified with $\R^k$ where the similarity in latent features is far more informative of image similarity than pixel-wise similarity. After embedding the images into the feature space, we constructed a $10$ nearest neighbor graph using the angular similarity, as in \cite{calder2020poisson}. 

\begin{figure}[!t]
\centering
\subfloat[MNIST]{\includegraphics[width=0.48\textwidth]{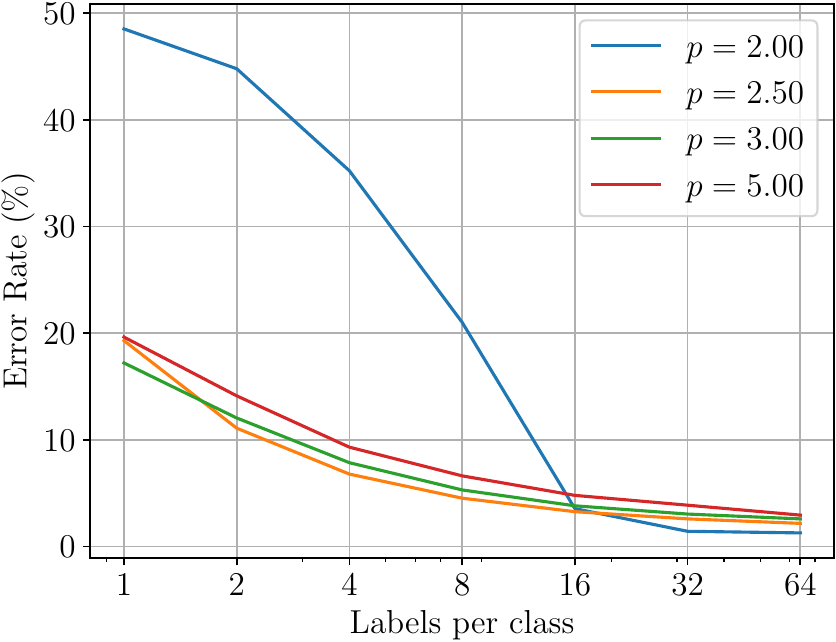}}
\hfill
\subfloat[Cifar-10]{\includegraphics[width=0.48\textwidth]{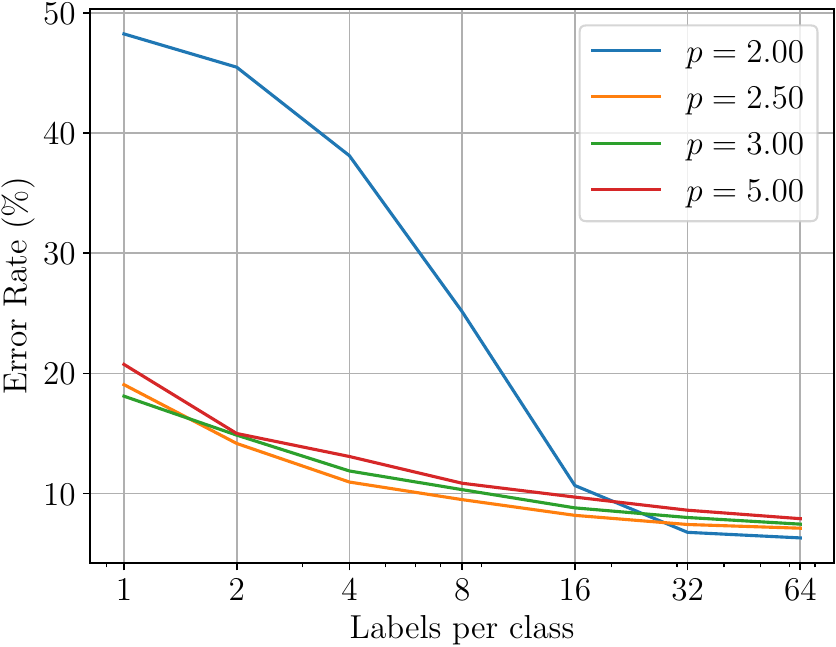}}
\caption{Accuracy for binary classification of two classes in the MNIST and Cifar-10 data sets using $p$-Laplacian regularization with different values of $p$. We used the 4s and 9s from MNIST, which are one of the more difficult pairs to separate, and the deer and trucks from Cifar-10.}
\label{fig:images}
\end{figure}

For our experiment we used 1 up to 64 labeled examples per class for each data set, ran 100 trials at each label rate, and reported the average error rate. Figure \ref{fig:images} shows the error rates for different values of $p$ as a function of label rate. We can see that larger values of $p$ give better classification performance at very low label rates, as expected from previous work \cite{flores2022algorithms}. Once the label rate is sufficiently large, the error rate becomes small. As predicted in Theorem \ref{th2}, the only points that are misclassified are those that are close to the boundary between classes in terms of graph distance, which are a small fraction of the data points --- based on Figure \ref{fig:images} about 2-4\% for MNIST and 5-10\% for Cifar-10.

%% file: sections/conclusion.tex

\section{Conclusions and future work}
\label{sec:conc}

In this paper, we gave a thorough overview of the intersection between graph-based semi-supervised learning and PDEs, and highlighted problems focused on consistency of classification that have not received significant attention yet in the community. We presented some preliminary results on consistency of $p$-Laplacian based semi-supervised learning. Our results use the stochastic tug-of-war interpretation of the $p$-Laplacian on a graph, and we also provided a brief overview of this field. We proved consistency results for general graphs, geometric graphs, and stochastic block model graphs, the latter of which are not usually covered by continuum PDE-based arguments. One of our findings is that the tug-of-war game transfers nicely between different graph structures, and does not require the geometric structure of the graph. We also presented numerical experiments on synthetic and real data that illustrated our results and suggested directions for future work. 

We highlight below some open problems for future work.
\begin{enumerate}
\item \textbf{Relaxing the assumption \as1}. The assumption \as1 \red is a strong simplifying assumption we made in order to obtain preliminary results. It \nc asks that every vertex on the graph has a labeled neighbor, and is used to make our results tractable. In practice, this means that, in the geometric graph setting, labels do not propagate very far on the graph. It would be very interesting, and more practically relevant, to extend these results to settings where \as1 does not hold. This would require a far more delicate martingale analysis than we provided in this paper. 
\item \textbf{Stochastic block models}. Our proof technique in Theorem \ref{thm:sbm} for the stochastic block model graph only exploited the random walk and ignored the tug-of-war component of the game. The numerical results in Section \ref{sec:numerics} suggest that the classification accuracy improves dramatically as $p>2$ increases, even at moderately large label rates --- 20\% in this case. In order to explain this, it would seem necessary to improve Theorem \ref{thm:sbm} by utilizing the tug-of-war game as well, so that the condition \eqref{eq:rq_lower} on $r/q$ depends on $p$.
\item \red \textbf{Noise and corruption:} In real world applications, we may observe noisy or corrupted labels $g = g^\dagger + \xi$, where $\xi(y)$ for $y\in \Gamma$ are independent and identically distributed random variables. For simplicity assume they have mean zero and variance $\sigma^2>0$. The results in this paper would establish that we recover $g$, while the true goal is to recover the clean uncorrupted labels $g^\dagger$. Applying the sub-martingale estimate Lemma \ref{lem:main_tug} in this case and subtracting $g^\dagger$ from both sides we obtain
\[|u(x) - g^\dagger(x)| \leq \E[|g^\dagger(X_\tau) - g^\dagger(x) \, | \, X_0=x] + \left|\E[ \xi(X_\tau)|\,  \, X_0=x]\right|,\]
where the expectation is over the random walk, and not the independent noise $\xi$.  To recover $g^\dagger$, we need to be able to show that the last term above is small, either in expectation or with high probability. Without the absolute values it is given by
\[\E[ \xi(X_\tau)\, | \, X_0=x] = \sum_{y\in \Gamma}\P(X_\tau = y) \xi(y).\]
If we assume the distribution of the stopping vertex $X_\tau$ is independent of the noise $\xi$, then the variance of this term with respect to the noise is given by
\begin{equation}\label{eq:vargamma}
\Var\E[ \xi(X_\tau)\, | \, X_0=x] = \sigma^2\sum_{y\in \Gamma}\P(X_\tau = y)^2.
\end{equation}
Thus, understanding the noisy setting requires a finer analysis of the stochastic process so that we can estimate the $\ell^2$ norm of the probability distribution of $X_\tau$. If $X_\tau$ is well-spread out on $\Gamma$, say  $\P(X_\tau=y) = \frac{1}{\#\Gamma}$, then the variance \eqref{eq:vargamma} is small: $\sigma^2/\#\Gamma$. On the other hand, if the distribution concentrates on a single node, $\P(X_\tau=y)=\delta_{y=z}$ for some $z\in \Gamma$, then the variance is $\sigma^2$; the same as that of the noise $\xi(z)$. Essentially, the more labels we average over to compute $u(x)$, the smaller the effect of noise. Estimating the variance in \eqref{eq:vargamma} appears to be a non-trivial undertaking that we leave for future work. In particular, $X_\tau$ will generally only be independent of $\xi$ when $p=2$; for $p>2$ the variance computation in \eqref{eq:vargamma} will also contain covariance terms that may be difficult to control.  

On the other hand, we expect the SBM analysis in Theorem \ref{thm:sbm} to extend in a fairly direct way to the setting where some fraction $\theta\in (0,1)$ of the labels are corrupted by flipping to the opposite class. In this case, we expect a similar result to hold where $\beta_i$ decreases, and $\gamma_i$ increases, by an amount proportional to $\theta$, and then the left hand side of \eqref{eq:rq_lower} is decreased by a factor like $1-\theta$. We leave this to a future study as well. 
\nc
\item \textbf{Extension to other random-walk models}. It would be interesting to extend these results to other models that have random walk interpretations, including Poisson learning \cite{calder2020poisson}, PWLL \cite{calder2023poisson,miller2023active}, and the properly weighted Laplacian \cite{calder2020properly}. Some of the same high level ideas may work, but we expect many of the ingredients to be different. The case of Laplace learning was essentially already studied in \cite{calder2023rates}.
\item \textbf{Similar results for other models}. There are a number of models that do not have random walk interpretations, such as the variational $p$-Laplacian \cite{el2016asymptotic,slepcev2019analysis} and the MBO methods \cite{merkurjev2013mbo,hu2013method,garcia2014multiclass,boyd2018simplified,merkurjev2018semi,merkurjev2014diffuse}. It would be interesting to prove similar consistency results for these methods, though the techniques would be substantially different, since as far as we are aware, there are no representation formulas that express the solutions through stochastic processes in these works. 
\end{enumerate}




%% file: main.bbl
\begin{thebibliography}{100}

\bibitem{alamgir2011phase}
M.~Alamgir and U.~V. Luxburg.
\newblock Phase transition in the family of p-resistances.
\newblock In {\em Advances in Neural Information Processing Systems}, pages
  379--387, 2011.

\bibitem{ando2007learning}
R.~K. Ando and T.~Zhang.
\newblock Learning on graph with {L}aplacian regularization.
\newblock In {\em Advances in neural information processing systems}, pages
  25--32, 2007.

\bibitem{armstrong2012finite}
S.~Armstrong and C.~Smart.
\newblock A finite difference approach to the infinity Laplace equation and
  tug-of-war games.
\newblock {\em Transactions of the American Mathematical Society},
  364(2):595--636, 2012.

\bibitem{aronsson2004tour}
G.~Aronsson, M.~Crandall, and P.~Juutinen.
\newblock A tour of the theory of absolutely minimizing functions.
\newblock {\em Bulletin of the American mathematical society}, 41(4):439--505,
  2004.

\bibitem{attouchi2021gradient}
A.~Attouchi, H.~Luiro, and M.~Parviainen.
\newblock Gradient and Lipschitz estimates for tug-of-war type games.
\newblock {\em SIAM Journal on Mathematical Analysis}, 53(2):1295--1319, 2021.

\bibitem{azad2022learning}
A.~Azad.
\newblock Learning Label Initialization for Time-Dependent Harmonic Extension.
\newblock {\em arXiv preprint arXiv:2205.01358}, 2022.

\bibitem{bardi1997optimal}
M.~Bardi, I.~C. Dolcetta, et~al.
\newblock {\em Optimal control and viscosity solutions of
  Hamilton-Jacobi-Bellman equations}, volume~12.
\newblock Springer, 1997.

\bibitem{belkin2004regularization}
M.~Belkin, I.~Matveeva, and P.~Niyogi.
\newblock Regularization and semi-supervised learning on large graphs.
\newblock In {\em Learning Theory: 17th Annual Conference on Learning Theory,
  COLT 2004, Banff, Canada, July 1-4, 2004. Proceedings 17}, pages 624--638.
  Springer, 2004.

\bibitem{belkin2002using}
M.~Belkin and P.~Niyogi.
\newblock Using manifold stucture for partially labeled classification.
\newblock {\em Advances in neural information processing systems}, 15, 2002.

\bibitem{belkin2003laplacian}
M.~Belkin and P.~Niyogi.
\newblock Laplacian eigenmaps for dimensionality reduction and data
  representation.
\newblock {\em Neural computation}, 15(6):1373--1396, 2003.

\bibitem{belkin2004semi}
M.~Belkin and P.~Niyogi.
\newblock Semi-supervised learning on {Riemannian} manifolds.
\newblock {\em Machine learning}, 56:209--239, 2004.

\bibitem{bengio2006label}
Y.~Bengio, O.~Delalleau, and N.~Le~Roux.
\newblock {\em Label Propagation and Quadratic Criterion}, pages 193--216.
\newblock MIT Press, semi-supervised learning edition, January 2006.

\bibitem{boucheron2013concentration}
S.~Boucheron, G.~Lugosi, and P.~Massart.
\newblock Concentration Inequalities: A Nonasymptotic Theory of Independence.
  Univ. Press, 2013.

\bibitem{boyd2018simplified}
Z.~M. Boyd, E.~Bae, X.-C. Tai, and A.~L. Bertozzi.
\newblock Simplified energy landscape for modularity using total variation.
\newblock {\em SIAM Journal on Applied Mathematics}, 78(5):2439--2464, 2018.

\bibitem{bozorgnia2023graph}
F.~Bozorgnia, M.~Fotouhi, A.~Arakelyan, and A.~Elmoataz.
\newblock Graph based semi-supervised learning using spatial segregation
  theory.
\newblock {\em Journal of Computational Science}, 74:102153, 2023.

\bibitem{bridle2013p}
N.~Bridle and X.~Zhu.
\newblock p-voltages: {L}aplacian regularization for semi-supervised learning
  on high-dimensional data.
\newblock In {\em Eleventh Workshop on Mining and Learning with Graphs
  (MLG2013)}, 2013.

\bibitem{brown2023contrastive}
J.~Brown, R.~O'Neill, J.~Calder, and A.~L. Bertozzi.
\newblock Utilizing Contrastive Learning for Graph-Based Active Learning of
  {SAR} Data.
\newblock {\em SPIE Defense and Commercial Sensing: Algorithms for Synthetic
  Aperture Radar Imagery XXX}, 2023.

\bibitem{bungert2022uniform}
L.~Bungert, J.~Calder, and T.~Roith.
\newblock Uniform Convergence Rates for {L}ipschitz Learning on Graphs.
\newblock {\em IMA Journal of Numerical Analysis}, 2022.

\bibitem{bungert2023ratio}
L.~Bungert, J.~Calder, and T.~Roith.
\newblock Ratio convergence rates for Euclidean first-passage percolation:
  Applications to the graph infinity Laplacian.
\newblock {\em To appear in Annals of Applied Probability}, 2023.

\bibitem{calder2018game}
J.~Calder.
\newblock The game theoretic p-{L}aplacian and semi-supervised learning with
  few labels.
\newblock {\em Nonlinearity}, 32(1), 2018.

\bibitem{calder2019consistency}
J.~Calder.
\newblock Consistency of {L}ipschitz learning with infinite unlabeled data and
  finite labeled data.
\newblock {\em SIAM Journal on Mathematics of Data Science}, 1(4):780--812,
  2019.

\bibitem{calder2023poisson}
J.~Calder, B.~Cook, M.~Thorpe, D.~Slep{\v{c}ev}, Y.~Zhang, and S.~Ke.
\newblock Graph-Based semi-supervised learning with {P}oisson equations.
\newblock {\em In preparation}, 2023.

\bibitem{calder2020poisson}
J.~Calder, B.~Cook, M.~Thorpe, and D.~Slep\v{c}ev.
\newblock {Poisson Learning: Graph based semi-supervised learning at very low
  label rates}.
\newblock {\em Proceedings of the 37th International Conference on Machine
  Learning, PMLR}, 119:1306--1316, 2020.

\bibitem{calder2022hamilton}
J.~Calder and M.~Ettehad.
\newblock Hamilton-{J}acobi equations on graphs with applications to
  semi-supervised learning and data depth.
\newblock {\em Journal of Machine Learning Research}, 23(318):1--62, 2022.

\bibitem{calder2022improved}
J.~Calder and N.~Garc\'ia~Trillos.
\newblock Improved spectral convergence rates for graph {Laplacians} on
  $\varepsilon$-graphs and k-{NN} graphs.
\newblock {\em Applied and Computational Harmonic Analysis}, 60:123--175, 2022.

\bibitem{calder2022Lip}
J.~Calder, N.~Garc\'ia~Trillos, and M.~Lewicka.
\newblock {Lipschitz regularity of graph Laplacians on random data clouds}.
\newblock {\em SIAM Journal on Mathematical Analysis}, 54(1):1169--1222, 2022.

\bibitem{calder2020properly}
J.~Calder and D.~Slep\v{c}ev.
\newblock {Properly-weighted graph {L}aplacian for semi-supervised learning}.
\newblock {\em Applied Mathematics and Optimization}, 82:1111--1159, 2020.

\bibitem{calder2023rates}
J.~Calder, D.~Slep\v{c}ev, and M.~Thorpe.
\newblock {Rates of convergence for {Laplacian} semi-supervised learning with
  low labeling rates}.
\newblock {\em Research in Mathematical Sciences special issue on PDE methods
  for machine learning}, 10(10), 2023.

\bibitem{cervino2023learning}
J.~Cervino, L.~F. Chamon, B.~D. Haeffele, R.~Vidal, and A.~Ribeiro.
\newblock Learning globally smooth functions on manifolds.
\newblock In {\em International Conference on Machine Learning}, pages
  3815--3854. PMLR, 2023.

\bibitem{chapelle2006}
O.~Chapelle, B.~Scholkopf, and A.~Zien.
\newblock {\em Semi-supervised learning}.
\newblock MIT, 2006.

\bibitem{chapman2023batch}
J.~Chapman, B.~Chen, Z.~Tan, J.~Calder, K.~Miller, and A.~L. Bertozzi.
\newblock Novel Batch Active Learning Approach and Its Application on the
  Synthetic Aperture Radar Datasets.
\newblock {\em SPIE Defense and Commercial Sensing: Algorithms for Synthetic
  Aperture Radar Imagery XXX (Best Student Paper)}, 2023.

\bibitem{charro2009mixed}
F.~Charro, J.~Garc{\'\i}a~Azorero, and J.~D. Rossi.
\newblock A mixed problem for the infinity Laplacian via tug-of-war games.
\newblock {\em Calculus of Variations and Partial Differential Equations},
  34(3):307--320, 2009.

\bibitem{chen2020simple}
T.~Chen, S.~Kornblith, M.~Norouzi, and G.~Hinton.
\newblock A simple framework for contrastive learning of visual
  representations.
\newblock In {\em International conference on machine learning}, pages
  1597--1607. PMLR, 2020.

\bibitem{coifman2006diffusion}
R.~R. Coifman and S.~Lafon.
\newblock Diffusion maps.
\newblock {\em Applied and computational harmonic analysis}, 21(1):5--30, 2006.

\bibitem{defferrard2016convolutional}
M.~Defferrard, X.~Bresson, and P.~Vandergheynst.
\newblock Convolutional neural networks on graphs with fast localized spectral
  filtering.
\newblock {\em Advances in neural information processing systems}, 29, 2016.

\bibitem{dong2021equivalence}
H.~Dong, J.~Chen, F.~Feng, X.~He, S.~Bi, Z.~Ding, and P.~Cui.
\newblock On the equivalence of decoupled graph convolution network and label
  propagation.
\newblock In {\em Proceedings of the Web Conference 2021}, pages 3651--3662,
  2021.

\bibitem{donoho2003hessian}
D.~L. Donoho and C.~Grimes.
\newblock Hessian eigenmaps: Locally linear embedding techniques for
  high-dimensional data.
\newblock {\em Proceedings of the National Academy of Sciences},
  100(10):5591--5596, 2003.

\bibitem{dunlop2020large}
M.~M. Dunlop, D.~Slep{\v{c}}ev, A.~M. Stuart, and M.~Thorpe.
\newblock Large data and zero noise limits of graph-based semi-supervised
  learning algorithms.
\newblock {\em Applied and Computational Harmonic Analysis}, 49(2):655--697,
  2020.

\bibitem{el2016asymptotic}
A.~El~Alaoui, X.~Cheng, A.~Ramdas, M.~J. Wainwright, and M.~I. Jordan.
\newblock Asymptotic behavior of $\ell_p$-based {L}aplacian regularization in
  semi-supervised learning.
\newblock In {\em Conference on Learning Theory}, pages 879--906, 2016.

\bibitem{elmoataz2017game}
A.~Elmoataz, X.~Desquesnes, and M.~Toutain.
\newblock On the game $p$-{L}aplacian on weighted graphs with applications in
  image processing and data clustering.
\newblock {\em European Journal of Applied Mathematics}, 28(6):922--948, 2017.

\bibitem{elmoataz2017nonlocal}
A.~Elmoataz, F.~Lozes, and M.~Toutain.
\newblock Nonlocal {PDE}s on graphs: From tug-of-war games to unified
  interpolation on images and point clouds.
\newblock {\em Journal of Mathematical Imaging and Vision}, 57(3):381--401,
  2017.

\bibitem{elmoataz2015p}
A.~Elmoataz, M.~Toutain, and D.~Tenbrinck.
\newblock On the $p$-{L}aplacian and $\infty$-{L}aplacian on graphs with
  applications in image and data processing.
\newblock {\em SIAM Journal on Imaging Sciences}, 8(4):2412--2451, 2015.

\bibitem{ennaji2023tug}
H.~Ennaji, Y.~Qu{\'e}au, and A.~Elmoataz.
\newblock Tug of War games and PDEs on graphs with applications in image and
  high dimensional data processing.
\newblock {\em Scientific Reports}, 13(1):6045, 2023.

\bibitem{enwright2023deep}
J.~Enwright, H.~Hardiman-Mostow, J.~Calder, and A.~L. Bertozzi.
\newblock Deep semi-supervised label propagation for SAR image classification.
\newblock {\em SPIE Defense and Commercial Sensing: Algorithms for Synthetic
  Aperture Radar Imagery XXX}, 2023.

\bibitem{EvansPDE}
L.~Evans.
\newblock {\em Partial Differential Equations (Graduate Studies in Mathematics,
  V. 19) GSM/19}.
\newblock American Mathematical Society, June 1998.

\bibitem{evans1982new}
L.~C. Evans.
\newblock A new proof of local C1, $\alpha$ regularity for solutions of certain
  degenerate elliptic pde.
\newblock {\em Journal of Differential Equations}, 45(3):356--373, 1982.

\bibitem{flores2022algorithms}
M.~Flores, J.~Calder, and G.~Lerman.
\newblock {Analysis and algorithms for Lp-based semi-supervised learning on
  graphs}.
\newblock {\em Applied and Computational Harmonic Analysis}, 60:77--122, 2022.

\bibitem{garcia2014multiclass}
C.~Garcia-Cardona, E.~Merkurjev, A.~L. Bertozzi, A.~Flenner, and A.~G. Percus.
\newblock Multiclass data segmentation using diffuse interface methods on
  graphs.
\newblock {\em IEEE transactions on pattern analysis and machine intelligence},
  36(8):1600--1613, 2014.

\bibitem{garcia2020error}
N.~Garc{\'\i}a~Trillos, M.~Gerlach, M.~Hein, and D.~Slep{\v{c}}ev.
\newblock Error estimates for spectral convergence of the graph {L}aplacian on
  random geometric graphs toward the {Laplace--Beltrami} operator.
\newblock {\em Foundations of Computational Mathematics}, 20(4):827--887, 2020.

\bibitem{garcia2016continuum}
N.~Garc{\'\i}a~Trillos and D.~Slep{\v{c}}ev.
\newblock Continuum limit of total variation on point clouds.
\newblock {\em Archive for rational mechanics and analysis}, 220:193--241,
  2016.

\bibitem{gasteiger2018predict}
J.~Gasteiger, A.~Bojchevski, and S.~G{\"u}nnemann.
\newblock Predict then Propagate: Graph Neural Networks meet Personalized
  PageRank.
\newblock In {\em International Conference on Learning Representations}, 2018.

\bibitem{gleich2015pagerank}
D.~F. Gleich.
\newblock PageRank beyond the Web.
\newblock {\em SIAM Review}, 57(3):321--363, 2015.

\bibitem{goodfellow2016deep}
I.~Goodfellow, Y.~Bengio, and A.~Courville.
\newblock {\em Deep learning}.
\newblock MIT press, 2016.

\bibitem{hafiene2018nonlocal}
Y.~Hafiene, J.~Fadili, and A.~Elmoataz.
\newblock Nonlocal $p$-{L}aplacian Variational problems on graphs.
\newblock {\em arXiv:1810.12817}, 2018.

\bibitem{ham2005semisupervised}
J.~Ham, D.~Lee, and L.~Saul.
\newblock Semisupervised alignment of manifolds.
\newblock In {\em International Workshop on Artificial Intelligence and
  Statistics}, pages 120--127. PMLR, 2005.

\bibitem{han2022time}
J.~Han.
\newblock Time-dependent tug-of-war games and normalized parabolic p-Laplace
  equations.
\newblock {\em Nonlinear Analysis}, 214:112542, 2022.

\bibitem{he2004manifold}
J.~He, M.~Li, H.-J. Zhang, H.~Tong, and C.~Zhang.
\newblock Manifold-ranking based image retrieval.
\newblock In {\em Proceedings of the 12th annual ACM international conference
  on Multimedia}, pages 9--16. ACM, 2004.

\bibitem{he2006generalized}
J.~He, M.~Li, H.-J. Zhang, H.~Tong, and C.~Zhang.
\newblock Generalized manifold-ranking-based image retrieval.
\newblock {\em IEEE Transactions on image processing}, 15(10):3170--3177, 2006.

\bibitem{hein2007graph}
M.~Hein, J.-Y. Audibert, and U.~v. Luxburg.
\newblock Graph {L}aplacians and their convergence on random neighborhood
  graphs.
\newblock {\em Journal of Machine Learning Research}, 8(6), 2007.

\bibitem{hein2005graphs}
M.~Hein, J.-Y. Audibert, and U.~Von~Luxburg.
\newblock From Graphs to Manifolds-Weak and Strong Pointwise Consistency of
  Graph Laplacians.
\newblock In {\em COLT}, volume 3559, pages 470--485. Springer, 2005.

\bibitem{hoffmann2022spectral}
F.~Hoffmann, B.~Hosseini, A.~A. Oberai, and A.~M. Stuart.
\newblock Spectral analysis of weighted Laplacians arising in data clustering.
\newblock {\em Applied and Computational Harmonic Analysis}, 56:189--249, 2022.

\bibitem{hu2013method}
H.~Hu, T.~Laurent, M.~A. Porter, and A.~L. Bertozzi.
\newblock A method based on total variation for network modularity optimization
  using the MBO scheme.
\newblock {\em SIAM Journal on Applied Mathematics}, 73(6):2224--2246, 2013.

\bibitem{huang2020combining}
Q.~Huang, H.~He, A.~Singh, S.-N. Lim, and A.~R. Benson.
\newblock Combining label propagation and simple models out-performs graph
  neural networks.
\newblock {\em arXiv preprint arXiv:2010.13993}, 2020.

\bibitem{jacobs2018auction}
M.~Jacobs, E.~Merkurjev, and S.~Esedoḡlu.
\newblock Auction dynamics: A volume constrained MBO scheme.
\newblock {\em Journal of Computational Physics}, 354:288--310, 2018.

\bibitem{ji_variance_2012}
M.~Ji and J.~Han.
\newblock A Variance Minimization Criterion to Active Learning on Graphs.
\newblock In {\em Artificial Intelligence and Statistics}, pages 556--564, Mar.
  2012.

\bibitem{jung2016semi}
A.~Jung, A.~O. Hero~III, A.~Mara, and S.~Jahromi.
\newblock Semi-supervised learning via sparse label propagation.
\newblock {\em arXiv preprint arXiv:1612.01414}, 2016.

\bibitem{kingma2013auto}
D.~P. Kingma and M.~Welling.
\newblock Auto-encoding variational bayes.
\newblock {\em arXiv preprint arXiv:1312.6114}, 2013.

\bibitem{kipf2016semi}
T.~N. Kipf and M.~Welling.
\newblock Semi-supervised classification with graph convolutional networks.
\newblock {\em arXiv:1609.02907}, 2016.

\bibitem{kohn2006deterministic}
R.~Kohn and S.~Serfaty.
\newblock A deterministic-control-based approach motion by curvature.
\newblock {\em Communications on Pure and Applied Mathematics: A Journal Issued
  by the Courant Institute of Mathematical Sciences}, 59(3):344--407, 2006.

\bibitem{krizhevsky2009learning}
A.~Krizhevsky, G.~Hinton, et~al.
\newblock Learning multiple layers of features from tiny images.
\newblock {\em University of Toronto}, 2009.

\bibitem{kyng2015algorithms}
R.~Kyng, A.~Rao, S.~Sachdeva, and D.~A. Spielman.
\newblock Algorithms for {L}ipschitz learning on graphs.
\newblock In {\em Conference on Learning Theory}, pages 1190--1223, 2015.

\bibitem{lecun1998gradient}
Y.~LeCun, L.~Bottou, Y.~Bengio, and P.~Haffner.
\newblock Gradient-based learning applied to document recognition.
\newblock {\em Proceedings of the IEEE}, 86(11):2278--2324, 1998.

\bibitem{lee2013graph}
W.-Y. Lee, L.-C. Hsieh, G.-L. Wu, and W.~Hsu.
\newblock Graph-based semi-supervised learning with multi-modality propagation
  for large-scale image datasets.
\newblock {\em Journal of visual communication and image representation},
  24(3):295--302, 2013.

\bibitem{lewicka2018random}
M.~Lewicka.
\newblock Random Tug of War games for the p-Laplacian: $1 < p < \infty$.
\newblock {\em arXiv preprint arXiv:1810.03413}, 2018.

\bibitem{lewicka2020course}
M.~Lewicka.
\newblock {\em A course on Tug-of-War games with random noise}.
\newblock Springer, 2020.

\bibitem{lewicka2022non}
M.~Lewicka.
\newblock Non-local Tug-of-War with noise for the geometric fractional
  p-Laplacian.
\newblock {\em Adv. Differential Equations}, 27(1-2):31--76, 2022.

\bibitem{lewicka2014game}
M.~Lewicka and J.~J. Manfredi.
\newblock Game theoretical methods in {PDE}s.
\newblock {\em Bollettino dell'Unione Matematica Italiana}, 7(3):211--216,
  2014.

\bibitem{lewicka2017obstacle}
M.~Lewicka and J.~J. Manfredi.
\newblock The obstacle problem for the p-laplacian via optimal stopping of
  tug-of-war games.
\newblock {\em Probability Theory and Related Fields}, 167:349--378, 2017.

\bibitem{lewicka2022robin}
M.~Lewicka and Y.~Peres.
\newblock The Robin mean value equation I: A random walk approach to the third
  boundary value problem.
\newblock {\em Potential Analysis}, pages 1--32, 2022.

\bibitem{lewicka2022robinII}
M.~Lewicka and Y.~Peres.
\newblock The Robin mean value equation II: asymptotic H{\"o}lder regularity.
\newblock {\em Potential Analysis}, pages 1--35, 2022.

\bibitem{li2022finding}
X.~Li, R.~Zhu, Y.~Cheng, C.~Shan, S.~Luo, D.~Li, and W.~Qian.
\newblock Finding global homophily in graph neural networks when meeting
  heterophily.
\newblock In {\em International Conference on Machine Learning}, pages
  13242--13256. PMLR, 2022.

\bibitem{lindqvist2019notes}
P.~Lindqvist.
\newblock {\em Notes on the stationary p-Laplace equation}.
\newblock Springer, 2019.

\bibitem{luiro2013harnack}
H.~Luiro, M.~Parviainen, and E.~Saksman.
\newblock Harnack's inequality for p-harmonic functions via stochastic games.
\newblock {\em Communications in Partial Differential Equations},
  38(11):1985--2003, 2013.

\bibitem{luxburg2004distance}
U.~v. Luxburg and O.~Bousquet.
\newblock Distance-based classification with {L}ipschitz functions.
\newblock {\em Journal of Machine Learning Research}, 5(Jun):669--695, 2004.

\bibitem{ma_sigma_2013}
Y.~Ma, R.~Garnett, and J.~Schneider.
\newblock {$\Sigma$}-Optimality for Active Learning on {Gaussian} Random
  Fields.
\newblock In C.~J.~C. Burges, L.~Bottou, M.~Welling, Z.~Ghahramani, and K.~Q.
  Weinberger, editors, {\em Advances in {Neural} {Information} {Processing}
  {Systems} 26}, pages 2751--2759. Curran Associates, Inc., 2013.

\bibitem{mai2018random}
X.~Mai.
\newblock A random matrix analysis and improvement of semi-supervised learning
  for large dimensional data.
\newblock {\em Journal of Machine Learning Research}, 19(79):1--27, 2018.

\bibitem{maicouillet2018random}
X.~Mai and R.~Couillet.
\newblock Random matrix-inspired improved semi-supervised learning on graphs.
\newblock In {\em International Conference on Machine Learning}, 2018.

\bibitem{mai2021consistent}
X.~Mai and R.~Couillet.
\newblock Consistent semi-supervised graph regularization for high dimensional
  data.
\newblock {\em The Journal of Machine Learning Research}, 22(1):4181--4228,
  2021.

\bibitem{manfredi2010asymptotic}
J.~Manfredi, M.~Parviainen, and J.~Rossi.
\newblock An asymptotic mean value characterization for p-harmonic functions.
\newblock {\em Proceedings of the American Mathematical Society},
  138(3):881--889, 2010.

\bibitem{manfredi2015nonlinear}
J.~J. Manfredi, A.~M. Oberman, and A.~P. Sviridov.
\newblock Nonlinear elliptic partial differential equations and p-harmonic
  functions on graphs.
\newblock {\em Differential Integral Equations}, 28(1-2):79--102, 2015.

\bibitem{manfredi2012dynamic}
J.~J. Manfredi, M.~Parviainen, and J.~D. Rossi.
\newblock Dynamic programming principle for tug-of-war games with noise.
\newblock {\em ESAIM: Control, Optimisation and Calculus of Variations},
  18(1):81--90, 2012.

\bibitem{manfredi2012definition}
J.~J. Manfredi, M.~Parviainen, and J.~D. Rossi.
\newblock On the definition and properties of $ p $-harmonious functions.
\newblock {\em Annali della Scuola Normale Superiore di Pisa-Classe di
  Scienze}, 11(2):215--241, 2012.

\bibitem{merkurjev2018semi}
E.~Merkurjev, A.~L. Bertozzi, and F.~Chung.
\newblock A semi-supervised heat kernel pagerank MBO algorithm for data
  classification.
\newblock {\em Communications in Mathematical Sciences}, 16(5):1241--1265,
  2018.

\bibitem{merkurjev2014diffuse}
E.~Merkurjev, C.~Garcia-Cardona, A.~L. Bertozzi, A.~Flenner, and A.~G. Percus.
\newblock Diffuse interface methods for multiclass segmentation of
  high-dimensional data.
\newblock {\em Applied Mathematics Letters}, 33:29--34, 2014.

\bibitem{merkurjev2013mbo}
E.~Merkurjev, T.~Kostic, and A.~L. Bertozzi.
\newblock An MBO scheme on graphs for classification and image processing.
\newblock {\em SIAM Journal on Imaging Sciences}, 6(4):1903--1930, 2013.

\bibitem{merriman1994motion}
B.~Merriman, J.~K. Bence, and S.~J. Osher.
\newblock Motion of multiple junctions: A level set approach.
\newblock {\em Journal of computational physics}, 112(2):334--363, 1994.

\bibitem{miller2022graphbased}
K.~Miller, X.~Baca, J.~Mauro, J.~Setiadi, Z.~Shi, J.~Calder, and A.~Bertozzi.
\newblock Graph-based active learning for semi-supervised classification of
  {SAR} data.
\newblock {\em SPIE Defense and Commercial Sensing: Algorithms for Synthetic
  Aperture Radar Imagery XXIX}, 12095, 2022.

\bibitem{miller2021model}
K.~Miller and A.~L. Bertozzi.
\newblock Model-change active learning in graph-based semi-supervised learning.
\newblock {\em arXiv preprint arXiv:2110.07739}, 2021.

\bibitem{miller2023active}
K.~Miller and J.~Calder.
\newblock Poisson Reweighted {L}aplacian Uncertainty Sampling for Graph-based
  Active Learning.
\newblock {\em To appear in SIAM Journal on Mathematics of Data Science}, 2023.

\bibitem{murphy_unsupervised_2019}
J.~M. Murphy and M.~Maggioni.
\newblock Unsupervised Clustering and Active Learning of Hyperspectral Images
  With Nonlinear Diffusion.
\newblock {\em IEEE Transactions on Geoscience and Remote Sensing},
  57(3):1829--1845, Mar. 2019.

\bibitem{nadler2009semi}
B.~Nadler, N.~Srebro, and X.~Zhou.
\newblock Semi-supervised learning with the graph {L}aplacian: The limit of
  infinite unlabelled data.
\newblock {\em Advances in neural information processing systems},
  22:1330--1338, 2009.

\bibitem{oberman2005convergent}
A.~Oberman.
\newblock A convergent difference scheme for the infinity Laplacian:
  construction of absolutely minimizing Lipschitz extensions.
\newblock {\em Mathematics of computation}, 74(251):1217--1230, 2005.

\bibitem{oberman2013finite}
A.~M. Oberman.
\newblock Finite difference methods for the infinity Laplace and p-Laplace
  equations.
\newblock {\em Journal of Computational and Applied Mathematics}, 254:65--80,
  2013.

\bibitem{parviainen2024notes}
M.~Parviainen.
\newblock {\em Notes on tug-of-war games and the p-Laplace equation}.
\newblock SpringerBriefs on PDEs and Data Science, 2024.

\bibitem{peres2010biased}
Y.~Peres, G.~Pete, and S.~Somersille.
\newblock Biased tug-of-war, the biased infinity Laplacian, and comparison with
  exponential cones.
\newblock {\em Calculus of Variations and Partial Differential Equations},
  38(3-4):541--564, 2010.

\bibitem{peres2009tug}
Y.~Peres, O.~Schramm, S.~Sheffield, and D.~Wilson.
\newblock Tug-of-war and the infinity Laplacian.
\newblock {\em Journal of the American Mathematical Society}, 22(1):167--210,
  2009.

\bibitem{peres2008tug}
Y.~Peres, S.~Sheffield, et~al.
\newblock Tug-of-war with noise: A game-theoretic view of the $ p
  $-{L}aplacian.
\newblock {\em Duke Mathematical Journal}, 145(1):91--120, 2008.

\bibitem{qiao_uncertainty_2019}
Y.-L. Qiao, C.~X. Shi, C.~Wang, H.~Li, M.~Haberland, X.~Luo, A.~M. Stuart, and
  A.~L. Bertozzi.
\newblock Uncertainty quantification for semi-supervised multi-class
  classification in image processing and ego-motion analysis of body-worn
  videos.
\newblock {\em Image Processing: Algorithms and Systems}, 2019.

\bibitem{sellars2022laplacenet}
P.~Sellars, A.~I. Aviles-Rivero, and C.-B. Sch{\"o}nlieb.
\newblock Laplacenet: A hybrid graph-energy neural network for deep
  semisupervised classification.
\newblock {\em IEEE Transactions on Neural Networks and Learning Systems},
  2022.

\bibitem{settles2009active}
B.~Settles.
\newblock Active learning literature survey.
\newblock {\em University of Wisconsin-Madison Department of Computer
  Sciences}, 2009.

\bibitem{shi2017weighted}
Z.~Shi, S.~Osher, and W.~Zhu.
\newblock Weighted nonlocal {L}aplacian on interpolation from sparse data.
\newblock {\em Journal of Scientific Computing}, 73(2-3):1164--1177, 2017.

\bibitem{slepcev2019analysis}
D.~Slepcev and M.~Thorpe.
\newblock Analysis of p-laplacian regularization in semisupervised learning.
\newblock {\em SIAM Journal on Mathematical Analysis}, 51(3):2085--2120, 2019.

\bibitem{trillos2021geometric}
N.~G. Trillos, F.~Hoffmann, and B.~Hosseini.
\newblock Geometric structure of graph Laplacian embeddings.
\newblock {\em The Journal of Machine Learning Research}, 22(1):2934--2988,
  2021.

\bibitem{vaswani2017attention}
A.~Vaswani, N.~Shazeer, N.~Parmar, J.~Uszkoreit, L.~Jones, A.~N. Gomez,
  L.~Kaiser, and I.~Polosukhin.
\newblock Attention is all you need.
\newblock {\em Advances in neural information processing systems}, 30, 2017.

\bibitem{velivckovic2017graph}
P.~Veli{\v{c}}kovi{\'c}, G.~Cucurull, A.~Casanova, A.~Romero, P.~Lio, and
  Y.~Bengio.
\newblock Graph attention networks.
\newblock {\em arXiv preprint arXiv:1710.10903}, 2017.

\bibitem{von_luxburg_tutorial_2007}
U.~von Luxburg.
\newblock A tutorial on spectral clustering.
\newblock {\em Statistics and Computing}, 17(4):395--416, Dec. 2007.

\bibitem{wang2013dynamic}
B.~Wang, Z.~Tu, and J.~K. Tsotsos.
\newblock Dynamic label propagation for semi-supervised multi-class multi-label
  classification.
\newblock In {\em Proceedings of the IEEE international conference on computer
  vision}, pages 425--432, 2013.

\bibitem{wang2013multi}
Y.~Wang, M.~A. Cheema, X.~Lin, and Q.~Zhang.
\newblock Multi-manifold ranking: Using multiple features for better image
  retrieval.
\newblock In {\em Pacific-Asia Conference on Knowledge Discovery and Data
  Mining}, pages 449--460. Springer, 2013.

\bibitem{weihs2023consistency}
A.~Weihs and M.~Thorpe.
\newblock Consistency of Fractional Graph-Laplacian Regularization in
  Semi-Supervised Learning with Finite Labels.
\newblock {\em arXiv preprint arXiv:2303.07818}, 2023.

\bibitem{williams1991probability}
D.~Williams.
\newblock {\em Probability with martingales}.
\newblock Cambridge university press, 1991.

\bibitem{xu2011efficient}
B.~Xu, J.~Bu, C.~Chen, D.~Cai, X.~He, W.~Liu, and J.~Luo.
\newblock Efficient manifold ranking for image retrieval.
\newblock In {\em Proceedings of the 34th international ACM SIGIR conference on
  Research and development in Information Retrieval}, pages 525--534. ACM,
  2011.

\bibitem{yan2022two}
Y.~Yan, M.~Hashemi, K.~Swersky, Y.~Yang, and D.~Koutra.
\newblock Two sides of the same coin: Heterophily and oversmoothing in graph
  convolutional neural networks.
\newblock In {\em 2022 IEEE International Conference on Data Mining (ICDM)},
  pages 1287--1292. IEEE, 2022.

\bibitem{yang2013saliency}
C.~Yang, L.~Zhang, H.~Lu, X.~Ruan, and M.-H. Yang.
\newblock Saliency detection via graph-based manifold ranking.
\newblock In {\em Proceedings of the IEEE conference on computer vision and
  pattern recognition}, pages 3166--3173, 2013.

\bibitem{yang2022survey}
X.~Yang, Z.~Song, I.~King, and Z.~Xu.
\newblock A survey on deep semi-supervised learning.
\newblock {\em IEEE Transactions on Knowledge and Data Engineering}, 2022.

\bibitem{yang2016revisiting}
Z.~Yang, W.~Cohen, and R.~Salakhudinov.
\newblock Revisiting semi-supervised learning with graph embeddings.
\newblock In {\em International conference on machine learning}, pages 40--48.
  PMLR, 2016.

\bibitem{yuan2022continuum}
A.~Yuan, J.~Calder, and B.~Osting.
\newblock {A continuum limit for the {PageRank} algorithm}.
\newblock {\em European Journal of Applied Mathematics}, 33:472--504, 2022.

\bibitem{zheng2022graph}
X.~Zheng, Y.~Liu, S.~Pan, M.~Zhang, D.~Jin, and P.~S. Yu.
\newblock Graph neural networks for graphs with heterophily: A survey.
\newblock {\em arXiv preprint arXiv:2202.07082}, 2022.

\bibitem{zhou2004learning}
D.~Zhou, O.~Bousquet, T.~N. Lal, J.~Weston, and B.~Sch{\"o}lkopf.
\newblock Learning with local and global consistency.
\newblock In {\em Advances in neural information processing systems}, pages
  321--328, 2004.

\bibitem{zhou2005learning}
D.~Zhou, J.~Huang, and B.~Sch{\"o}lkopf.
\newblock Learning from labeled and unlabeled data on a directed graph.
\newblock In {\em Proceedings of the 22nd international conference on Machine
  learning}, pages 1036--1043. ACM, 2005.

\bibitem{zhou2004random}
D.~Zhou and B.~Sch{\"o}lkopf.
\newblock Learning from labeled and unlabeled data using random walks.
\newblock In {\em Joint Pattern Recognition Symposium}, pages 237--244.
  Springer, 2004.

\bibitem{zhou2005regularization}
D.~Zhou and B.~Sch{\"o}lkopf.
\newblock Regularization on discrete spaces.
\newblock In {\em Joint Pattern Recognition Symposium}, pages 361--368.
  Springer, 2005.

\bibitem{zhou2004ranking}
D.~Zhou, J.~Weston, A.~Gretton, O.~Bousquet, and B.~Sch{\"o}lkopf.
\newblock Ranking on data manifolds.
\newblock In {\em Advances in neural information processing systems}, pages
  169--176, 2004.

\bibitem{zhou2011semi}
X.~Zhou and M.~Belkin.
\newblock Semi-supervised learning by higher order regularization.
\newblock In {\em Proceedings of the fourteenth international conference on
  artificial intelligence and statistics}, pages 892--900. JMLR Workshop and
  Conference Proceedings, 2011.

\bibitem{zhu2020beyond}
J.~Zhu, Y.~Yan, L.~Zhao, M.~Heimann, L.~Akoglu, and D.~Koutra.
\newblock Beyond homophily in graph neural networks: Current limitations and
  effective designs.
\newblock {\em Advances in neural information processing systems},
  33:7793--7804, 2020.

\bibitem{zhu2002learning}
X.~Zhu and Z.~Ghahramani.
\newblock Learning from labeled and unlabeled data with label propagation.
\newblock {\em ProQuest number: information to all users}, 2002.

\bibitem{zhu2003semi}
X.~Zhu, Z.~Ghahramani, and J.~D. Lafferty.
\newblock Semi-supervised learning using Gaussian fields and harmonic
  functions.
\newblock In {\em Proceedings of the 20th International conference on Machine
  learning (ICML-03)}, pages 912--919, 2003.

\bibitem{zhu_combining_2003}
X.~Zhu, J.~Lafferty, and Z.~Ghahramani.
\newblock Combining Active Learning and Semi-Supervised Learning Using
  {Gaussian} Fields and Harmonic Functions.
\newblock In {\em International Conference on Machine Learning (ICML) 2003
  workshop on The Continuum from Labeled to Unlabeled Data in Machine Learning
  and Data Mining}, pages 58--65, 2003.

\end{thebibliography}
